\documentclass[11 pt]{amsart}
\usepackage{latexsym,amsmath,amsfonts,graphicx}
\usepackage{graphicx}
\usepackage{subfigure}
\usepackage{amssymb,cases}
\usepackage{amstext}
\usepackage[dvips]{color}
\newtheorem{theorem}{Theorem}[section]
\newtheorem{thm}[theorem]{Theorem}
\newtheorem{pro}{Proposition}[section]

\newtheorem{lemma}[pro]{Lemma}
\newtheorem{remark}[pro]{Remark}
\newtheorem{defi}{Definition}[section]

\newtheorem{example}{Example}[section]
\numberwithin{equation}{section}

\def\cal{\mathcal }
\def\R{\mathbb R}

\def\mathscr{\mathcal }

\newcommand{\bbeta}{{\boldsymbol{\beta}}}

\newcommand{\bomega}{{\boldsymbol {\omega}}}
\newcommand{\bgamma}{{\boldsymbol {\gamma}}}

\newcommand{\mi}{{\mathbf{i}}}
\newcommand{\cS}{{\mathcal S}}

\begin{document}

\title[Space-filling curves]
{Space-filling curves  of  self-similar  sets (II):
Edge-to-trail substitution rule}

\author{Xin-Rong Dai}
\address{Xin-Rong Dai: School of Mathematics, Sun Yat-sen University, CHINA}
\email{daixr@mail.sysu.edu.cn}
\author{Hui Rao}
\address{Hui Rao: Department of Mathematics and Statistics, Central China Normal University, CHINA}
\email{hrao@mail.ccnu.edu.cn}
\author{Shu-Qin Zhang $\dag$}
\address{Shu-Qin Zhang: Chair of Mathematics and Statistics, University of Leoben, AUSTRIA }
\email{zhangsq\_ccnu@sina.com}

\thanks{$\dag$ The correspondence author.}
\thanks{The work is supported by CNFS Nos 11431007,~11171128 and 11471075,
 and the doctoral program W1230 granted by the Austrian Science Fund (FWF)}
%\textbf{ABSTRACT:}
 \maketitle
\begin{abstract}
  It is well-known that the constructions of space-filling curves depend on certain substitution rules.
For a given self-similar set, finding such  rules is somehow mysterious,
and it is the  main concern of the present paper.

 Our first idea is to introduce the notion of skeleton for a self-similar set. Then, from a  skeleton, we construct several graphs, define edge-to-trail substitution rules, and
  explore conditions  ensuring the  rules  lead to space-filling curves.
  Thirdly, we summarize the classical constructions of the space-filling curves into two classes: the traveling-trail class and the positive Euler-tour class. Finally,
  we propose a general Euler-tour method, using which we
       show that if a  self-similar set satisfies the open set condition and possesses a skeleton,
      then space-filling curves can be constructed.
   Especially, all connected self-similar sets of finite type fall into this class.
    Our study actually  provides an algorithm to construct space-filling curves of self-similar sets.
%Almost all the known examples of space-filling curves can be explained by our theory.
 \medskip

\end{abstract}

\section{\textbf{Introduction}}\label{sec:intro}

Space-filling curves (SFC) have fascinated mathematicians for over a century since Peano's monumental work in 1890.
In a series of three papers (\cite{RaoZh15},  the present paper,
and \cite{RaoZh17}), we develop a theory to construct SFCs of self-similar sets.
 For a given connected self-similar set, finding substitution rules  leading to SFCs is a long-standing and difficult problem, and
 it is the main concern  of the present paper.

\subsection{A brief history of space-filling curves.}
SFCs of the first generation
 were constructed by  Peano (1890),
 Hilbert (1891),  Sierpi\'nski (1912) and P\'olya (1913)   (\cite{Peano1890, Hilbert1891, Sierp1912, Polya1913}), where the base sets are squares or triangles.
 %see Figure \ref{PHS}.
 Later on, people found many beautiful reptiles as well as their space-filling curves, where the boundaries of the reptiles are fractals; for instance,
  Heighway dragon curve (\cite{Davis70}),
  Gosper curve (\cite{Gardner76}), \textit{etc.}.
  A survey of the early results can be found in Sagan \cite{Hans94}. In recent years, various interesting  SFCs
  appear  on the internet,
see for example, ``www.fractalcurves.com'' (\cite{Ventrella})
     and  ``teachout1.net/village/'' (\cite{Gary}).
     Figure \ref{Wedge-1} illustrates two of them.

% \begin{figure}[h]
% \includegraphics[width=.3 \textwidth]{peano_003}
% \includegraphics[width=.3 \textwidth]{Hilbert003.eps}
%  \includegraphics[width=.3 \textwidth]{Hilbert_CAG_001}
%  \caption{The constructions of Peano, Hilbert and Sierpi\'nski. An excellent survey of the Sierpi\'nski curve can be found in \cite{Ciesie2012}.
%  }
%  \label{PHS}
%\end{figure}

 From the 1960's to the 1980's, two systematic methods were introduced to handle the SFCs. The first one is
the \emph{$L$-system method} introduced by Lindenmayer (\cite{Lind68}),
  a biologist;
this method is known to a very wide audience,
see  Bader \cite{Bader13}.
 The second method is
 the \emph{recurrent set method}  introduced by  Dekking \cite{Dekking82},
  which is   an improvement of the $L$-system method. See \cite{RaoZh15}
  for the comparing of the two methods.

\begin{figure}[h]
  \includegraphics[width=0.35 \textwidth]{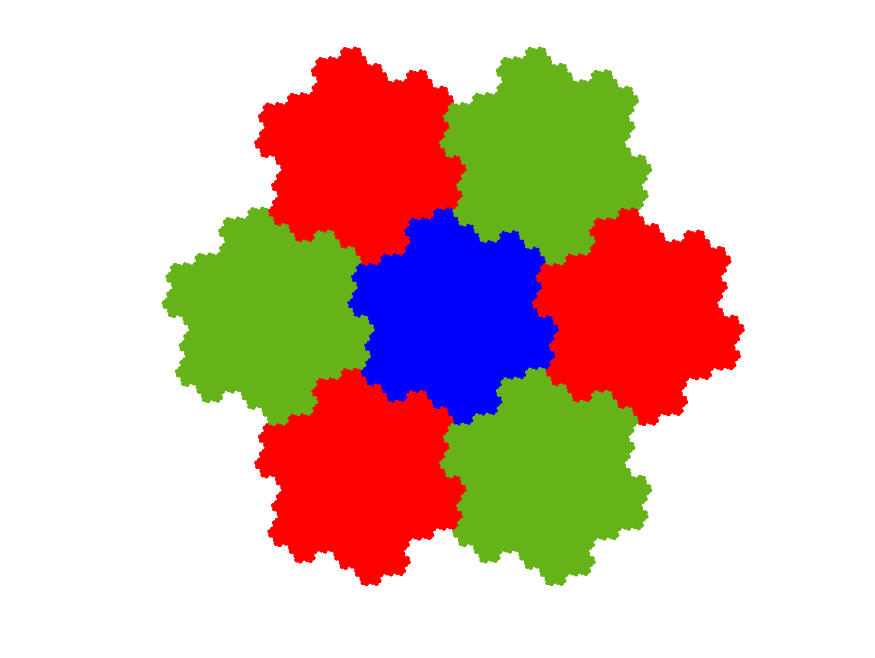}
  \includegraphics[width=0.35 \textwidth]{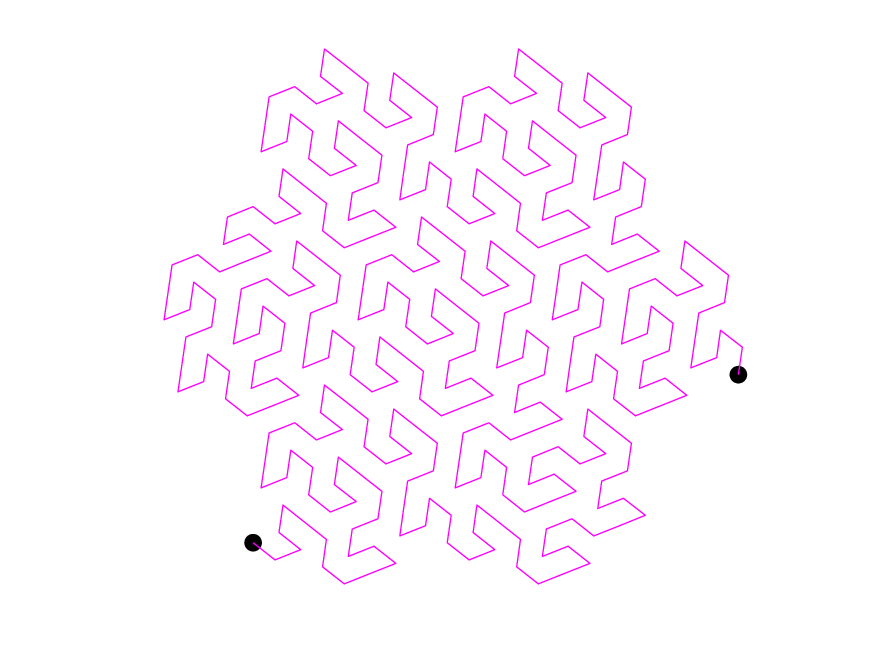} \hspace{-1.cm}
  \caption{Gosper island and Gosper curve.}
  \label{fig:Gosper}
\end{figure}

\begin{figure}[h]
  \subfigure[\text{Wedge tile } ]{\includegraphics[width=0.42\textwidth]{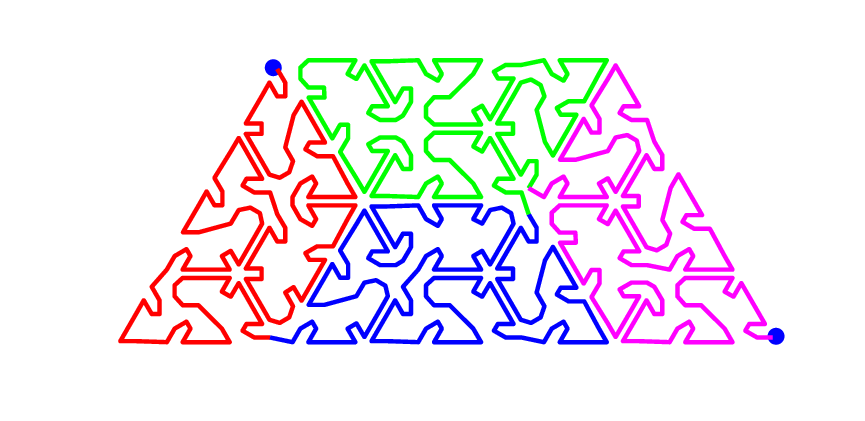}}\quad\quad
  \subfigure[\text{Four-tile star}]{\includegraphics[width=0.37 \textwidth]{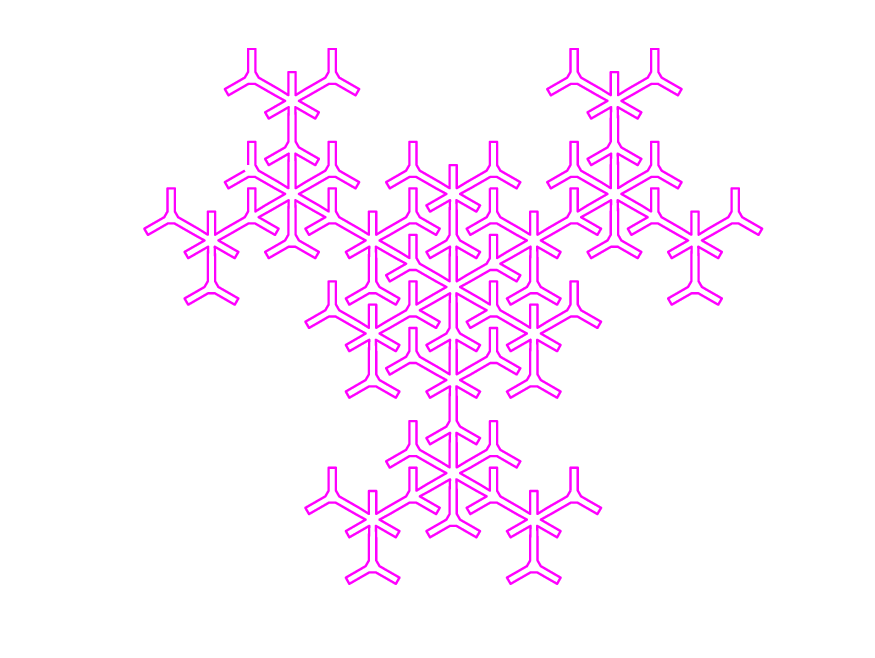}}\\
  \caption{ Two SFCs  taking from \cite{Gary}: (a) is discussed in Example \ref{ex-Wedge}, and
   a detailed discussion
   of (b) is given in Section 7 of \cite{RaoZh15}.
}
  \label{Wedge-1}
\end{figure}

All the known constructions of SFC depend on certain `substitution rules'. Indeed,
the $L$-system method and the recurrent set method provide exact meaning of `substitution rule' and build a bridge from substitution rules to SFCs, but they do not tell us how to construct substitution rules.

There are few works on the construction of substitution rules.
Hata \cite{Hata85} shows that if a self-similar set
is generated by a \emph{linear IFS} (see Section \ref{sec:travel} for the definition),  then a substitution rule can be obtained.
% but very few self-similar sets belong to this case.
Another way is to consider the attractor of the    so-called  \emph{path-on-lattice} IFS (this  name is given by \cite{RaoZh15}).
 A path on a planar lattice   defines a substitution rule as well as a self-similar set; if the self-similar set happens to satisfy the open set condition, then the substitution rule leads to a SFC.
 This method is widely used to
  find   reptiles and  space-filling curves  by computer searching, see Fukuda \textit{et al.} \cite{Fukuda}, Arndt \cite{Arndt16} and   ``www.fractalcurves.com'' \cite{Ventrella}.
   But in this method,  the self-similar sets are not priori given.
  Other attempts of constructing substitution rules for special self-similar sets can be found in Remes \cite{Remes98} and Sirvent \cite{Sirvent2003, Sirvent2008,Sirvent2012}.

Before stating our results, we note that   as we did in \cite{RaoZh15},   the terminology `space-filling curve' is used in a  strong sense, that is,  it is a  kind of \emph{optimal  parametrization}. Let ${\mathcal H}^s$ denote the $s$-dimensional Hausdorff measure.

\begin{defi}{\rm
Let $K$ be a compact subset of ${\mathbb R}^d$ with $0<{\mathcal H}^s(K)<\infty$.
 An  onto mapping   $\psi:[0,1]\rightarrow K$ is called an
\emph{optimal parametrization} of $K$ if:

\ \ $(i)$  $\psi$ is \emph{almost one-to-one}, that is, there exist $K'\subset K$ and $I'\subset [0,1]$ with full measures such that
 $\psi:~I'\to K'$ is a bijection;

  \ $(ii)$ $\psi$ is \emph{measure-preserving} in the sense that
  $$
  {\mathcal H}^s(\psi(F))=c{\mathcal L}(F) \text{ and } {\mathcal L}(\psi^{-1}(B))=c^{-1}{\mathcal H}^s(B),
  $$
  for any Borel set $F\subset [0,1]$ and any Borel set $B\subset K$, where $c={\mathcal H}^s(K)$.

  $(iii)$ $\psi$ is $1/s$-\emph{H\"older continuous}, that is, there is a constant $c'>0$ such that
  $$
  |\psi(x)-\psi(y)|\leq c'|x-y|^{\frac{1}{s}} \  \text{ for all } x,y\in [0,1].
  $$
  }
  \end{defi}

 Recall that   a non-empty compact set   $K\subset \R^d$  is called a \emph{self-similar set}, if  there exist contractive similitudes  $S_1,\dots, S_N: \R^d\to \R^d$ such that
$$
K=\bigcup_{j=1}^N S_j(K).
$$
The family $\{S_1,\dots, S_N\}$ is called an \emph{iterated function system}, or \emph{IFS} in short;  $K$ is called the \emph{invariant set} of the IFS. (See for instance, \cite{Hut81, Fal90}).
The IFS $\{S_1,\dots, S_N\}$ is said to satisfy the \emph{open set condition (OSC)}, if there is an open set $U$ such that
$\bigcup_{i=1}^N S_i(U)\subset U$ and the sets $S_i(U)$ are disjoint (\cite{Hut81}).
%If the contractions $S_j$ are all similitudes, then we call $K$ a \emph{self-similar set}.
If a self-similar set $K$ satisfies the OSC  and has non-empty interior,
 %This is proved by A. Shief (\cite{Shief94}) and
 %it is highly non-trivial.
 then  we call   $K$  a \emph{self-similar tile};
  if in addition, the contraction ratios of $S_i$ are all equal , then $K$ is called a
\emph{reptile} (\cite{Kenyon92}).

In this paper, we introduce a new  notion, called   \emph{skeleton} of a self-similar set, which is crucial in our theory.  As soon as we have a skeleton,
 we  construct a space-filling curve along the following line:

 \centerline{\textit{
Skeleton $\rightarrow$ graphs and edge-to-trail substitution  $\rightarrow$ space-filling curve  $\rightarrow$ visualization.}}

To `see' a SFC, we need to visualize or to approximate the curve, and this has  been studied in \cite{RaoZh15}.
In the following, we give a brief description of the first three  steps.

%\begin{remark}\textbf{Visualizations.} {\rm
%To `see' a SFC, we need to visualize or to approximate the curve.
%Thanks to the concept of linear GIFS,  in \cite{RaoZh15}, it is shown that as soon as a certain `initial pattern' is given, we obtain a visualization;
% hence various beautiful visualizations can be obtained by very simple procedures.
% (For reader's sake, we provide initial patterns for all examples, so the readers can easily reproduce them.)
%Usually, we prefer to have a  self-avoiding visualization, but
%  when and  how to obtain such visualizations is a difficult problem, see Dekking \cite{Dekking2012}.
%}
%\end{remark}

\subsection{Skeleton of a self-similar set}
%The first important idea is that we introduce a   notion of  \emph{skeleton} of a self-similar set.
Let $\{S_j\}_{j=1}^{N}$ be an IFS with invariant set   $K$. Let $A$ be a finite subset of $K$.
We call $A$ a \emph{skeleton} if it satisfies the following two conditions:

$(i)$ It is stable under iteration, that is,  $A \subset \bigcup_{j=1}^N S_j(A)$;

$(ii)$ It is a representative with respect to the connectedness, that is, the  so-called  Hata graph associated with $A$ is connected. (See Section \ref{sec:skeleton} for the precise definition.)

\begin{figure}[h]
  \centering
  \includegraphics[width=.27 \textwidth]{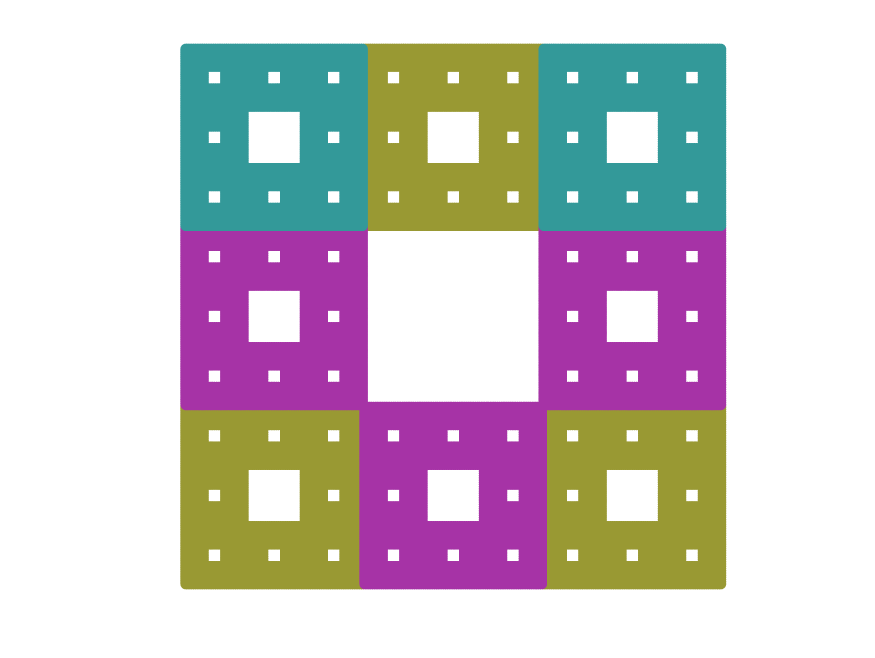}
  \includegraphics[width=.27 \textwidth]{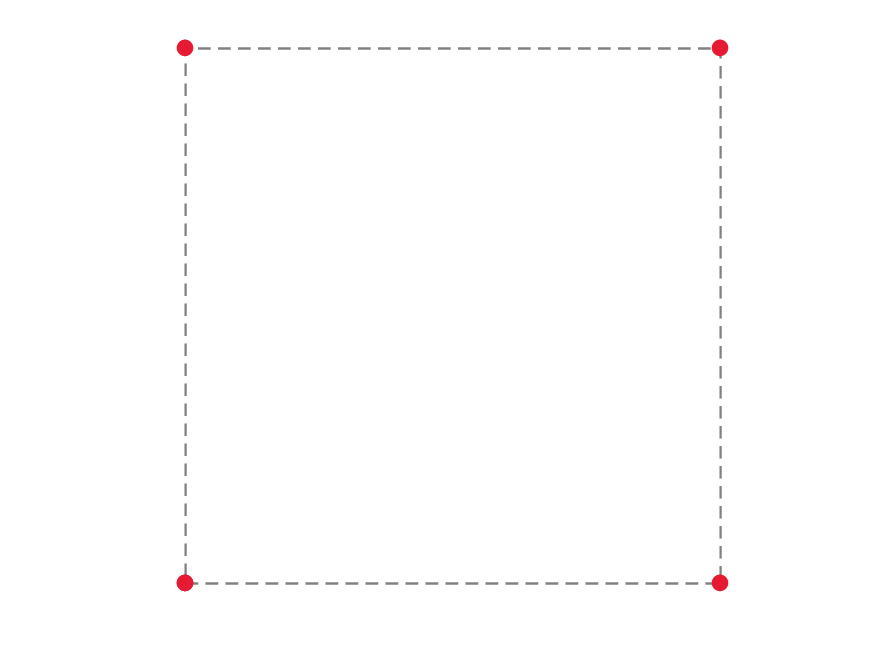}
  \includegraphics[width=.27 \textwidth]{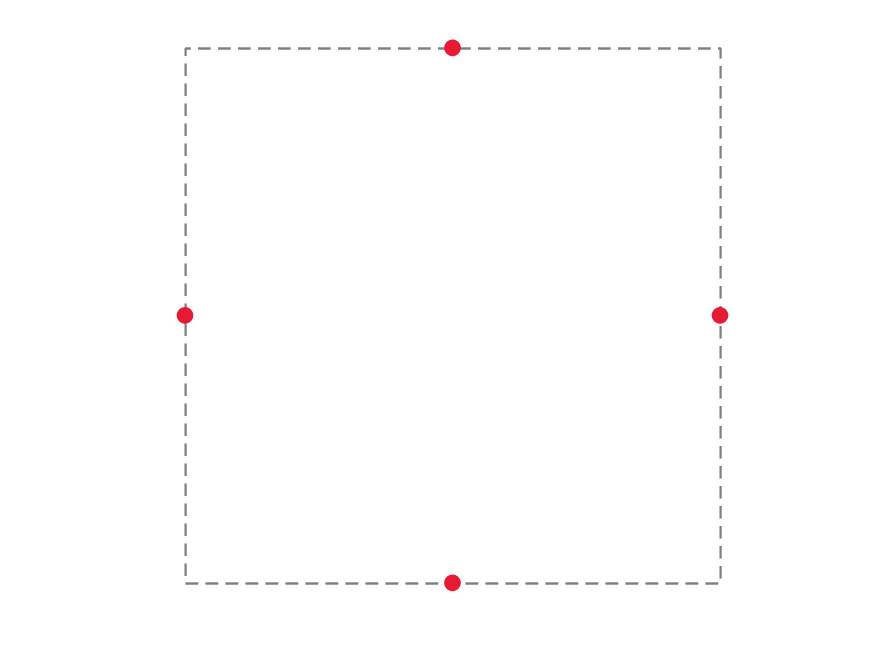}
  \caption{The  Sierpi\'nski carpet and  two skeletons of it. }\label{fig-carpet}
\end{figure}

 From now on, we always assume that $K$ is a connected self-similar set possessing a skeleton which  we  denote by
 \begin{equation}\label{eq:ske}
 A=\{a_1,\dots, a_m\}.
 \end{equation}

\subsection{Graphs and edge-to-trail substitution} First, we recall some terminologies of graph theory, see for instance, \cite{BalaRang2000}.
Let $H$ be a directed graph. We shall use $e_1+\dots+e_k$ to denote a walk consisting of
the edges $e_1,\dots, e_k$. We call the starting vertex and terminate vertex of a walk the \emph{origin} and \emph{terminus}, respectively.
The walk is \emph{closed} if the origin of $e_1$ and the terminus of $e_k$ coincide.

 A walk is called a \emph{trail}, if all the edges appearing in the walk are distinct.
A trail  is called a \emph{path} if all the vertices are distinct.
 A  closed path is called a \emph{cycle}.

 A subgraph $H'$ of $H$ is called \emph{spanning}, if $H'$ contains all the vertices of $H$.
 An \emph{Euler trail} in $H$ is a spanning trail in $H$ that contains all the edges of $H$.  An \emph{Euler tour} of $H$ is a closed Euler trail of $H$.

 \begin{figure}[h]
  \centering
  \subfigure[\text{Terdragon}]{\includegraphics[width=0.33 \textwidth]{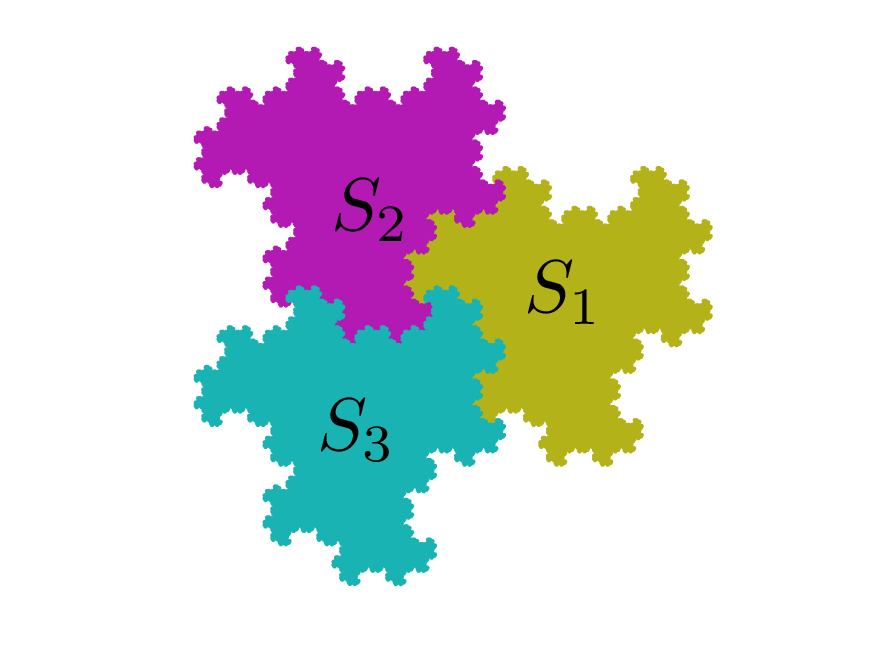}}
  \subfigure[\text{Initial graph}]{\includegraphics[width=0.3 \textwidth]{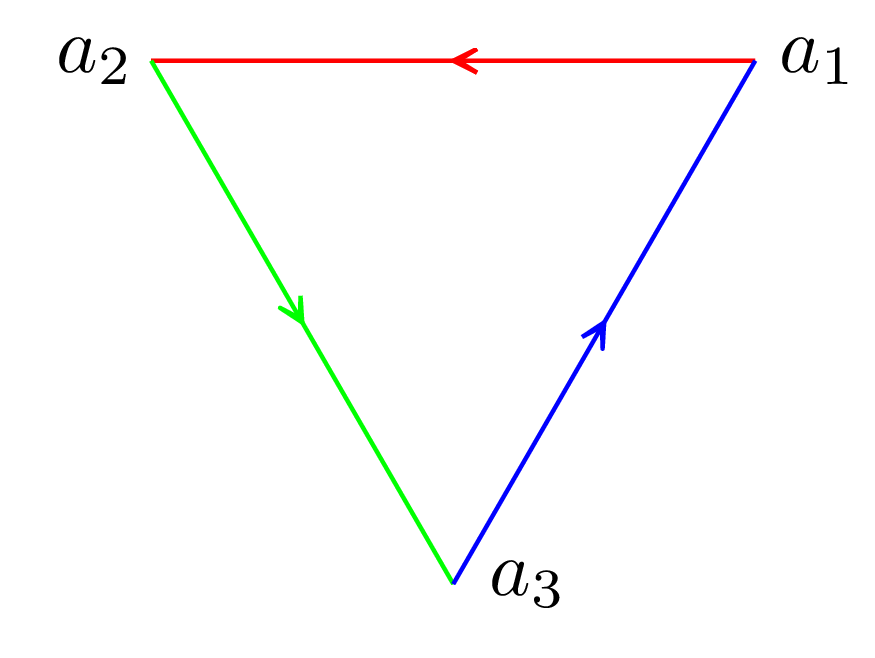}}
  \subfigure[\text{Refined graph}]{\includegraphics[width=0.33 \textwidth]{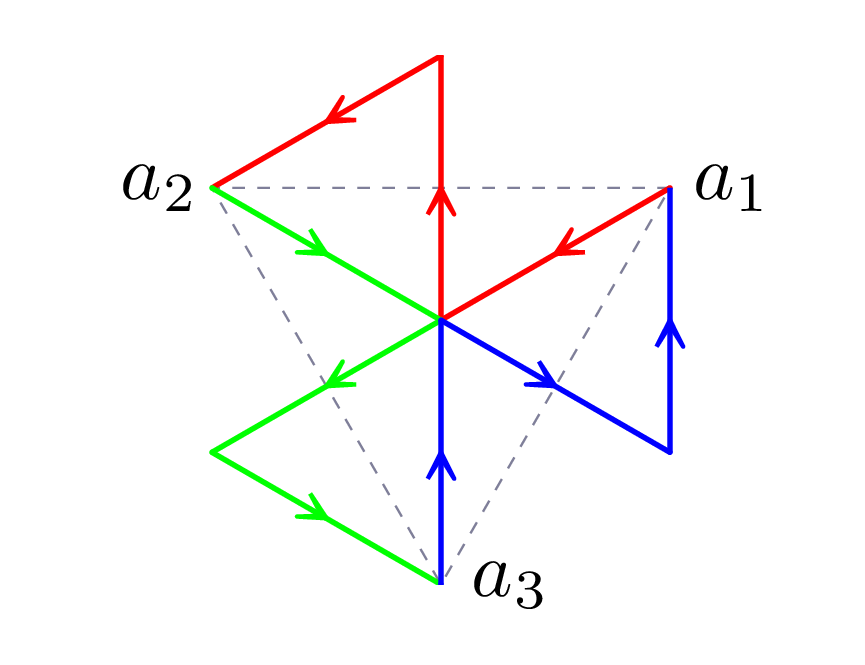}}\\
  \caption{An edge-to-trail substitution for Terdragon: an edge in (b)
  maps to the trail in (c) with the same color.
  Detail is given in Example \ref{Tdragon}.}
  \label{Tdragon}
\end{figure}

 Using the skeleton $A=\{a_1,\dots, a_m\}$, we define
 \begin{equation}
G_0=\{\overrightarrow{a_ia_j}; 1\leq i,j\leq m\}
 \end{equation}
  which is  the directed complete graph with vertex set $A$.

  We  select a spanning subgraph $\Lambda=(A,V)$ of $G_0$ as the initial  graph of our construction, where $A$ is the vertex set and $V$ is the edge set.
 Let $G$ be the union of   affine images of $V$ under $S_j$, i.e.,
  \begin{equation}
  G=\bigcup_{j=1}^N S_j(\Lambda).
  \end{equation}
 We call $G$ the \emph{refined graph} induced by $\Lambda$.
  (See Section \ref{sec:rule}.)

  If for each edge $u\in V$, we can find a trail $P_u$  in $G$ which shares
   the origin and terminus  with $u$, then we call the mapping
  $$u\mapsto P_u, u\in V$$
   an \emph{edge-to-trail substitution}.

\subsection{Feasible edge-to-trail substitutions}
%Now we look for conditions which guarantee an edge-to-trail substitution  leading to a space-filling curve.

Now we investigate when a rule $\tau$ leads to a space-filling curve, or equivalently, an optimal parametrization.  Rao and Zhang \cite{RaoZh15} introduce
  \emph{linear graph-directed iterated function system} (linear GIFS), which
    provides a criterion for optimal parameterizations. (See also Section \ref{sec:linear}.)

  \begin{theorem}\label{thm:RaoZh} (\cite{RaoZh15})
Let  $\{E_j\}_{j=1}^N$
be the invariant sets
 of a linear GIFS satisfying the open set condition and $0<\mathcal{H}^\delta(E_j)<\infty$ for $1\leq j \leq N$, where $\delta$ is the similarity dimension.
 Then $E_j$ admits  optimal parameterizations for every $j=1,\dots, N$.
\end{theorem}

An edge-to-trail substitution induces a linear GIFS in a nature way (Theorem \ref{thm:induce}).
We will impose conditions on  edge-to-trail substitutions so that
Theorem \ref{thm:RaoZh} can be applied.

First, using our point of view,  we summarize the constructions of SFC  in the literature into two classes,
the \emph{traveling-trail method} and the \emph{positive Euler-tour method}, and  provide a rigorous treatment to them.

\medskip

 \emph{(1) Traveling-trail method and self-similar zippers.}

 A trail of length $N$ in the refined graph $G=\bigcup_{j=1}^NS_j(\Lambda)$ is called a \emph{traveling trail}, if for every $j\in\{1,\dots, N\}$, the trail contains exactly one edge in  $S_j(\Lambda)$. We show that
\begin{theorem}\label{thm:travel} If all the trails $P_u$ in a edge-to-trail substitution  $\tau$ are
traveling trails, then $\tau$ leads to a space-filling curve of $K$.
\end{theorem}
\begin{figure}
  \centering
  \subfigure[\text{ Hexaflake fractal.}]{\includegraphics[width=0.28 \textwidth]{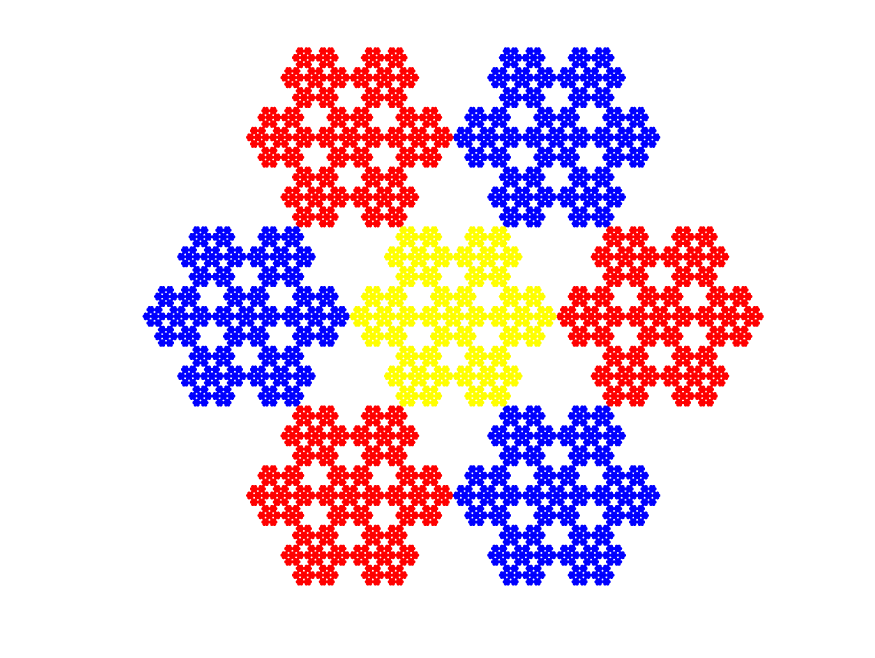}}
  \subfigure[{By traveling-trail method.}]{\includegraphics[width=0.29 \textwidth]{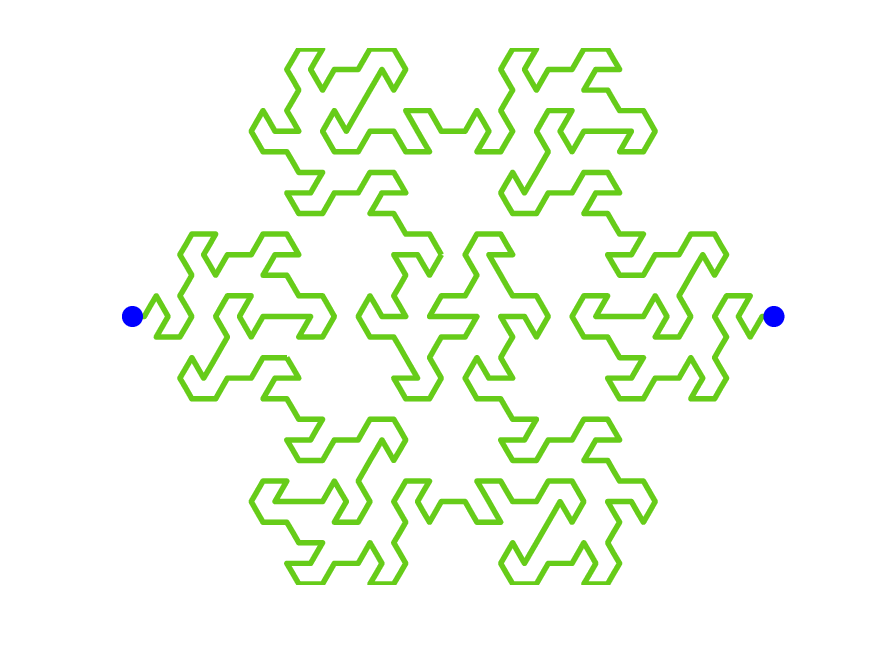}}
  \subfigure[{By Euler-tour method.}]{\includegraphics[width=0.28 \textwidth]{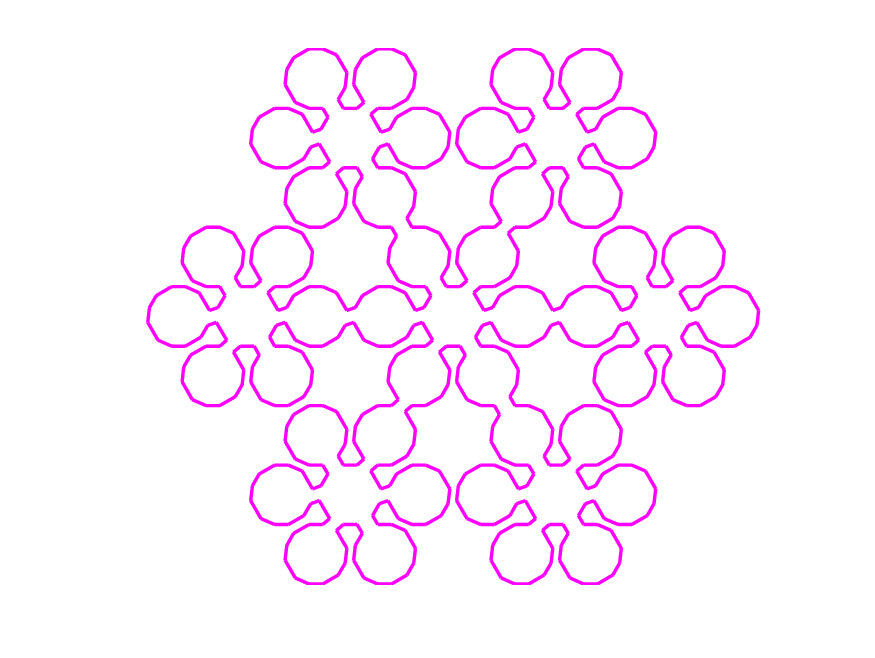}}\\
  \subfigure[Traveling trail from $a_4$ to $a_1$]{\includegraphics[width=0.3\textwidth]{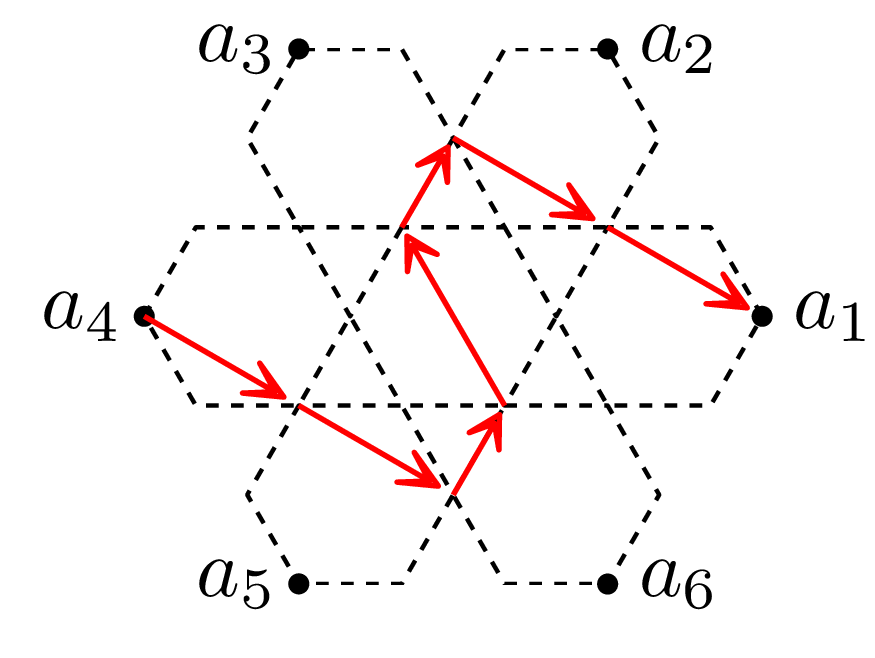}}
  \subfigure[Traveling trail from $a_4$ to $a_5$]{\includegraphics[width=0.3\textwidth]{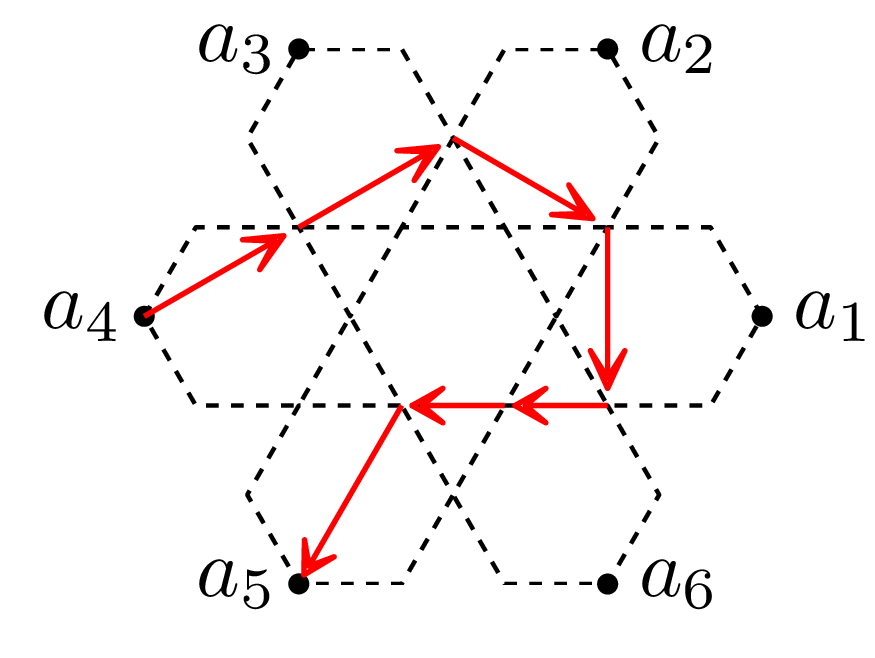}}
  \subfigure[Traveling trail from $a_4$ to $a_6$]{\includegraphics[width=0.3\textwidth]{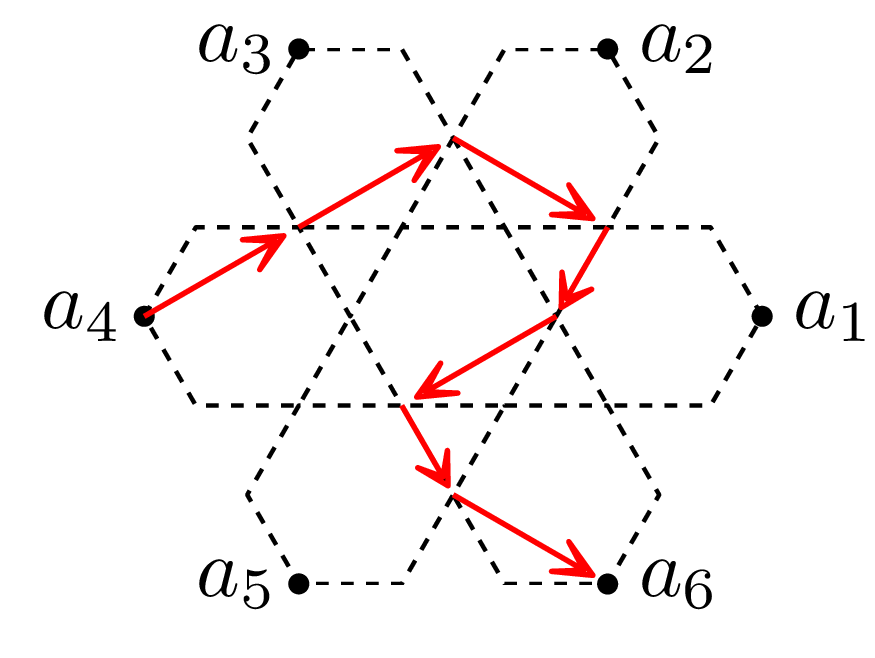}}
  \caption{Applying the traveling-trail method, we construct a space-filling curve of
  the Hexaflake fractal, see  Figure (b). The details are given in Example \ref{ex-Hexa}.
  The Hexaflake can also be parameterized by the positive Euler-tour method, see Figure (c).
  }\label{Hexaflake2}
\end{figure}

The following SFCs  are constructed by this method (see Section \ref{sec:travel}):
 Peano curve,  Hilbert curve,  Heighway dragon curve, Gosper curve, curves in Fukuda \text{et al.} \cite{Fukuda} and the  web-sites \cite{Ventrella, Gary}.

%including the wedge tiles, the sphinx tile,
%the chair tile, \emph{etc.}.

A special class of IFS, called self-similar zipper, is first introduced by Thurston (\cite{Thurston86}) and plays a role in
 complex analysis  (\cite{Astala88}). Recently, there are some works on self-similar zippers on the fractal aspect \cite{Aseev03, Tetenov08, Tetenov16}.
\begin{defi}
\rm{Let $(S_j)_{j=1}^N$ be an  IFS where the mappings are ordered. If there exists  a set $\{x_0,\dots, x_N\}$ of points and a sequence $(\beta_1,\dots, \beta_N)\in \{-1,1\}^N$, such that the mapping $S_j$ takes the pair $(x_0,x_n)$ either into the pair $(x_{j-1}, x_j)$ if $s_j=1$ or into the pair $(x_j,x_{j-1})$ if $s_j=-1$, then we call $(S_j)_{n=1}^N$ a  \emph{self-similar zipper}.

We call $\{x_0,\dots, x_N\}$ the set of vertices and call $(\beta_1,\dots, \beta_N)$ the vector of signature.}
\end{defi}

 Indeed, there is a one-to-one correspondence between self-similar zippers and `symmetric' linear GIFS with two states.
 % Hence, self-similar zippers  enjoy nice fractal properties.

\begin{thm}\label{lem:zipper} An IFS  $(S_i)_{i=1}^N$ is a self-similar zipper with signature $(\beta_1,\dots, \beta_N)$ if and only if the following ordered GIFS with two states
\begin{equation}\left \{
\begin{split}
E_1=&S_1(E_{\beta_1})+S_2(E_{\beta_2})+\dots+S_N(E_{\beta_N}),\\
E_{-1}=&S_N(E_{-\beta_N})+S_{N-1}(E_{-\beta_{N-1}})+\dots+S_1(E_{-\beta_1})
\end{split}
\right .
\end{equation}
 is a linear GIFS.
If the OSC holds in addition, then  the invariant set $K$  of $(S_i)_{i=1}^N$  admits optimal parameterizations.
\end{thm}

 This theorem is proved in  Section \ref{sec:linear}. Actually, the proof  gives us an easy algorithm to determine whether an IFS is a zipper or not.
 \medskip

 \emph{(2) Positive Euler-tour method.}

A natural selection of the initial graph is
$
\Lambda=\overrightarrow{a_1a_2}+\dots+\overrightarrow{a_{m-1}a_m}+\overrightarrow{a_ma_1},
$
  the cycle passing all the elements of $A$.
Next, we choose an Euler tour of
 the refined graph $G$ and a  partition of this Euler tour. This partition can
 give us an edge-to-trail substitution, if a \emph{consistency condition} is fulfilled (see Section \ref{sec:euler}).
%In the present paper, we give a theoretical analysis of the  positive Euler-tour method.
Besides, we pose two more conditions: the \emph{primitivity condition} and \emph{pure-cell condition}.
 We show that if
  these conditions are fulfilled, then the associated   edge-to-trail substitution
   leads to a SFC of $K$ (Theorem \ref{thm:old}).

 The following curves are constructed by this method: Sierpi\'nski curve,
 Terdragon curve in Dekking \cite{Dekking82}, the four-tile star in \cite{Gary} (see Figure \ref{Wedge-1}(b)).
 In this paper,  we present two new examples:  Sierpi\'nski carpet (Example \ref{EX:carpet}) and the Rocket tile (Example \ref{Rocket2}).

 \begin{figure}[h]
  \centering
  \subfigure[\text{Initial graph.}]{\includegraphics[width=.25 \textwidth]{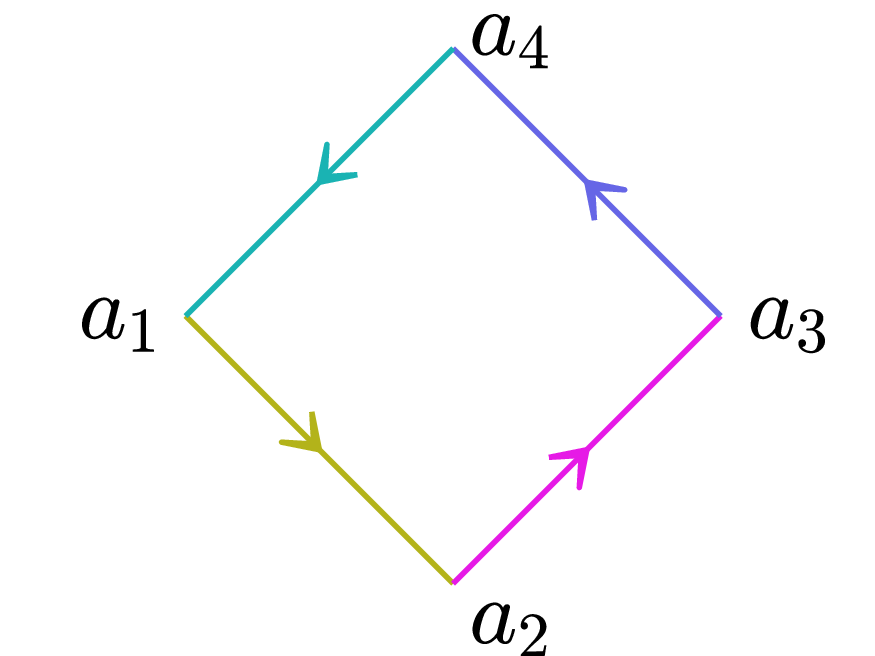}}
  \subfigure[\text{Euler tour.}]{\includegraphics[width=.25 \textwidth]{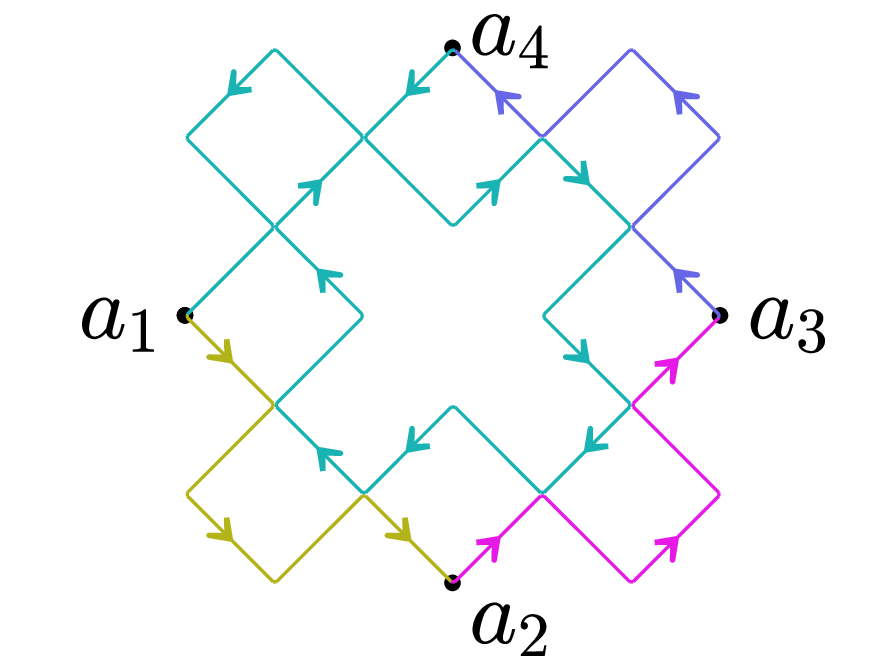}}
  \subfigure[\text{Initial patterns.}]{\includegraphics[width=.25 \textwidth]{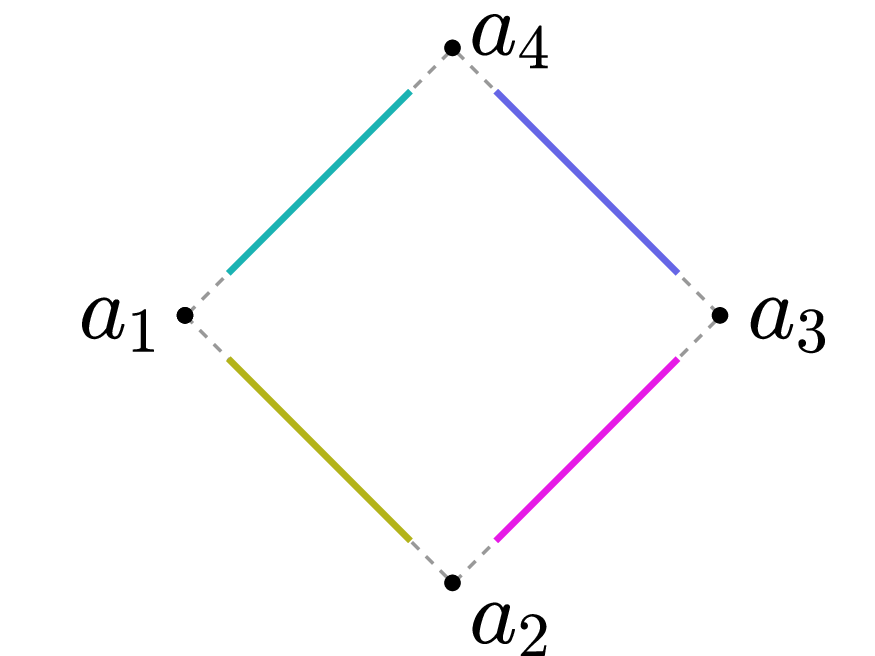}}\\
  \subfigure[\text{The first iteration.}]{\includegraphics[width=.25 \textwidth]{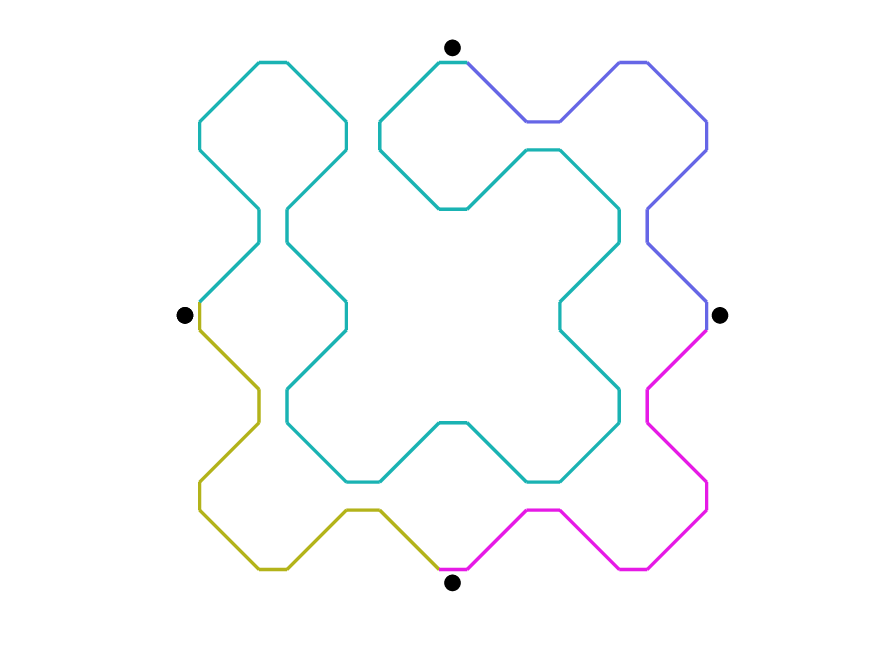}}
  \subfigure[\text{The second interation.}]{\includegraphics[width=.25\textwidth]{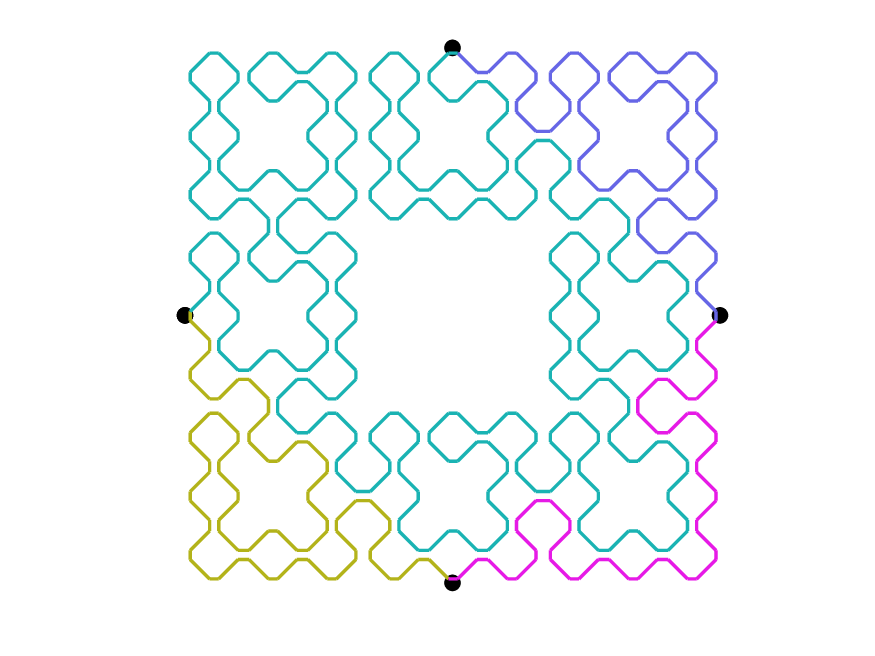}}
  \subfigure[\text{Invariant sets of  GIFS.}]{\includegraphics[width=.25 \textwidth]{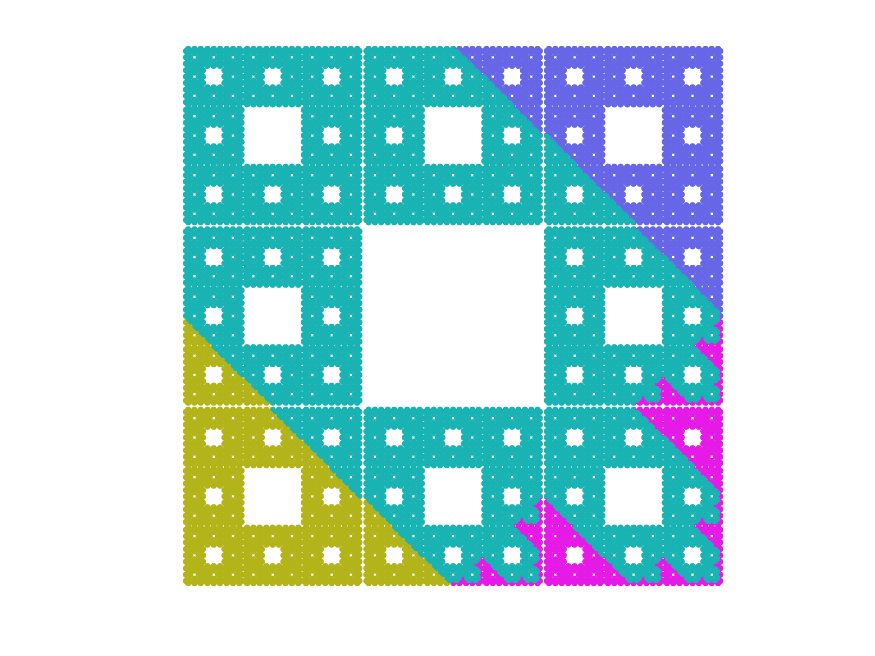}}
  \caption{A space-filling curve of the Sierpi\'nski carpet.}\label{carpet000}
\end{figure}

 Does every self-similar sets admit space-filling curves?
 To answer this question, we introduce a general Euler-tour method.

 \medskip

 \emph{(3) (General) Euler-tour method.}

Let $\Lambda^{-1}$ be the reverse cycle of $\Lambda$ in the positive Euler-tour method.
 Now, in the refined graph $G$, we allow \emph{negative orientation}, that is,
  for some $j$, we replace  $S_j(\Lambda)$  by $S_j(\Lambda^{-1})$.
 Then we have much more choices, which produce more refined graphs and Euler tours,
  and increases the possibility of finding a feasible edge-to-trail substitution.
 We  show that

\begin{theorem} \label{thm:main}
Let $\{S_j\}_{j=1}^N$ be an IFS  possessing  skeletons  and satisfying the open set condition.
Then the invariant set $K$ admits   space-filling curves.
More precisely,  either an rearrangement of $\{S_j\}_{j=1}^N$  is a self-similar zipper, or $K$ admits  space-filling curves constructed by the Euler-tour method.
\end{theorem}

%\begin{remark}{\rm
The difficult part of Theorem \ref{thm:main} is to prove the non self-similar zipper case, which we divide into two steps.
First, we prove Theorem \ref{thm:old}, which
transfers the space-filling curve problem to a graph theory problem.
Then, we solve  the graph theory problem in Section \ref{sec:consistency} and \ref{sec:primitive}, where we use a bubbling process to produce the desired Euler tour and
then make a suitable choice of orientations of the basic cells.

  Self-similar sets of \emph{finite type} is an important class of fractals,
see for instance, \cite{RaoWen98, NgaiWang01, BandtMesi09}.
In a subsequent paper  \cite{RaoZh17},
we show that if $K$ is a self-similar set of finite type, then $K$ possesses skeletons.
Consequently, we have

\begin{thm}\label{thm:finite-type}
Let $K$ be  a connected self-similar set  of finite type and satisfying the open set condition. Then $K$ admits   space-filling curves.
 \end{thm}

  \begin{example}\label{Rocket1}
\textbf{Integral self-affine tiles.} {\rm Let $A$ be an integral $n\times n$ expanding matrix, and $D\subset {\mathbb Z}^n$
be a set with $\# D=|\det A|$, where $\# D$ denotes the cardinality of $D$. Let $T$ be the unique compact set satisfying
$$
AT=\bigcup_{d\in D} T+D.
$$
The set $T$ is called an \emph{integeral self-affine tile} if it has positive Lebesgue measure. (See Lagarias and Wang \cite{LW96} and the reference therein).
By Theorem \ref{thm:finite-type}, every integral self-affine tile admits space-filling curves.

The Rocket tile is taken from Duvall \textit{et al.} \cite{Vince2000}. It is an integral self-similar tile with  $9$ branches.
A SFC is shown in  Figure \ref{fig:Rocket}. The details are given in Example \ref{Rocket2}.

}
\begin{figure}[h]
\subfigure[]{\includegraphics[width=.4 \textwidth]{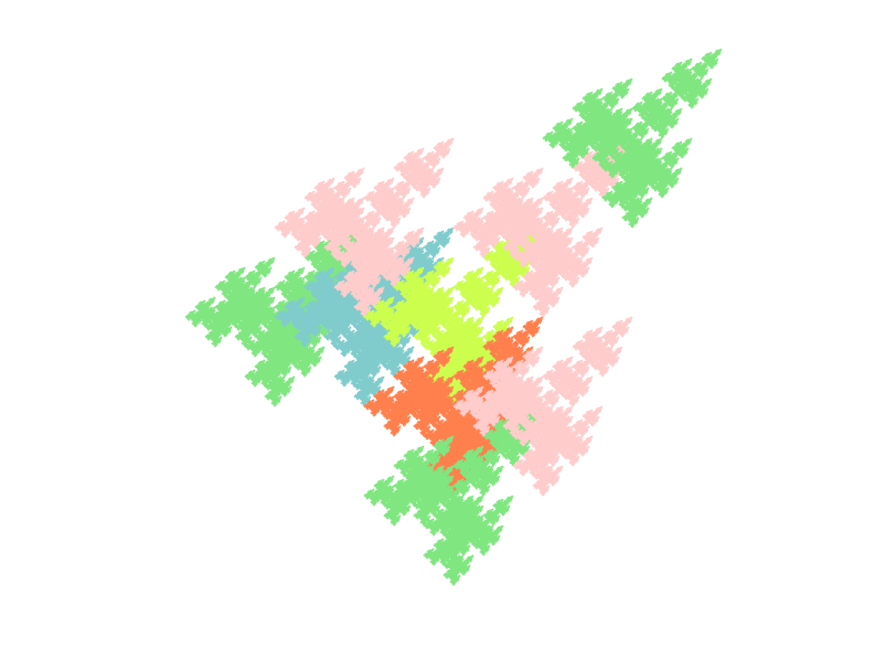}}
   \subfigure[]{\includegraphics[width=.4 \textwidth]{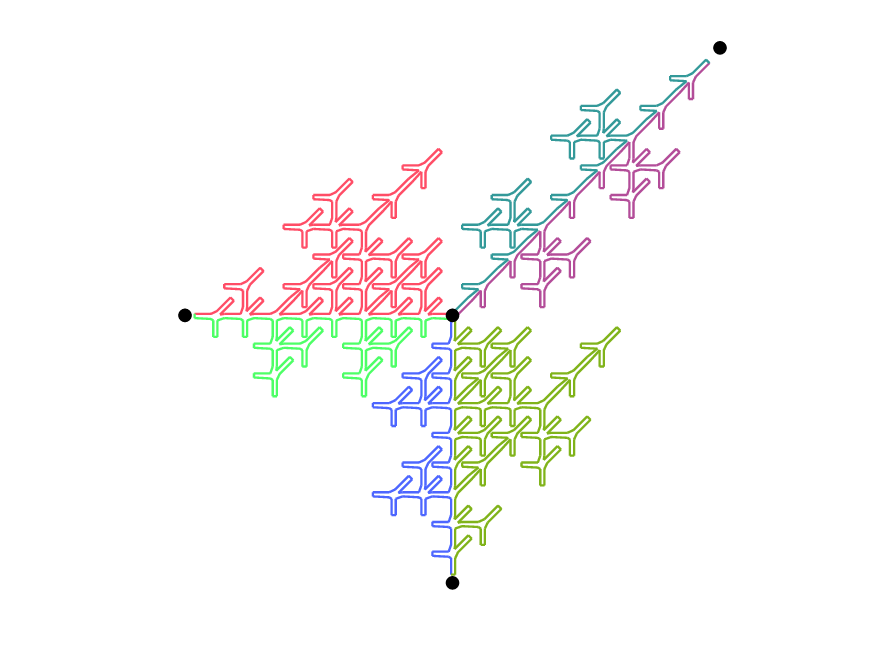}}
  \caption{A space-filling curve of the Rocket tile.}
  \label{fig:Rocket}
\end{figure}
\end{example}

\begin{example}
\textbf {The Christmas tree.} {\rm
The Christmas tree is a fractal with $5$ branches. Figure \ref{Fig_Christree} provides a space-filling curve of it, where the negative orientation
is involved; the details are given in Example \ref{Ex-Christ}.
It can be parameterized with the positive Euler-tour method, but the parametrization is  not measure-preserving,
since the edge-to-trail substitution cannot be primitive.
}
\end{example}

%\begin{figure}
%\subfigure[Christmas tree.]{\includegraphics[width=.45 \textwidth]{Chmas_tree.eps}}
%\subfigure[The third approximation.]{\includegraphics[width=.45 \textwidth]{Chmas_third.eps}}
%\caption{Christmas tree and its space-filling curve.}
%\label{fig:tree}
%\end{figure}

 \medskip

 The paper is organized as follows. In Section \ref{sec:skeleton}, we give a brief description of skeletons of self-similar sets.
 In Section \ref{sec:rule}, we describe the general philosophy of constructing edge-to-trail substitutions by graphs.
 Section \ref{sec:travel} and Section \ref{sec:positive} are devoted to the traveling-trail method and the positive Euler-tour method, respectively.
 In Section  \ref{sec:linear} and
 Section \ref{sec:induce}, we show an induced GIFS is always a linear GIFS, and Theorem \ref{thm:travel} is proved there.
   Section 8-11 are devoted to the  Euler-tour method, and Theorem \ref{thm:main} is proved in Section \ref{sec:pure}.

 % sec:intro

 \section{\textbf{Skeleton of self-similar set}}\label{sec:skeleton}

In this section, we give a brief introduction to skeletons of self-similar sets.
A  detailed study is carried out in Rao and Zhang \cite{RaoZh17}.

Let $\cS=\{S_j\}_{j=1}^{N}$ be an IFS with invariant set   $K$.
For any subset $A$ of $K$, we define a graph $H(A)$ as following:
The    vertex set is $\{S_1,S_2,\dots,S_N\}$, and there  is an edge between two vertices $S_i$ and $S_j$ if and only if $S_i(A)\cap S_j (A)\neq\emptyset$. We call $H(A)$  the \emph{Hata graph} induced by $A$.

\begin{remark}{\rm Such graphs are first studied by Hata \cite{Hata85}, where he proved that
a self-similar set $K$ is connected if and only if the graph $H(K)$ is connected.}
\end{remark}

\begin{defi}\label{def-skeleton}
{\rm Let $K$ be a connected self-similar set, and let $A$ be a finite subset of $K$.  We call $A$ a \emph{skeleton} of $\{S_j\}_{j=1}^{N}$  (or  $K$), if $A \subset \bigcup_{j=1}^N S_j(A)$
 and the Hata graph $H(A)$ is connected.}
\end{defi}

\begin{remark}{\rm
 Kigami \cite{Kigami} and Mor\'an \cite{Mor99} have studied the `boundary' (also called 'vertices' if it is finite) of a fractal.
 A skeleton  is usually chosen to be a  subset of the `boundary' of a self-similar set
 since we want a small skeleton;   for a so-called p.c.f. self-similar set, the set of vertices is a skeleton. Indeed,   the choices of skeletons are much more arbitrary.
}\end{remark}

\subsection{ Iteration}
We denote
$\Sigma=\{1,\dots, N\}$
and call it an \emph{alphabet}.
Let $I=i_1i_2\dots i_n\in \Sigma^n$, we call it a \emph{word of length $n$}.  We denote the length of a word $I$ by $|I|$.
We define the $n$-th \emph{iteration} of $\mathcal{S}$ to be the IFS
 $$\cS^n=\{S_I;~I \in \Sigma^n\},$$
 where  
 $$S_I=S_{i_1}\circ S_{i_2}\circ\cdots\circ S_{i_n},\quad \text{ if }  I=i_1\dots  i_n.$$
It is well-known  that the invariant set $K$ of the IFS $\cS$ is also the invariant set  of $\cS^n$ (see Falconer \cite{Fal90}).
Similarly, we have

\begin{pro} (\cite{RaoZh17})
 If $A$ is a skeleton of $\mathcal{S}$, then $A$ is also a skeleton of $\mathcal{S}^n$.
\end{pro}

  Using neighbour graph of self-similar sets,  it is shown (\cite{RaoZh17})  that
a connected self-similar set of finite type always
   possesses skeletons, and  actually   an algorithm of finding skeletons is given there.
There do exist self-similar sets without skeletons.

\begin{figure}[h]
  \centering
  \includegraphics[width=0.4\textwidth]{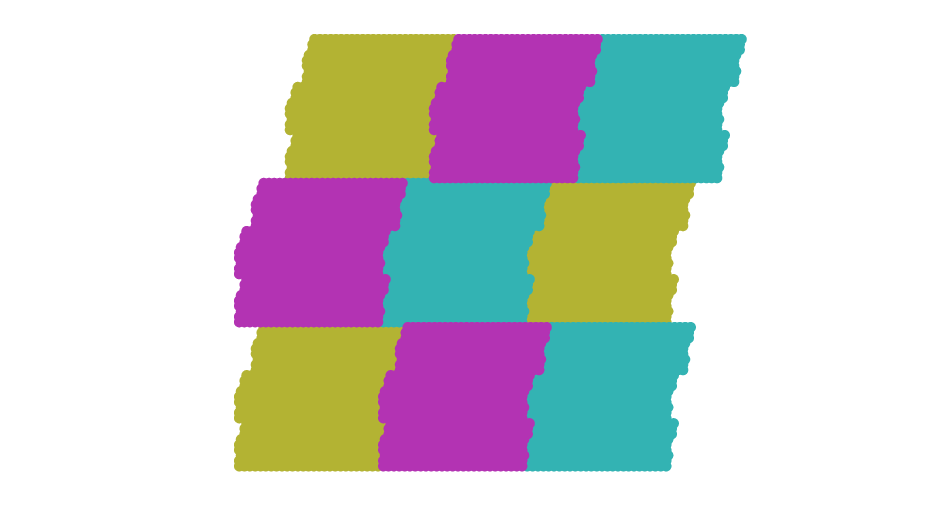}
  \caption{A reptile  without skeleton, see \cite{RaoZh17}.}\label{non-ske}
\end{figure}
 % sec:skeleton

\section{\textbf{Graphs and edge-to-trail substitutions: the general philosophy}}\label{sec:rule}
%In this section, we provide a general philosophy of how to find substitution rules
%leading to space-filling curves.
Let $\mathcal{S}=\{S_j\}_{j=1}^N$ be an IFS   possessing a  skeleton  and  satisfying the OSC.
Let us denote its invariant set by $K$,
and let  $A=\{a_1,a_2,\dots,a_m\}$ be a  skeleton. Recall
$$G_0=\{\overrightarrow{a_ia_j};~1\leq i,j\leq m\},$$
 is a directed complete graph with vertex set $A$.
 We note that the edges
$\overrightarrow{a_ia_j}$ are abstract edges rather than oriented line segments;
  moreover, $i$ may be equal to $j$, see Example \ref{ex-Wedge}.

Next, we choose a subgraph  $\Lambda=(A,V)$  of $G_0$, and we call $\Lambda$ the \emph{initial graph}.
To continue our construction, we need to define the \emph{affine copy} of a directed graph.

 \begin{defi}{\rm Let $G=(\mathcal{A},\Gamma)$ be a directed graph such that
 the vertex set  $\mathcal{A} \subset \mathbb{R}^d$.
  Let $S: {\mathbb R}^d\to \R^d$ be a affine mapping. We define a directed graph $G_S=(S(\mathcal{A}), \Gamma_S)$ as follows:
there is an edge in $\Gamma_S$ from $S(x)$ to $S(y)$,
if and only if there is an edge $e\in \Gamma$ from vertex $x$ to $y$. Moreover, we denote this edge by  $(e,S)$.
For simplicity, we shall denote $G_S,  \Gamma_S,$ and $(e,S)$ by $S(G),  S(\Gamma)$ and $S(e)$, respectively.
}
\end{defi}

\begin{remark}{\rm (i) If  $(\mathcal{A}_1,\Gamma_1)$ and $(\mathcal{A}_2,\Gamma_2)$ are two graphs
 without common edges, then we define  their union   to be the graph $(\mathcal{A}_1\cup \mathcal{A}_2, \Gamma_1\cup \Gamma_2)$.

 (ii) Even if $S_{j}(e_k)$ coincides with $S_{j'}(e_{k'})$ as oriented line segment, they should regarded  as different edges,
 since $(e_k, S_{j})\neq (e_{k'}, S_{j'})$.
 }
 \end{remark}

\subsection{Refined graph and edge-to-trail substitution}
Let $G$ be the union of affine images of $\Lambda$ under $S_j$, that is,
\begin{equation}
G=\bigcup_{j=1}^NS_j(\Lambda),
\end{equation}
and we call it  the \emph{refined graph induced by $\Lambda$}.

Let $\tau$ be a mapping from $V$  to trails of $G$; we shall denote $\tau(u)$ by $P_u$ to emphasize that $\tau(u)$ is a trail. We call $\tau$   an \emph{edge-to-trail substitution}, if for all $u\in V$,
 $P_u$  has the same origin and terminus  as $u$.

An edge-to-trail substitution $\tau$ can be thought as replacing each big edge  $u$ by  a trail $P_u$ consisting of
 small edges. Our goal is to construct feasible edge-to-trail substitutions which lead to SFCs.

\subsection{Iteration of  edge-to-trail substitutions}
%Now we define the iteration of $\tau$. 
%Let $\{1,2,\dots, N\}^*$ denote the set of finite words over the alphabet $\{1,2,\dots, N\}$.
We use the following two rules to iterate $\tau$:

(i) For $I\in\bigcup_{n\geq 1} \{1,2,\dots,N\}^n$ and $u\in V$,  if $\tau(u)=\gamma_1+\dots+\gamma_\ell$, we  set
\begin{equation}\label{iter}
\tau(S_I(u))=S_I(\gamma_1)+\dots+S_I(\gamma_\ell);
\end{equation}

 (ii) Set
 $$\tau( w_1+w_2+\dots+ w_k)=\tau(w_1)+\tau(w_2)+\dots +\tau(w_k)$$
 if $w_j\in \{S_{I}(v); I\in \{1,2,\dots, N\}^n, v\in V\}$ for some $n$.

   Hence, we can define $\tau^n(u)$ recurrently, which is a trail consisting of small edges.
  Geometrically,  we can explain $\tau^n(u)$ as an oriented broken line which  provides an approximation of the corresponding SFC.

  % sec:rule

\section {\textbf{Traveling-trail method}}\label{sec:travel}

Recall that a trail $P$ in $G=\bigcup_{j=1}^NS_j(\Lambda)$ is called a \emph{traveling trail}, if for every $j$, $P$ contains exactly one edge in  $S_j(\Lambda)$.
Theorem \ref{thm:travel} asserts that if all the trails $P_u$ in a edge-to-trail substitution  $\tau$ are
traveling trails, then $\tau$ leads to a SFC of $K$.
We postpone the proof of Theorem \ref{thm:travel} to Section \ref{sec:induce}.
 In this section, we summarize the SFCs constructed by this method.

\subsection{ \textbf{Linear IFS}}
Let $K$ be a self-similar set. Assume that $K$ has a skeleton consisting of two points, say $A=\{a,b\}$. Denote  $u=\overrightarrow{ab}$ and let $V=\{u\}$. If
$S_1(u)+\cdots+S_N(u)$ is a trail from $a$ to $b$, where we use the symbol `+' to connect the
consecutive edges or sub-trails, then
\begin{equation}\label{Peano-rule}
\tau: u\mapsto S_1(u)+\cdots+S_N(u)
\end{equation}
is a feasible edge-to-trail substitution, since $\tau(u)$ is a traveling trail.
 Such IFS, called a linear IFS in \cite{RaoZh15},
was first studied by De Rahm \cite{deRham57} and Hata \cite{Hata85}, where they proved that a linear IFS leads to a SFC,  which is a direct generalization of
 Peano's original construction.

\subsection{ \textbf{Self-similar zipper}} Let $K$ be the 
 invariant set of a self-similar zipper $\{S_j\}_{j=1}^N$ with vertices $\{x_0, \dots, x_N\}$ and signature $\{\beta_1,\dots, \beta_N\}$.
 Then $K$ possesses a skeleton $A=\{x_0, x_N\}$.
  Denotes $u=\overrightarrow{x_0x_N}$ and set $V=\{u, u^{-1}\}$, where $u^{-1}$ denote the reverse edge of $u$. Clearly
 %If there exists a vector $\beta=(\beta_1,\dots, \beta_N)\in \{1,-1\}^N$ such that
%$S_1(u^{\beta_1})+\cdots+S_N(u^{\beta_N})$
%is a trail from $a$ to $b$, then
$$
\tau: \left \{
\begin{array}{rl}
u & \mapsto S_1(u^{\beta_1})+\cdots+S_N(u^{\beta_N})\\
 \quad u^{-1} &\mapsto \text{reverse trail of } S_1(u^{\beta_1})+\cdots+S_N(u^{\beta_N})
 \end{array}\right .
$$
is a  feasible edge-to-trail substitution since  both $\tau(u)$ and $\tau(u^{-1})$ are traveling trails.
(The path-on-lattice IFS in \cite{RaoZh15} is a special case of the self-similar zipper).
 Hilbert curve, Heighway dragon curve and Gosper curve are obtained by this way.
See  Section 5 of \cite{RaoZh15} for details.

\subsection{\textbf{Space-filling curves of polygonal reptiles.}}
In the web-site Teachout \cite{Gary}, there are many interesting SFCs of polygonal reptiles.  These curves are  obtained by traveling-trail method.  We take the Wedge tile as an example.

 \begin{example}\label{ex-Wedge} \textbf{The Wedge Tile.} {\rm   The Wedge tile is a self-similar set
generated by the IFS $\{S_i\}_{i=1}^4$, where the maps are indicated by Figure \ref{Wedge-2}(a),(b).
(Here we use an arrow to specify the linear part of $S_j$ contains reflection or not. )

 We choose $A=\{a_2, a_4\}$, two vertices of the wedge,  to be the skeleton.
Choose
$$V=\{\overrightarrow{a_4a_2},\overrightarrow{a_2a_4},  \overrightarrow{a_4a_4} \}=:\{\alpha, \beta, \gamma\}.$$ Then
$$\tau:~ \left \{
\begin{array}{rl}
\alpha & \mapsto S_1(\beta)+S_2(\beta)+S_3(\gamma)+S_4(\alpha),\\
\beta& \mapsto S_4(\beta)+S_3(\gamma)+S_2(\alpha)+S_1(\alpha),\\
  \gamma &\mapsto S_3(\beta)+S_4(\gamma)+S_2(\alpha)+S_1(\alpha)
  \end{array} \right .
  $$
is an edge-to-trail substitution.
 The trails  $\tau(\alpha)$ and $\tau(\gamma)$ are illustrated by
 Figure \ref{Wedge-2}(c) and (d); the trail  $\tau(\beta)$ is the reverse trail of $\tau(\alpha)$.
 A visualization of the SFC corresponding to $\tau$  is shown in Figure \ref{Wedge-1}(a).
  }
\end{example}

\begin{figure}[h]
 \centering
  \subfigure[\text{}]
 { \includegraphics[width=0.35\textwidth]{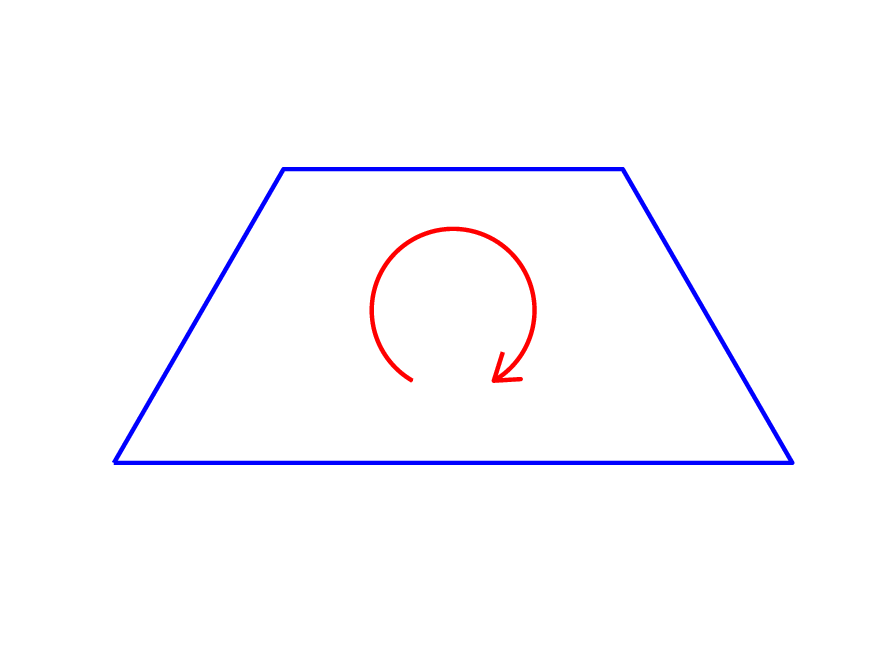}}
  \subfigure[\text{}]
  {\includegraphics[width=0.35\textwidth]{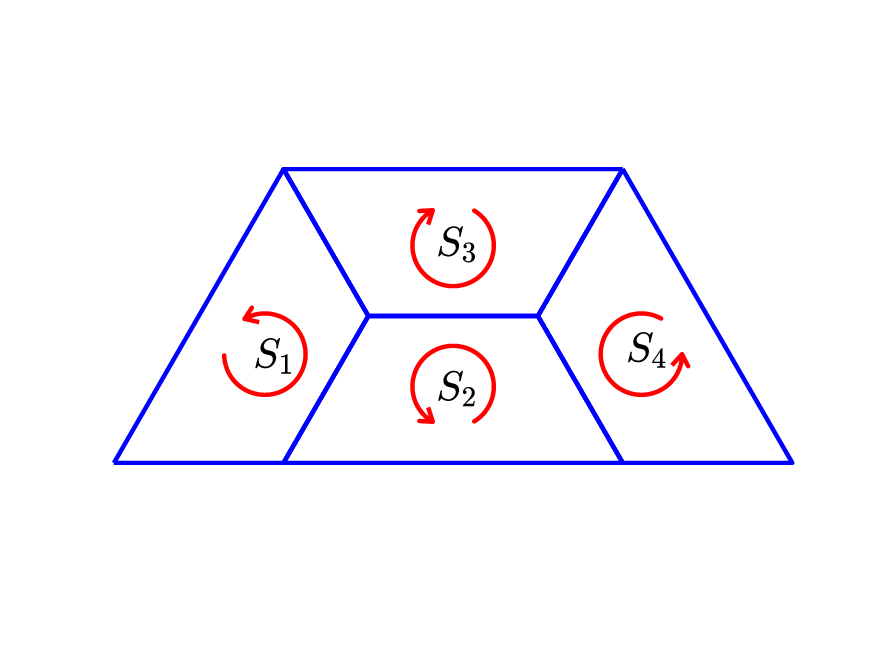}}\\
   \subfigure[\text{A trail from $a_4$ to $a_2$.}]
   { \includegraphics[width=0.35\textwidth]{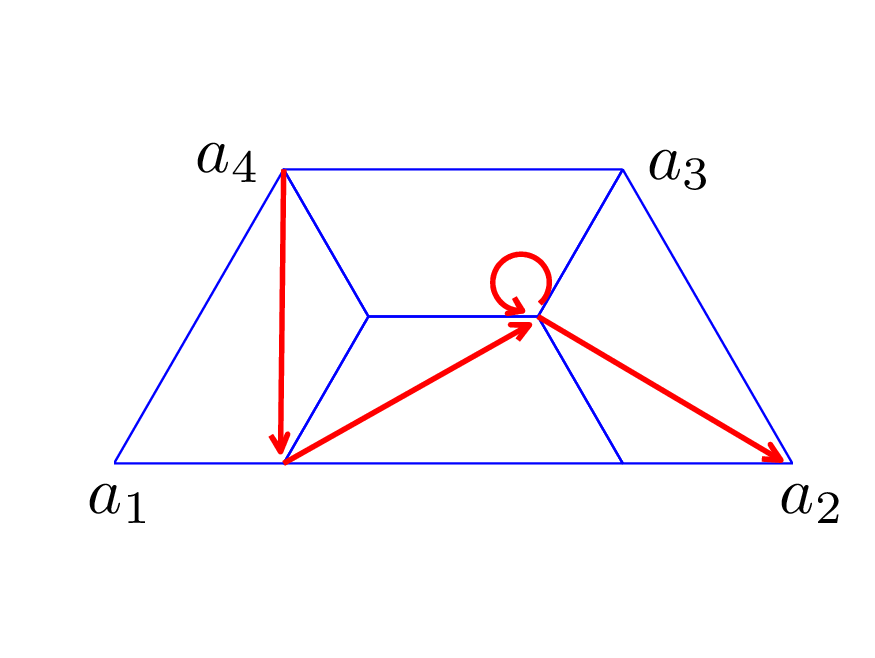}}
    \subfigure[\text{A closed trail from $a_4$ to $a_4$.}]
  {\includegraphics[width=0.35\textwidth]{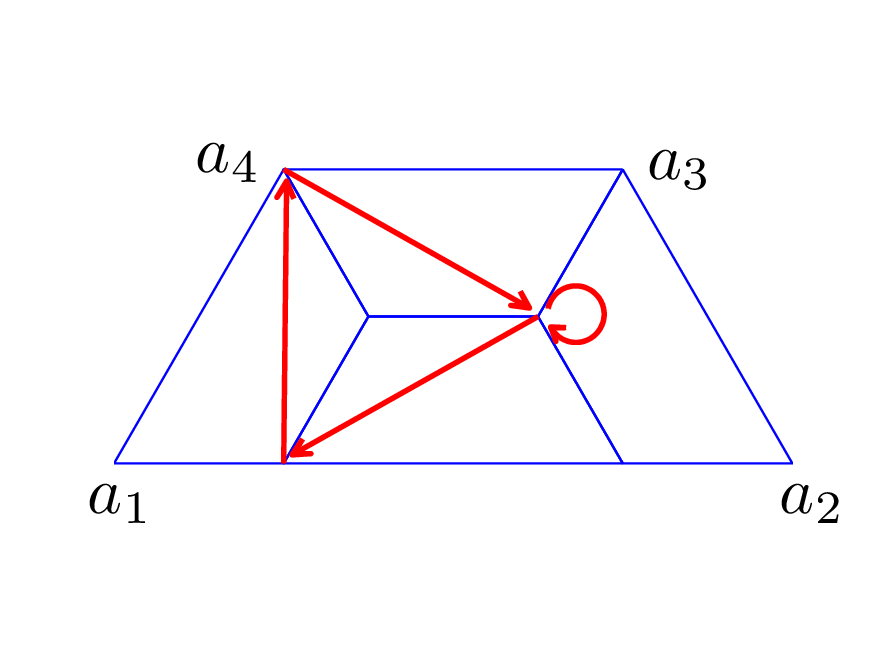}}
  \caption{}\label{Wedge-2}
\end{figure}

  \begin{example}\label{ex-Hexa} \textbf{Hexaflake.}
{\rm
 Hexaflake is a fractal constructed by iteratively exchanging each hexagon by a flake of seven hexagons, see Figure \ref{Hexaflake2} (a). We choose $A$ to be the vertex set of the original hexagon, and choose the initial graph to be
$$
V=\{\overrightarrow{a_ia_j}; ~i\neq j \}.
$$
Figure \ref{Hexaflake2}~(d)(e)(f) show the trails for $\overrightarrow{a_4a_1}$, $\overrightarrow{a_4a_5}$ and $\overrightarrow{a_4a_6}$. For any $\overrightarrow{a_ia_j}$, there is a similitude which maps
one of $\overrightarrow{a_4a_1}$, $\overrightarrow{a_4a_5}$ and $\overrightarrow{a_4a_6}$ to $\overrightarrow{a_ia_j}$, and hence this similitude induces a trail which we
define to be  $\tau(\overrightarrow{a_ia_j})$.
  The rule  $\tau$ satisfies the condition of Theorem \ref{thm:travel}.  Figure \ref{Hexaflake2}(b) gives a visualization of the   corresponding SFC.

 A more popular   sFC  of Hexaflake is coustructed by   the positive Euler-tour method, see Figure \ref{Hexaflake2}(c).
}

\end{example}

 % sec:travel

\section{\textbf{Positive Euler-tour method}}\label{sec:positive}
Another frequently used method of constructing SFC is the positive Euler-tour method.
   Let  $A=\{a_1,\dots, a_m\}$ be a skeleton. Define
 \begin{equation}\label{eq:cycle}
\Lambda=\Lambda_A:=\overrightarrow{a_1a_2}+\cdots +\overrightarrow{a_{m-1}a_m}+\overrightarrow{a_ma_1}
\end{equation}
 to be the cycle passing  $a_1,\dots,a_m$ in turn. We denote the edge set of $\Lambda$ by
 $$V^+=\{\overrightarrow{a_1a_2},\dots, \overrightarrow{a_{m-1}a_m}, \overrightarrow{a_ma_1}\}.$$
 We set $\Lambda$ to be our  initial graph, and let
 $G=\bigcup_{i=1}^N S_j(\Lambda)$ be the refined graph.

\begin{lemma}\label{lem:tour} The refine graph $G$     admits Euler tours.
\end{lemma}

\begin{proof} First, we prove $G$ is connected.
Take two vertices $S_i(a)$ and $S_j(b)$ in $G$, where $a, b\in A$.
Let $S_i=S_{i_0}, S_{i_1},\dots, S_{i_k}=S_j$ be the vertices of a path connecting $S_i$
and $S_j$ in the Hata graph $H(\cS, A)$. Then one can easily prove by induction that the graphs
$$
 S_{i_0}(\Lambda)\cup \cdots \cup S_{i_\ell}(\Lambda),  \quad 0\leq \ell\leq k,
$$
are all connected. Set $\ell=k$, we get that $G$ is connected.

Finally,  $G$ admits Euler tour, since it is a disjoint union
 of closed trails (see \cite{Biggs}).
\end{proof}

  Let $P$ be an \emph{Euler tour}
 of the refined graph $G$.
 Let $P=P_1+\cdots+P_m$ be a partition of $P$ such that that for all $j$,
 $P_j$ is a sub-trail having the same origin and terminus as $v_j$. Then
 $$
 \tau: v_j\mapsto P_j,\quad j=1,2,\dots, m,
 $$
 gives us an edge-to-trail substitution.
 If $P$ and the partition are carefully chosen, then   $\tau$
 will lead to a SFC. The rigorious treatment is given in
   Section \ref{sec:euler}.

\begin{figure}[h]
  \centering
 % \subfigure[\text{Terdragon}]{\includegraphics[width=0.35 \textwidth]{Terdragon_0010}}
  %\subfigure[\text{Initial cycle}]{\includegraphics[width=0.28 \textwidth]{Terdragon_002}}
  %\subfigure[\text{Eulerian path}]{\includegraphics[width=0.35 \textwidth]{Terdragon_003}}\\
  \subfigure[\text{Initial pattern}]{\includegraphics[width=0.31 \textwidth]{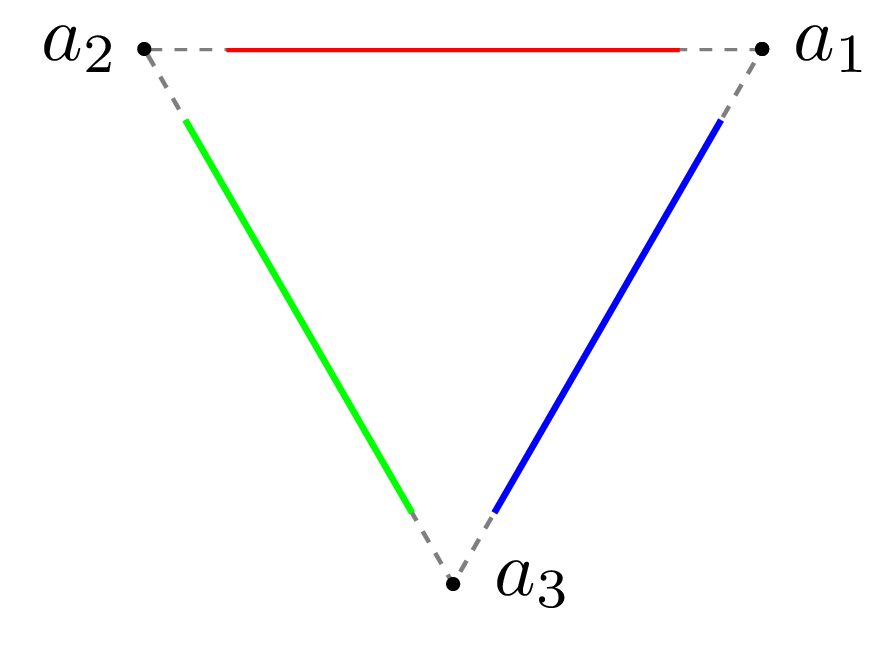}}
  \subfigure[\text{The second iteration.}]{\includegraphics[width=0.35 \textwidth]{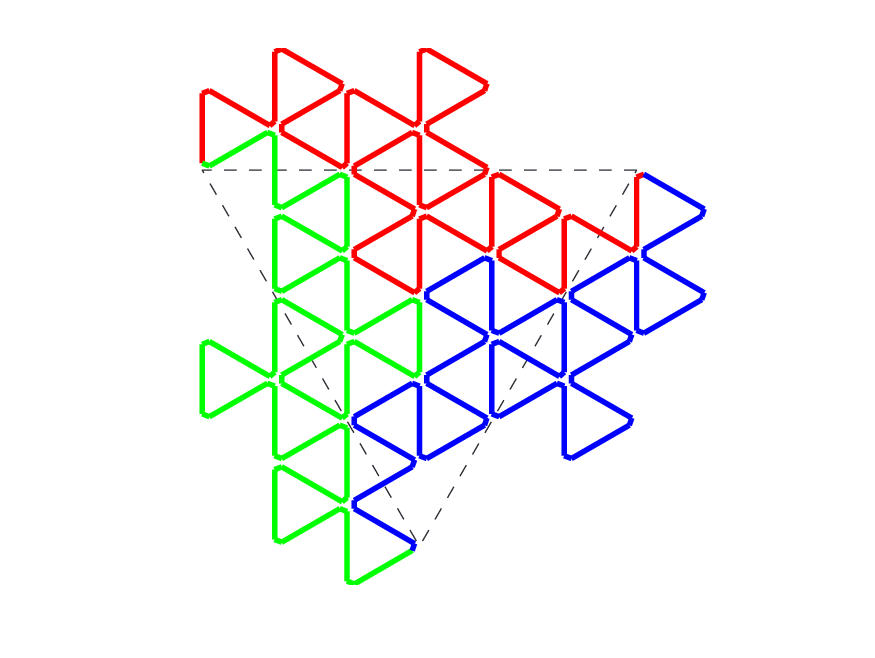}}
  \subfigure[\text{Invariant sets of the GIFS.}]{\includegraphics[width=0.32 \textwidth]{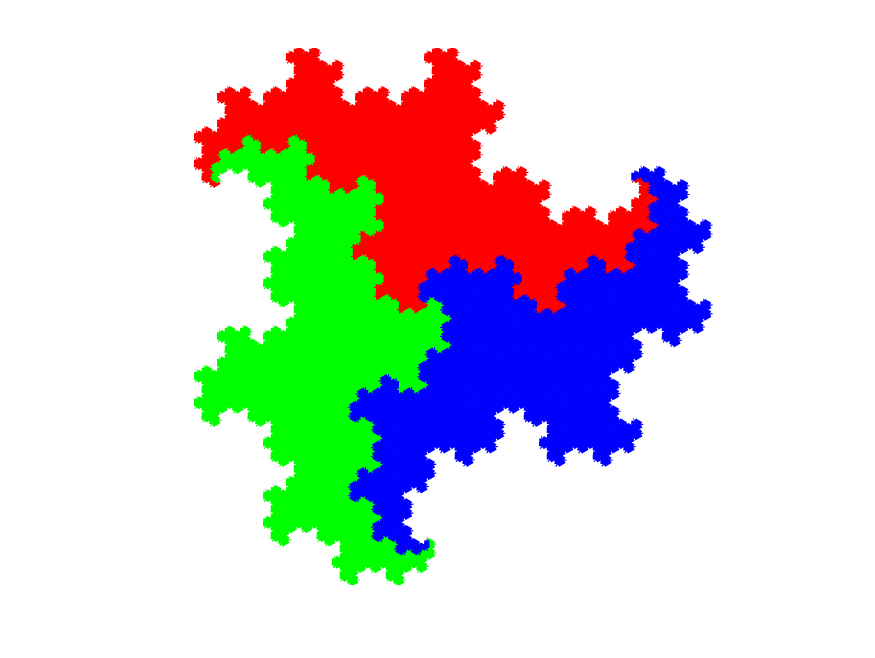}}
  \caption{A space-filling curve of Terdragon.}
  \label{fig:Tdragon:2}
\end{figure}

\begin{example}\label{Terdragon-1} \textbf{Terdragon.} {\rm Terdragon is generated by the IFS
$$\{S_1(z)=\lambda z+1,\ ~S_2(z)=\lambda z+\omega,\ ~S_3(z)=\lambda z+\omega^2\},$$
where $\lambda=\exp(\pi \mi/6)/\sqrt{3}$ and $\omega=\exp(2\pi \mi/3)$. We choose the skeleton $$A=\{a_1,a_2,a_3\}=\{-\omega^2/\lambda,-1/\lambda,-\omega/\lambda\},$$
 which consists of the  fixed points of $S_j, j=1,2,3$.
  Choose the initial graph to be
 the cycle $$
  \Lambda=\overrightarrow{a_1a_2}+\overrightarrow{a_2a_3}+\overrightarrow{a_3a_1}:=v_1+v_2+v_3
  $$
  (see Figure \ref{Tdragon}(b)).  The refined graph
  $G=S_1(\Lambda)\cup S_2(\Lambda)\cup S_3(\Lambda)$
  is the graph in Figure \ref{Tdragon}(c).
  According to the of the Euler-partition   in Figure \ref{Tdragon}(d), we obtain the following edge-to-trail substitution:
\begin{equation}\label{T-rule}
\tau: \left \{
\begin{array}{l}
v_1\mapsto P_{v_1}=S_1(v_1)+S_2(v_3)+S_2(v_1),\\
v_2\mapsto P_{v_2}=S_2(v_2)+S_3(v_1)+S_3(v_2), \\
v_3\mapsto P_{v_3}=S_3(v_3)+S_1(v_2)+S_1(v_3).
\end{array}
\right .
\end{equation}
}
\end{example}

\begin{remark}
{\rm
We shall see in Section \ref{sec:euler} that SFCs obtained  by the Euler-tour method are  all  closed curves, and the invariant sets of the induced GIFS form a partition of the original self-similar set $K$.
}
\end{remark}

\begin{example}\label{EX:carpet}
 \textbf{Sierpi\'nski carpet.}  {\rm
 We choose the skeleton to be  the  middle points of the edges of the unit square, see Figure \ref{fig-carpet} (right).
   An Euler tour of the refine graph is shown in Figure \ref{carpet000} (b), and
    a partition of this Euler tour  is indicated by  colors. The edge-to-trail substitution  is
    indicated by Figure \ref{carpet000} (a) and (d). Figure \ref{carpet000} (f) shows the invariants of GIFS associated with the edge-to-trail substitution, which form a partition of the carpet.
     (Dekking \cite{Dekking82} gives a certain `parametrization' of the  carpet where there are infinitely many jumps.)

}
\end{example}

%%%%%%%%%%%%%%%%%%%%%%%%%%
%
% Rocket Tile
%
%%%%%%%%%%%%%%%%%%%%%%%%%%%%%
\subsection{A stronger version of positive Euler tour method.}
 The initial cycle $\Lambda$  can be
 chosen to be any closed trail passing all the elements of $A$, say,
 $$
 \Lambda=v_1+\cdots +v_h.
 $$
 Still we can define refined graph, Euler-tour and its partition, and edge-to-trial substitution rule.
 In the following example, to make the visualizations self-avoid, we
  chooses a special closed trial $\Lambda$.

\begin{example}\label{Rocket2}
\textbf{ The Rocket Tile.} {\rm
%This reptile tile is taking from Duvall, Keesling and Vince \cite{Vince2000}.
The Rocket tile  is generated by the IFS $\{S_j\}_{j=1}^9=\{(x+d)/3; {d\in D}\}$, where
$$D=\{-2,-1,-1+\mi,0,-\mi,-2\mi,1-\mi,1+\mi,2+2\mi\}.$$
See Figure \ref{Rocket}$(a)$.
It is easy to check that $A=\{a_1, a_2, a_3, a_4\}=\{0, -1, -\mi, 1+\mi\}$ is a skeleton.
 (The $a_j$'s  are  fixed points of $S_4, S_1, S_6 \text{ and } S_9$ respectively.)

 We choose the initial graph to be the closed trail
$$
V=\overrightarrow{a_1a_2}+\overrightarrow{a_2a_1}+\overrightarrow{a_1a_3}+\overrightarrow{a_3a_1}+\overrightarrow{a_1a_4}+
\overrightarrow{a_4a_1},
$$
see Figure \ref{Rocket}$(b)$. The refined graph of $V$ contains $54$ edges.
 An Euler tour of the refined graph is indicated by  Figure \ref{Rocket}$(c)$, where
    a partition of this tour is indicated by $6$ different colors.
    The edge-to-trail substitution  is indicated in  Figure \ref{Rocket}$(b)$ and $(e)$.

}
\begin{figure}
\subfigure[\text{Rocket.}]{\includegraphics[width=.35 \textwidth]{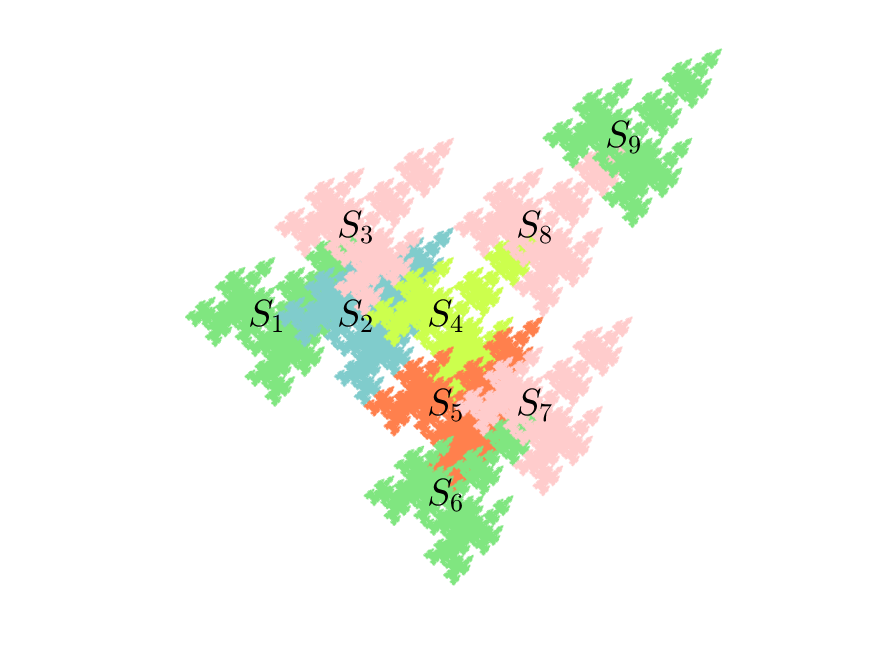}}
\subfigure[\text{Initial graph.}]{\includegraphics[width=.3 \textwidth]{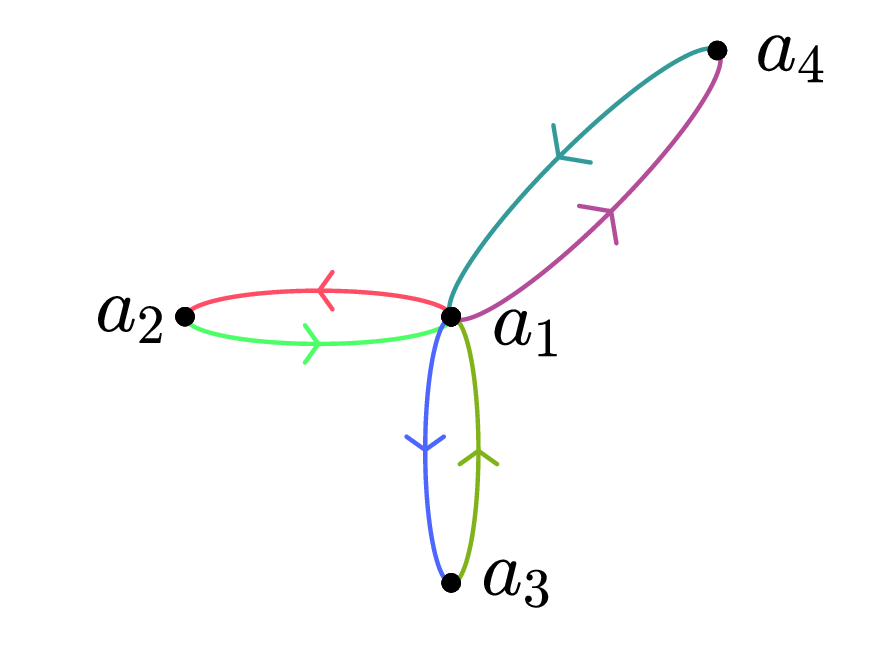}}
\subfigure[\text{Euler tour.}]{\includegraphics[width=.32 \textwidth]{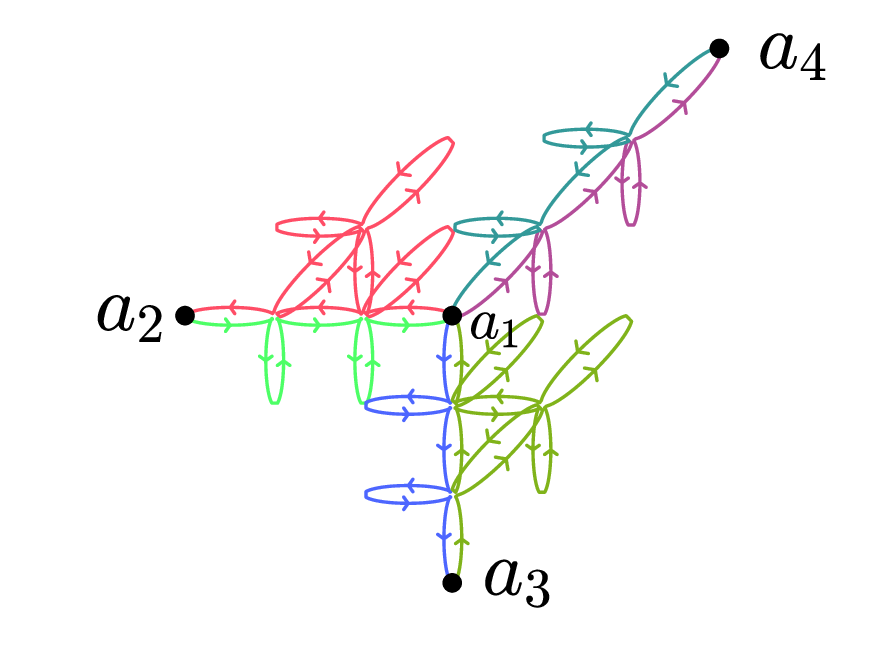}}\\
  \subfigure[\text{Initial patterns.}]{\includegraphics[width=.32 \textwidth]{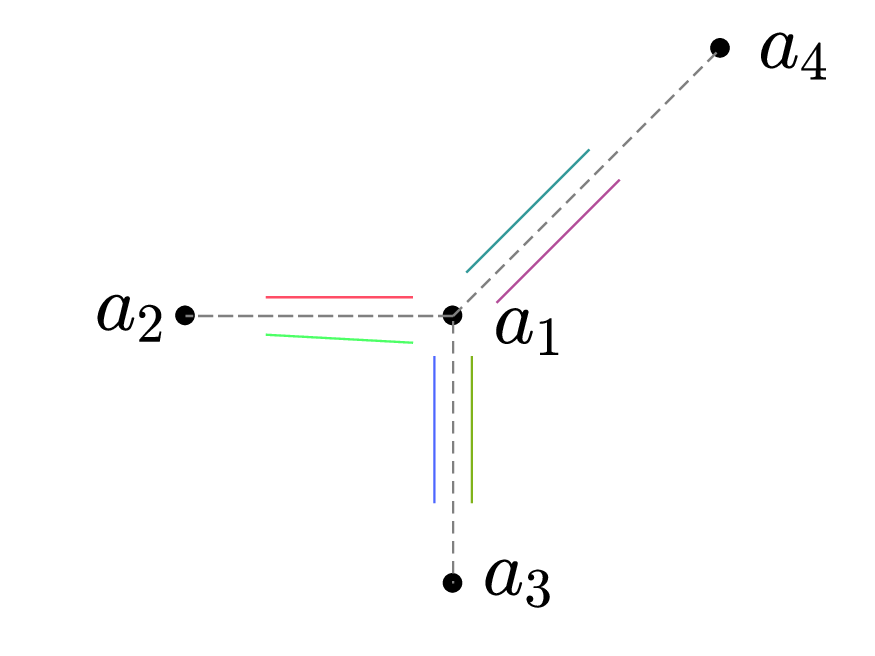}}
  \subfigure[\text{The first interation.}]{\includegraphics[width=.3\textwidth]{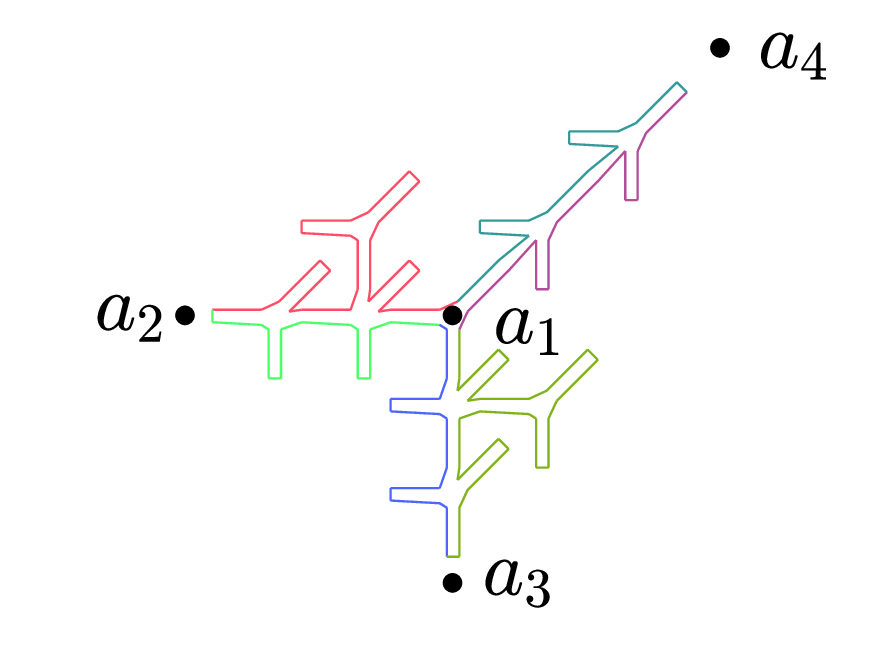}}
  \subfigure[\text{Invariant sets of the GIFS.}]{\includegraphics[width=.35\textwidth]{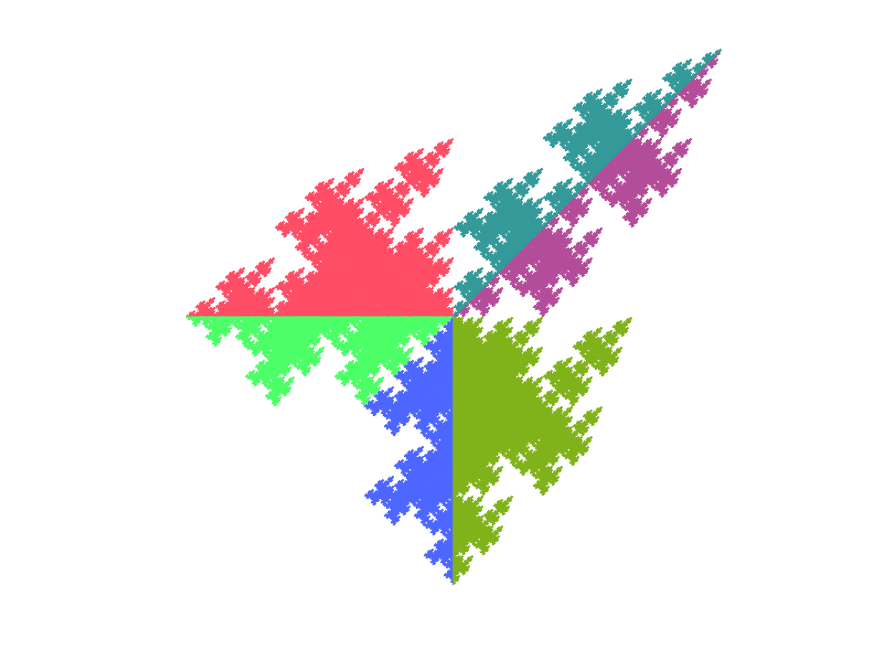}}\\
   %\subfigure[\text{The second interation.}]{\includegraphics[width=.6\textwidth]{Rocket_07}}
  \caption{A space-filling curve of the Rocket tile.}\label{Rocket}
\end{figure}
\end{example}

  %sec: positive

\section{\textbf{Linear GIFS}}\label{sec:linear}

In this section, we recall the definition of the linear GIFS introduced in \cite{RaoZh15}. %For a more careful discussion, we refer to \cite{RaoZh15}.

\subsection{GIFS}
Let $G=(\mathcal{A},\Gamma)$ be a directed graph with vertex set $\mathcal{A}$ and edge set $\Gamma$. Let
 \begin{equation}\label{GIFS}
 {\mathcal G}=\left (g_{\boldsymbol \gamma}:\mathbb{R}^d\rightarrow \mathbb{R}^d \right )_{{\boldsymbol \gamma}\in \Gamma}
 \end{equation}
 be a family of similitudes. We call the triple $(\mathcal{A},\Gamma,{\mathcal G})$, or simply ${\mathcal G}$,  a \emph{graph-directed iterated function system} (GIFS).
 We  call $({\mathcal A}, \Gamma)$ the \emph{base graph} of the GIFS. In what follows, we shall call  ${\mathcal A}$ a \emph{state set} instead of  vertex set, to avoid confusion
 with  other graphs. Very often but not always,  we set ${\mathcal A}$ to be $\{1,\dots, N\}$.

Let $\Gamma_{ij}$ be the set of edges from state $i$ to $j$. It is well known that there exist unique non-empty compact sets $\{E_i\}_{i=1}^N$ satisfying
\begin{equation}\label{uniset}
E_i=\bigcup_{j=1}^N\bigcup_{{\boldsymbol \gamma}\in \Gamma_{ij}}g_{{\boldsymbol \gamma}}(E_j),\quad 1\leq i\leq N.
\end{equation}
We call $\{E_i\}_{i=1}^N$ the \emph{invariant sets} of the GIFS \cite{MW88}.
 The set equations \eqref{uniset} provide an alternative  way to define a GIFS.  We shall call \eqref{uniset} the \emph{set equation form} of  GIFS \eqref{GIFS}.

%%%%%%%%%%%%%%%%%%%%%%%%%%%%%%%%%%%%%symbolic space %%%%%%%%%%%%%%%%%%%%%%%%%%%%%%%
\subsection{Symbolic space related to a graph $G$}
 A sequence of edges in $G$, denoted by $\bomega=\omega_1\omega_2\dots\omega_n$, is called a \emph{walk}, if the
  terminus of $\omega_i$ coincides with the origin of  $\omega_{i+1}$.
  (Here we don't use the notation $\omega_1+\cdots+\omega_n$ for simplicity.)
We will use the following notations to for the sets of finite or infinite walks on $G=(\mathcal{A},\Gamma)$. For $i\in \mathcal{A}$, let
 $$
 \Gamma_i^k,\ \Gamma_i^\ast \text{ and } \Gamma_{i}^{\infty}
 $$
  be the set of
 all walks with length $k$, the set of all walks with finite length, and the set of all infinite walks,
  emanating  from the state $i$, respectively.
Note that $\Gamma_i^\ast=\bigcup_{k\geq 1}\Gamma_i^k$.

For an infinite walk ${\boldsymbol \omega}=(\omega_n)_{n=1}^\infty\in \Gamma_i^\infty$,
we set ${\boldsymbol \omega}|_k=\omega_1\omega_2\dots\omega_k$ and call
$$[\omega_1\dots\omega_n]:=\{{\boldsymbol \gamma}\in \Gamma_i^\infty;~~{\boldsymbol \gamma}|_n=\omega_1\dots\omega_n\}$$
the \emph{cylinder} associated with $\omega_1\dots \omega_n$.
For a walk ${\boldsymbol \gamma}=\gamma_1\dots \gamma_n$, we denote
$$
E_{\boldsymbol \gamma}:=g_{\gamma_1}\circ\cdots \circ g_{\gamma_n}(E_{t(\bgamma)}),
$$
where $t(\bgamma)$ denotes the terminus of the walk $\boldsymbol \gamma$ (and $\gamma_n$). Iterating \eqref{uniset} $k$-times, we obtain
\begin{equation}\label{uni-set3}
E_i=\bigcup_{{\boldsymbol \gamma}\in \Gamma_i^k}E_{\boldsymbol \gamma}.
\end{equation}

We define a projection $\pi: (\Gamma_1^{\infty},\dots,\Gamma_N^{\infty}) \rightarrow (\mathbb{R}^d,\dots,\mathbb{R}^d)$, where $\pi_i: \Gamma_i^{\infty} \rightarrow \mathbb{R}^d$ is defined by
\begin{equation}\label{eq-projection}
\{\pi_i({\boldsymbol \omega})\}:=\bigcap_{n\geq 1}E_{\boldsymbol {\omega|_{n}}}.
\end{equation}
 For $x\in E_i$, we call ${\boldsymbol \omega}$ a coding of $x$ if $\pi_i({\boldsymbol \omega})=x$. It is folklore that $\pi_i(\Gamma_{i}^{\infty})=E_i$.

\subsection{\textbf{Order GIFS and linear GIFS} (\cite{Akiyama10,RaoZh15})}
Let $({\mathcal{A}},\Gamma,\mathcal{G})$ be a GIFS. To study the `advanced' connectivity property of the invariant sets,
 we equip a partial order on the edge set $\Gamma$ enlightened by set equation \eqref{uni-set3}.
Let $\Gamma_i=\Gamma_i^1$ be the set of edges emanating from the state $i$.
\begin{defi}
{\rm

We call the quadruple $({\mathcal{A}},\Gamma, \mathcal{G}, \prec)$ an \emph{ordered GIFS},
if  $\prec$ is a partial order on $\Gamma$ such that
\begin{enumerate}
  \item[($i$)] $\prec$ is a linear order when restricted on $\Gamma_j$ for every $j\in {\mathcal A}$;
  \item[($ii$)]  elements in $\Gamma_i \text{ and } \Gamma_j$ are not comparable if  $i \neq j$.
\end{enumerate}
}
\end{defi}

The order $\prec$ can be extend to $\Gamma_i^n$ and $\Gamma_i^\infty$ for every $i\in{\mathcal A}$.
On $\Gamma_i^n$, two elements
$\gamma_1\gamma_2\dots \gamma_k \prec \omega_1\omega_2\dots \omega_k$
if and only if $ \gamma_1\dots \gamma_{\ell-1} = \omega_1\dots \omega_{\ell-1}$ and $\gamma_\ell\prec \omega_\ell$ for some $1\leq \ell \leq k$.
  Observe that $(\Gamma_i^n, \prec)$ is a linear order. In the same manner, we can obtain a linear order of $\Gamma_i^\infty$, which we still denote by $\prec$.
  In the following, we say $\bomega$ is \emph{lower} than $\bgamma$ if  $\bomega\prec \bgamma$. Two walks $\bomega\prec \bgamma$ in $\Gamma_i^n$ are
  said to be \emph{adjacent} if there is no walk $\bbeta$ such that $\bomega\prec\bbeta\prec\bgamma$.

\begin{defi}{\rm Let $({\mathcal{A}},\Gamma,\mathcal{G},\prec)$ be an ordered GIFS with invariant sets $\{E_i\}_{i=1}^N$.
 It is termed a \emph{linear} GIFS,  if for all $i\in {\mathcal A}$ and $k\geq 1$,
 $$
 E_{\boldsymbol \gamma}\cap E_{\boldsymbol \omega}\neq \emptyset
 $$
  provided ${\boldsymbol \gamma},{\boldsymbol \omega}$ are adjacent walks in
  $ \Gamma_{i}^k$.}
\end{defi}

\subsection{Chain condition} Let $(\mathcal{A},\Gamma,{\cal G}, \prec)$ be an ordered GIFS. %
Fix $i\in \mathcal{A}$.  Let $\bomega$ be an edge with emanating from state $i$. Let
$${\mathbf \delta}(\bomega) \text{ and } {\mathbf \Delta}(\bomega)$$
be the lowest and highest walks in $\Gamma_i^\infty$ initialled by the edge $\bomega$, respectively.

\begin{defi}{\rm
An ordered GIFS is said to satisfy the \emph{chain condition}, if for any $i\in {\mathcal A}$, and any two adjacent edges
$\bomega, \bgamma\in \Gamma_i$ with $\bomega \prec \bgamma$, it holds that
$$ \pi_i({\mathbf \Delta}(\bomega))=\pi_i({\mathbf \delta}(\bgamma)).$$
}
\end{defi}

The chain condition provides a simple criterion for linear GIFS.

\begin{theorem}\label{G_chain} (\cite{RaoZh15})
 An ordered GIFS is a linear GIFS if and only if it satisfies the chain condition.
\end{theorem}

\subsection{Proof of Theorem \ref{lem:zipper}.}
Let $\{S_j\}_{j=1}^N$ be a self-similar zipper with vertices $\{x_0, \dots, x_N\}$ and
signature $\{\beta_1,\dots, \beta_N\}$. We define the following two-state GIFS:
\begin{equation}\label{zip}
\left \{
\begin{array}{l}
E_1=S_1(E_{\beta_1})+\cdots+S_N(E_{\beta_N}),\\
E_{-1}=S_N(E_{-\beta_N})+\cdots+S_1(E_{-\beta_1}).\\
\end{array}
\right .
\end{equation}
Clearly $E_1=E_{-1}=K$ are the invariant sets of the above ordered GIFS, where $K$ is the invariant set of $\{S_j\}_{j=1}^N$.

Let $\pi:\{1,\dots, N\}\to K$ be the projection map associated with the IFS $\{S_j\}_{j=1}^N$.
Let $\pi_1:\Gamma_1\to E_1$ and $\pi_{-1}:\Gamma_{-1}\to E_{-1}$ be the projections associated with $E_1$ and $E_{-1}$ defined as
\eqref{eq-projection}.

Let $\mathbf \delta$ and $\mathbf \Delta$ be the lowest path and highest path emanating from the state $E_1$, respectively.
The form of $\mathbf \delta$ and $\mathbf \Delta$ are completely determined by $\beta_1$ and $\beta_N$.
We claim that
$$\pi_1(\mathbf \delta)=x_0, \ \pi_1(\mathbf \Delta)=x_N.$$

If $\beta_1=1$,  then the lowest edge emanating from $E_1$ is $\omega=(E_1, 1, S_1, E_1)$, where the components of $\omega$ means
the initial state, the order, the contraction map, and the terminal state, respectively;
hence
 $\delta=(E_1, 1, S_1, E_1)^\infty$ and the contraction maps associated with $\delta$  is $(S_1)^\infty$.
 On the other hand, the zipper condition implies that $S_1(x_0)=x_0$, so we obtain
 $$\pi_1(\delta)=\pi(1^\infty)=x_0.$$

If $\beta_N=1$, then the highest edge emanating from $E_1$ is $(E_1,N, S_N, E_1)$, hence
 $\delta=(E_1, N, S_N, E_1)^\infty$ and the contraction maps associated with $\Delta$  is $(S_N)^\infty$.
 On the other hand, the zipper condition implies that $S_N(x_N)=x_N$, so we obtain
 $$\pi_1(\Delta)=\pi(N^\infty)=x_N.$$

 A similar calculation as above shows that,
 $$
 \pi_1(\delta)=\pi(1N^\infty)=x_0, \  \pi_1(\Delta)=\pi(N^\infty)=x_N \quad \text{ if } \beta_1=-1, \beta_N=1;
 $$
  $$
 \pi_1(\delta)=\pi(1^\infty)=x_0, \  \pi_1(\Delta)=\pi(N1^\infty)=x_N \quad \text{ if } \beta_1=1, \beta_N=-1;
 $$
  $$
 \pi_1(\delta)=\pi((1N)^\infty)=x_0, \  \pi_1(\Delta)=\pi((N1)^\infty)=x_N \quad \text{ if } \beta_1=-1, \beta_N=-1.
 $$
 Our claim is proved.

 Let ${\mathbf \delta}'$ and ${\mathbf \Delta}'$ be the lowest path and highest path emanating from the state $E_{-1}$, respectively.
 Similarly, we have  $\pi_{-1}({\mathbf \delta}')=x_N, \ \pi_{-1}({\mathbf \Delta}')=x_0.$

 Finally, notice that the zipper condition is exactly the chain condition, so \eqref{zip} is a linear GIFS.

 The other side is easy to prove.
$\Box$

\begin{remark}{\rm Let $(S_j)_{j=1}^N$ be an ordered IFS. One can easy calculate the four possibilities of
the pair $\{x_0,x_N\}$ in the above proof.  Clearly $(S_j)_{j=1}^N$ is a zipper if and only if
there exists $(\beta_1,\dots, \beta_N)\in \{-1,1\}^N$ such that
$$
S_1(\overrightarrow{x_0x_N}^{\beta_1})+\cdots+S_N(\overrightarrow{x_0x_N}^{\beta_N})
$$
is a broken line.}
\end{remark}
 % sec:linear

\section{\textbf{Induced GIFS}}\label{sec:induce}

In this section, we define the induced GIFS of edge-to-trail substitutions, especially we show that they are linear GIFS.

   Let $\tau: u\mapsto P_u, u\in V$ be an edge-to-trail substitution  defined in Section \ref{sec:rule}.
The trail $P_u$ can be written as
  \begin{equation}\label{eq:induce1}
 P_u= S_{u,1}(v_{u,1})+\dots +S_{u,\ell_u}(v_{u,\ell_u}),
  \end{equation}
  where $S_{u,j}\in \mathcal{S}$ and $v_{u,j}\in V$ for $j=1,\dots,\ell_u$.

    According to   $\tau$     we can construct an ordered GIFS as follows.
Replacing $P_u$ by $E_u$ on the left hand side, and replacing $v$ by $E_v$ on the right hand side of \eqref{eq:induce1},  we obtain an ordered GIFS:
\begin{equation}\label{seteq}
E_u= S_{u,1}(E_{v_{u,1}})+ \cdots +S_{u,\ell_u}(E_{v_{u,\ell_u}}), \quad u\in V,
\end{equation}
 which we     the \emph{induced GIFS} of $\tau$.  In an ordered GIFS,  we use $``+"$ to replace the $``\cup"$ in  the set equation to emphasis the order structure.

\begin{example} \textbf{Induced GIFS of Terdragon.} {\rm The edge-to-trail substitution $\tau$  in Example \ref{Terdragon-1} induces the following GIFS:
$$
\left \{
\begin{array}{l}
E_1=  S_1(E_1)+S_2(E_3)+S_2(E_1),\\
E_2=  S_2(E_2)+S_3(E_1)+S_3(E_2), \\
E_3=  S_3(E_3)+S_1(E_2)+S_1(E_3).
\end{array}
\right .
$$
The invariant sets $E_j$ are illustrated in Figure \ref{fig:Tdragon:2}(c).
}
\end{example}

Let
\begin{equation}\label{Graph-form}
(V, \Gamma, \mathcal {G}, \prec)
\end{equation}
 be the basic-graph form of the  ordered GIFS \eqref{seteq}.
 The state set is $V$, which is the edges of the initial graph.
The edge set $\Gamma$ consists of quadruples $(u, S_j, v, k)$, that is,
 if $S_j(v)$ is the $k$-th edge in the trail $P_u$, then we add an edge to $\Gamma$ and denote this edge by
 \begin{equation}\label{eq:edge}
 (u, S_j, v,k)\in  \Gamma.
 \end{equation}
 The contraction associated with this edge is $S_j$.

\begin{theorem}\label{thm:induce}
The induced  GIFS   \eqref{seteq} is a linear GIFS.
\end{theorem}

\begin{proof}
Let $u\in V$.  We denote by  $a_u$ and $b_u$ the origin and the terminus of $u$ as an edge in the
initial graph $\Lambda=(A,V)$. We claim that
the lowest and highest elements in $\Gamma_u^\infty$  are  codings of $a_u$ and $b_u$, respectively.

 Let $S(v)$ be the first edge in $P_u$, then
$\omega=(u,S,v,1)$  is the lowest edge emanating from $u$ in the basic graph $\Gamma$.
It follows that
\begin{equation}\label{comm_point}
a_u=S(a_v).
\end{equation}
Therefore,  if $(\omega_n)_{n=1}^\infty$ is a coding of $a_v$, then%, set $\omega=(u,S,v,1)$,
$$\omega (\omega_n)_{n-1}^\infty$$
is a coding of $a_u$. Applying the same argument to $v$, we obtain a coding of $a_u$,
such that the first two edges of this coding is the lowest walk in $\Gamma_u^2$.
Continuing this argument, we conclude the lowest element in $\Gamma_u^\infty$  is a coding of $a_u$.

Similarly, the highest element in $\Gamma_u^\infty$  is a coding of $b_u$.

Now, let
$\omega=(u, S, v, k)$ and $\gamma=(u, T, v',k+1)$ be two consecutive edges in $\Gamma_u$.
This means that $S(v)$ and $T(v')$ are two adjacent edges in $P_u$,
so $S(b_v)=T(a_{v'})$.

On the other hand, write $\Delta(\omega)=\omega(\omega_n)_{n\geq 1}$,
then $(\omega_n)_{n\geq 1}$ is the highest coding in $\Gamma_v$.
So $\pi_v((\omega_n)_{n\geq q})=b_v$ by the claim above, and
$$
\pi_u(\Delta(\omega))=S\circ\pi_v((\omega_n)_{n\geq q})=S(b_v).
$$
Similarly, we have $\pi_u(\delta(\gamma))=T(a_{v'})$.
This  verifies  the chain condition. Therefore, the ordered IFS in consideration  is linear.
\end{proof}

\noindent \textbf{Proof of Theorem \ref{thm:travel}.} Let $\tau$ be an edge-to-trail substitution over
$V$ such that $\tau(u)$ are all traveling trails. Let us denote the induced GIFS of $\tau$ by $(V,\Gamma, {\mathcal G},\prec)$. Let $g_\omega$ be the contraction associated with the edge $\omega$
in $\Gamma$.

 Fix an $u\in V$. The traveling property of $\tau$  implies that
the associated contractions with edges in $\Gamma_u^n$ are exactly the maps $S_I, {I\in \{1,\dots,N\}^n}$, and each maps appears only once. It follows that
$$
\bigcup_{\bomega\in \Gamma_u^n} g_{\bomega}(\{0\})=\bigcup_{I\in \{1,\dots,N\}^n}S_I(\{0\}),
$$
where $g_\bomega=g_{\omega_1}\circ \cdots \circ g_{\omega_n}$ is the contraction associated to the trail $\bomega$.
Taking the limit in Hausdorff metric at both sides, we obtain $E_u=K$.
This proves that the invariant sets of the GIFS are all equal to $K$.

Moveover, if we ignor the order structure, then the induced GIFS
degenerates to the original IFS of $K$, so the induced GIFS satisfies the OSC.
Hence, $K$ admits optimal parameterizations according to Theorem \ref{thm:RaoZh}.
$\Box$

 % sec:induce

\section{\textbf{(General) Euler-tour method}}\label{sec:euler}

In this section, we introduce the general Euler-tour method.
Recall that $A=\{a_1,\dots,a_m\}$ is a skeleton, and $\Lambda=(A, V^+)$
is the cycle passing all elements of $A$, see Section 5.
For simplicity, we  denote
$$v_j=\overrightarrow{a_ja_{j+1}} \text{ and } v_j^{-1}=\overrightarrow{a_{j+1}a_j}, \quad j=1,\dots,m,$$
where we identify $a_{m+1}$ with $a_1$.
Then
$
\Lambda= v_1+v_2+\cdots+v_m,
$
and we denote
$
\Lambda^{-1}=v_m^{-1}+\cdots +v_1^{-1}
$
to be  the reverse  of $\Lambda$.

Let $\beta=(\beta_1,\beta_2,\cdots,\beta_N)\in \{1,-1\}^N$, and we  call  it an \emph{orientation vector}.
Denote
\begin{equation}\label{induced-G}
G(\mathcal{S},A,\beta)= \bigcup_{j=1}^N S_j(\Lambda^{\beta_j}).
\end{equation}
By the same argument as Lemma \ref{lem:tour}, one can show that
 $G(\mathcal{S},A,\beta)$ admits Euler tour.

\begin{defi}{\rm
 Let $P$ be an Euler tour of $G(\cS, A, \beta)$, and let
 $$P=P_1+\cdots+P_m$$
 be a partition of $P$. We call it an \emph{Euler-partition} of $G(\cS, A_\tau, \beta)$ if
 the initial points of $P_1,\dots,P_m$, denoted by $a_{\tau'(1)},\dots, a_{\tau'(m)}$, is a permutation of $A$.
 We call $\tau'$ the \emph{output} permutation of $P_1+\cdots +P_m$.
 }
\end{defi}

%\begin{figure}[h]
%  \centering
%  % Requires \usepackage{graphicx}
%  \includegraphics[width=5 cm]{Gasket_Euler}
%  \caption{The trails in different colors give an Euler partition of refined graph in Figure \ref{Exam_refined}.}\label{EulerParti}
%\end{figure}

 \begin{figure}[h]
  \centering
  \subfigure[$P=P_1+P_2+P_3$]{\includegraphics[height=3.5 cm]{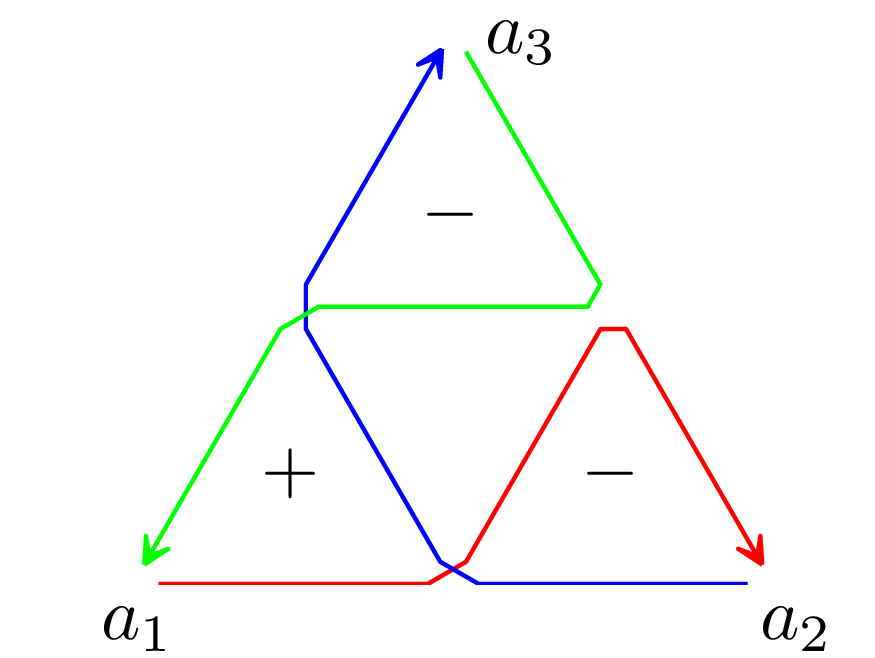}}
   \subfigure[$P^{-1}=P_3^{-1}+P_2^{-1}+P_3^{-1}$]{\includegraphics[height=3.5 cm]{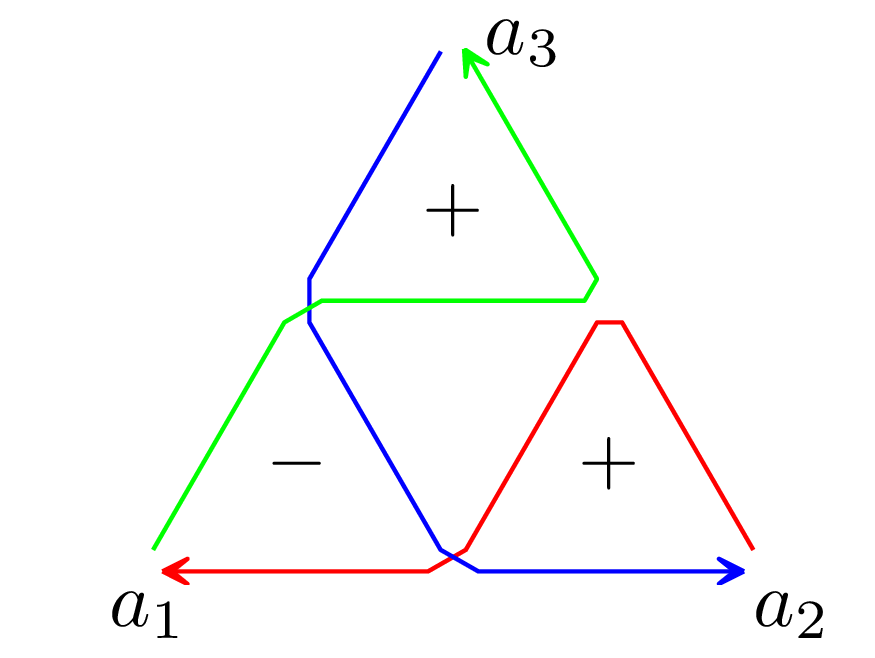}}
   \subfigure[$P^2=P\circ P$]{\includegraphics[width=4.5 cm]{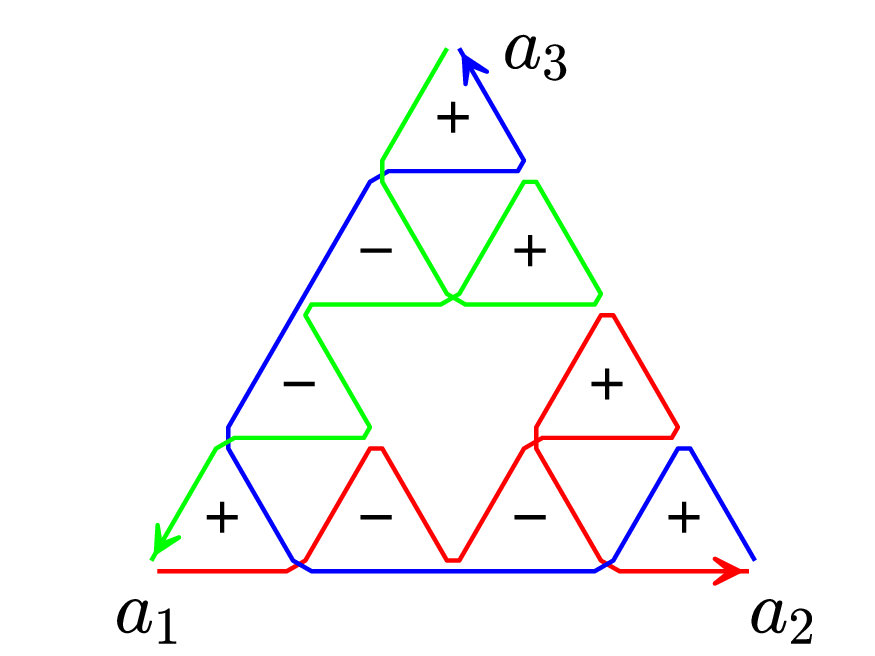}}
  \caption{ $G(\mathcal{S},A,\beta)$ and $G(\mathcal{S},A,-\beta)$.
 }\label{fig:Gasket}
\end{figure}

%We say two paths in a graph  are disjoint if they have no common edge.

\subsection{Consistency and induced edge-to-trail substitution}

\begin{defi}{\rm
We say an Euler-partition $P_1+\cdots +P_m$ of $G(\mathcal{S}, A, \beta)$ is  \emph{consistent}, if $P_i$ has the same origin and terminus as $v_i$ for  each $i\in\{1,2,\dots,h\}$,
or equivalently,   its output
permutation is the identity.
  }
\end{defi}

As soon as $P_1+\cdots+P_m$ is  consistent,
 we can define an  edge-to-trail substitution  over
%\begin{equation}\label{eq:V}
$$
 V=V^+\cup V^{-}=\{v_1,\dots, v_m\}\cup\{v_1^{-1},\dots, v_m^{-1}\}
 $$
%\end{equation}
as
\begin{equation}\label{eq:path-sub}
\tau: v_j \mapsto P_j \text{ and } v_j^{-1}  \mapsto P_j^{-1}, \quad j=1,\dots, m.
\end{equation}
We call $\tau$ the \emph{induced edge-to-trail substitution} of the Euler-partition $P_1+\cdots+P_m$.

\begin{remark}{\rm
Set $G=\bigcup_{i=1}^N S_i(\Lambda\cup \Lambda^{-1})=G(\mathcal{S},A,\beta) \cup G(\mathcal{S},A,-\beta)$, and we see that  the generalized
Euler-tour method is  a special case of the construction in Section \ref{sec:rule}.
}
\end{remark}

Denote the length of  $P_j$  by $\ell_j$, then  $P_j$  has the form
\begin{equation}\label{P_j}
P_j=S_{j,1}(u_{j,1})+S_{j,2}(u_{j,2})+\cdots +S_{j,\ell_{j}}(u_{j,\ell_{j}}), \quad j=1,\dots, m.
\end{equation}
where  $S_{j,k}\in \{S_1,\dots, S_N\}$, and
 $u_{j,k}\in V$.
Accordingly,
\begin{equation}\label{P_j^{-1}}
P_j^{-1}=S_{j,\ell_j}(u^{-1}_{j,\ell_j})+\cdots+ S_{j,2}(u^{-1}_{j,2})+S_{j,1}(u^{-1}_{j,1}),\quad j=1,\dots, m.
\end{equation}

The substitution \eqref{eq:path-sub} induces the following  linear GIFS (by Theorem
\ref{thm:induce} ):
\begin{equation}\label{setequation}
\left \{
\begin{array}{c}
E_{v_j}=S_{j,1}(E_{u_{j,1}})+ S_{j,2}(E_{u_{j,2}})+ \cdots + S_{j,\ell_j}(E_{u_{j,\ell_j}}), \\
E_{v_j^{-1}}=  S_{j,\ell_j}(E_{u^{-1}_{j,\ell_j}})+ \cdots
+ S_{j,2}(E_{u^{-1}_{j,2}})+ S_{j,1}(E_{u^{-1}_{j,1}}).
\end{array}
\right .
\quad 1\leq j \leq m.
\end{equation}
For simplicity, we   call GIFS \eqref{setequation}
  the \emph{induced GIFS} of the Euler-partition $P_1+\cdots+P_m$.
  We shall use $(V,{\mathcal E}, {\cal G},\prec)$ to denote this GIFS,  as we explained in Section \ref{sec:induce}.

Now we give some basic properties of the induced GIFS.

\begin{lemma}\label{lem:basic}
 (i) $E_{v}=E_{v^{-1}}$ for $v\in V$.

(ii)  For  each   $v\in V$ and $i\in\{1,\dots,N\}$,  $S_i(E_v)$ appears exactly once in the right-hand side of
  \eqref{setequation}.

(iii) $K=\bigcup_{v\in V^+} E_{v}$.
\end{lemma}
\begin{proof}
(i) Set $F_{v}=E_{v^{-1}}$ for $v\in V$. Clearly $\{F_v\}_{v\in V}$ also satisfy the set equation \eqref{setequation} if we ignore the order structure. Hence (i) follows from the uniqueness of the invariant sets.

(ii) Since   $P$ is an Euler tour, for each   $v\in V$ and $i\in\{1,\dots,N\}$, the edge $S_i(v)$ appears exactly once in $P\cup P^{-1}$, which implies (ii).

(iii) Taking the union of both sides of \eqref{setequation} and using the conclusion of (ii), we have
\begin{equation}\label{uni-set2}
\bigcup_{v\in V}E_v=\bigcup_{j=1}^m \bigcup_{k=1}^{l_j}S_{j,k}(E_{u_{j,k}}\cup E_{u_{j,k}^{-1}})=\bigcup_{i=1}^NS_i(\cup_{v\in V}E_{v}).
\end{equation}
Therefore, by the uniqueness of the invariant set of an IFS, we obtain
$$K=\bigcup_{v\in V} E_v=\bigcup_{v\in V^+} E_{v},$$
where the last equality holds since $E_{v}=E_{v^{-1}}$.
\end{proof}

  Lemma \ref{lem:basic} (ii) means that
for each $v\in V$, there are exactly $N$ edges in $\Gamma$ terminating at $v$; moreover,  the associated maps of these edges run over the maps $S_1,\dots, S_N$.

\subsection{Primitivity and the open set condition}
To ensure the induced GIFS satisfies the OSC, we need a primitivity condition.

{Since $E_v=E_{v^{-1}}$,  in the induced GIFS, we can identify $E_v$ and $E_{v^{-1}}$ to simplify the discussion.}
By identifying $v$ and $v^{-1}$ in $\tau$, and ignoring the maps $S_j$ in $\tau(v)$, we define  a morphism over $\{v_1,\dots, v_m\}$   by
\begin{equation}
\tau^*:~~v_j \longrightarrow  |v_{j,1}|~|v_{j,2}|~\dots ~|v_{j,l_j}|,
\quad  j=1,2,\dots,m,
\end{equation}
where  $|v|=|v^{-1}|=v$ for $v\in V^+$.
We call  $\tau^*$ the \emph{induced morphism}.
Recall that $\tau^*$  is \emph{primitive} if there exists an integer $n$ such that for any $i,j\in \{1,2,\dots, m\}$, $v_i$ appears in $(\tau^*)^n(v_j)$;
see for instance \cite{Queff87}.

\begin{defi}{\rm We say a consistent  Euler-partition $P_1+\cdots +P_m$ of $G({\mathcal S},A, \beta)$ is \emph{primitive}, if the
 induced morphism $\tau^*$ is primitive.
 }\end{defi}

\begin{example}{\rm Consider the Sierpi\'nski gasket.
Let $\beta=(1,-1,-1)$.
  Figure \ref{fig:Gasket}(a) illustrates a consistent partition $P_1+P_2+P_3$ , where the trails are indicated by colors.
  The induced edge-to-trail substitution is:
\begin{equation}\label{Path}
\tau: \left \{
\begin{array}{l}
v_1\mapsto P_1= S_1(v_1)+S_2(v_3^{-1})+S_2(v_2^{-1}),\\
v_2\mapsto P_2= S_2(v_1^{-1})+S_1(v_2)+S_3(v_3^{-1}),\\
v_3\mapsto P_3= S_3(v_2^{-1})+S_3(v_1^{-1})+S_1(v_3).\\
\end{array}
\right .
\end{equation}
The induced morphism  is
$$\tau^*: \left \{
\begin{array}{cl}
v_1 & \mapsto   v_1 v_3 v_2,\\
v_2 & \mapsto   v_1 v_2 v_3, \\
v_3 & \mapsto  v_2 v_1 v_3
 \end{array}
 \right .
 $$
 and it is primitive. Figure \ref{fig:Gasket} illustrates $\tau^3(\Lambda)$.
}

\end{example}

%Next, we prove that the induced GIFS   satisfies  the OSC.

For a GIFS $(V, \Gamma, {\cal F})$, let $r_e$ be the contraction ratio of the mapping associated with the edge $e$.
For $s>0$, the Mauldin-Williams matrix  $M(s)$ is defined to be  (\cite{MW88})
\begin{equation}\label{M^s}
\left(
  \begin{array}{c}
    \sum_{e\in \Gamma_{uv}}r_e^s \\
\end{array}
\right)_{u,v\in V}.
\end{equation}
The real number $s$ satisfying $\Phi(M(s))=1$ is called the \emph{similarity dimension}
of the GIFS, where $\Phi(\cdot)$ denotes the spectral radius of a matrix.

Let ${\mathcal H}^s$ be the $s$-dimensional Hausdorff measure; a set $E$ is called an $s$-set,
if $0<{\mathcal H}^s(E)<\infty$.
A directed graph is said to be \emph{strongly connected}, if for any pair of vertice $i$ and $j$,
there is trail from $i$ to $j$.
The following criterion of the OSC is proved in \cite{MW88} (the `only if' part) and \cite{LiWX94} (the `if' part).

\begin{lemma}\label{MWLi} Let ${\cal F}$ be a GIFS  with a strongly connected base graph.
 Denote its invariant sets by $\{E_j\}_{j=1}^N$, and  the self-similar dimension by $\delta$.
Then ${\cal F}$ satisfies the open set condition if and only if
$$
0<{\mathcal H}^\delta(E_j)<+\infty
$$
for some $1\leq j\leq N$ (or for all $1\leq j\leq N$).
\end{lemma}

\begin{theorem}\label{OSC01}
Let $P_1+\cdots+P_m$ be a consistent and primitive Euler-partition of  $G(S, A, \beta)$.
Then the induced  GIFS \eqref{setequation}   satisfies the OSC, and
$$0<{\cal H}^s(E_v)<\infty, \text{ for all }v\in V,$$
 where $s=dim_H K$.
\end{theorem}

\begin{proof} First, we
show that the similarity dimension of the induced GIFS is $s=\dim_H K$.
 Let $c_j$ be the contraction ratio of $S_j$.
 Then $s$ fulfills the equation $\sum_{j=1}^Nc_j^s=1$ and $0<{\cal H}^s(K)<\infty$, since ${\mathcal S}$ satisfies the OSC. Recall that
\begin{equation}\label{M^s}
M(s)=\left(
  \begin{array}{c}
  \sum_{e\in \Gamma_{uv}}r_e^s \\
\end{array}
\right)_{u,v\in V}.
\end{equation}
By Lemma \ref{lem:basic}(ii), we have (since $\cup_{u\in V}\Gamma_{uv}$ is the set of edges with terminus $v$)
$$
\sum_{u\in V}\sum_{e\in \Gamma_{uv}}r_e^s=\sum_{i=1}^Nc_i^s=1,\quad  v\in V,
$$
namely, the sum of entries of each collum  of $M(s)$ is $1$.
By Perron-Frobenius Theorem (see Lemma \ref{perron_frob} below), the spectral radius of $M(s)$ equals $1$, which implies that  $s$ is the similarity dimension of the induced GIFS.

We identify $E_{v}$ and $E_{v^{-1}}$ in the induced GIFS \eqref{setequation} and forget the order, then we
obtain a \emph{simplified GIFS} as follows:
\begin{equation}\label{simple-gifs}
E_{v_j}=S_{j,1}(E_{|u_{j,1}|})\cup S_{j,2}(E_{|u_{j,2}|})\cup \cdots \cup S_{j,\ell_j}(E_{|u_{j,\ell_j}|}),\quad j=1,\dots, m.
\end{equation}
We have the following facts concerning the simplified GIFS:

$(i)$  The base graph of the simplified GIFS is strongly connected, since the
induced morphism $\tau^*$ is primitive.

$(ii)$  The relation $\mathcal{H}^s(E_{v_j})>0$ holds for at least one $j$, since the union $K=\bigcup_{j=1}^m E_{v_j}$ is has a finite and positive $s$-dimensional Hausdorff measure.

$(iii)$ The simplified GIFS has the same similarity dimension as the induced GIFS, which is $s$.
This is true because the sum of each row of the associated matrix $\tilde M(s)$ of the simplified GIFS is still $1$, where
$$\tilde M(s)=
\left(
  \begin{array}{c}
  \sum_{e\in \Gamma_{uv}}r_e^s +\sum_{e\in \Gamma_{uv^{-1}}}r_e^s\\
\end{array}
\right)_{u,v\in V^+}.
$$
Item (i)-(iii) verify the conditions of  Lemma \ref{MWLi},  hence
 the simplified GIFS satisfies the OSC, which implies
$0<{\mathcal H}^s(E_{v})<\infty$ for all $v\in V^+$ (again by Lemma \ref{MWLi}).

Finally, we observe that if the simplified GIFS satisfies the OSC with open sets $U_1,\dots, U_m$, then
 the induced GIFS  satisfies the OSC with the open sets $U_1,\dots, U_m, U_1,\dots, U_m$.
\end{proof}

The following is a part of the Perron-Frobenius Theorem, see for example \cite{Newman72}.

\begin{lemma} (Perron-Frobenius Theorem) \label{perron_frob}
Let $B=[a_{ij}]$ be a non-negative $k\times k$ matrix.

(i) There is a non-negative eigenvalue $\lambda$ such that it is the spectral radius $B$.

(ii) We have $\displaystyle \min_i\left (\sum_{j=1}^k a_{ij} \right )\leq \lambda \leq \max_i\left (\sum_{j=1}^k a_{ij} \right )$.
\end{lemma}

\subsection{Pure-cell property and the disjointness of $K=\bigcup_{v\in V^+} E_{v}$}
Finally, we investigate when $K=\bigcup_{v\in V^+} E_{v}$ is a disjoint union in Hausdorff measure.
%To this end, we introduce the notion of pure cell.
For $I\in \Sigma^n$, we call the set of edges
$$\{S_I(v); v\in V\} \text{ and } \{S_I(v); v\in V^{-1}\})$$
 a positive and a negative \emph{$S_I$-cell}, respectively.

\begin{defi}\label{def:pure}{\rm Let $\tau$ be an edge-to-trail substitution induced by an
Euler-partition $P_1+\cdots+P_m$.
Let $I\subset \Sigma^n$.
The cell $S_I(\Lambda^{\pm 1})$ is called a \emph{pure cell}, if
there exists $u\in V$ such that   the $S_I$-cell is a subgraph of $\tau^{n}(u)$.
In this case,
 we also say
the partition $P_1+\cdots+P_m$ \emph{potentially contains  pure cell}.
}
\end{defi}

%\begin{theorem}\label{pure cell} Let $P_1+\cdots+P_h$ be a consistent and primitive Eulerian partition of  $G(S, A, \beta)$,
%and the partition potentially contains  pure cells.
% Then $K=\bigcup_{v\in \Lambda_0} E_{v}$ is a disjoint union in the $s$-dimensional Hausdorff measure, where $s=\dim_H K$.
%\end{theorem}

\begin{theorem}\label{thm:old}
Let $\cS=\{S_i\}_{1\leq i \leq N}$ be an IFS satisfying the OSC, and
let $A$ be  a  skeleton of $\cS$.
If there exist a vector $\beta\in \{-1,1\}^N$
  and an Euler-partition  $P_1+\cdots +P_m$ of  $G(\cS, A,\beta)$ such that
the partition is consistent, primitive and potentially contains pure cells,
then

(i) $K=\bigcup_{v\in V^+} E_{v}$ is a disjoint union in the $s$-dimensional Hausdorff measure, where $s=\dim_H K$;

(ii) the edge-to-trail substitution  $\tau$ in \eqref{eq:path-sub} leads to a   a space-filling curve of $K$.
\end{theorem}

\begin{proof}  (i)
Suppose that an $S_I$-cell is a pure cell, \emph{i.e.}, all the edges of the $S_I$-cell belong to $\tau^{n}(u)$
 for some $u\in V$, where $n=|I|$.
 Without loss of generality, we may assume that the orientation of the $S_I$-cell is positive.
 Then all $S_I(E_{v}), v\in V^+$ appear in
 the right hand side of the $n$-th iteration of the set equation \eqref{setequation} corresponding to $E_{u}$.
 Hence,  $\bigcup_{v\in V^+} S_I(E_{v})$ is a  disjoint union in the $s$-dimensional   Hausdorff measure, since the
 induced GIFS satisfies the OSC (by Theorem \ref{OSC01});
  it follows that its image under $S_I^{-1}$, $\bigcup_{v\in V^+} E_{v}$ is also a  disjoint union  in Hausdorff measure.

(ii) By Theorem \ref{thm:induce}, the induced GIFS
is a linear GIFS.
By Theorem \ref{OSC01}, the induced GIFS
satisfies the OSC, and all the invariant sets are $s$-sets.

Let  $L_j={\mathcal H}^s(E_{v_j})$, and  $L=\sum_{j=1}^m L_j$.
By Theorem \ref{thm:RaoZh},
for each $j\in \{1,\dots,m\}$, there exists an optimal parametrization
$\varphi_j: [0,L_j]\mapsto E_{v_j}$
such that $\varphi_j(0)$ and $\varphi_j(L_j)$ are the origins and terminus
of $v_j$, respectively.

 Let $\varphi: [0, L]\to K$ be the curve obtained by
joining all the $\varphi_i$ one by one. Since $K=\bigcup_{v\in V^+} E_{v}$ is a disjoint union
in measure, we conclude that $\varphi$ is an optimal parametrization of $K$.
\end{proof}

\begin{figure}[h]
\subfigure[Initial trail: positive]{\includegraphics[width=.25 \textwidth]{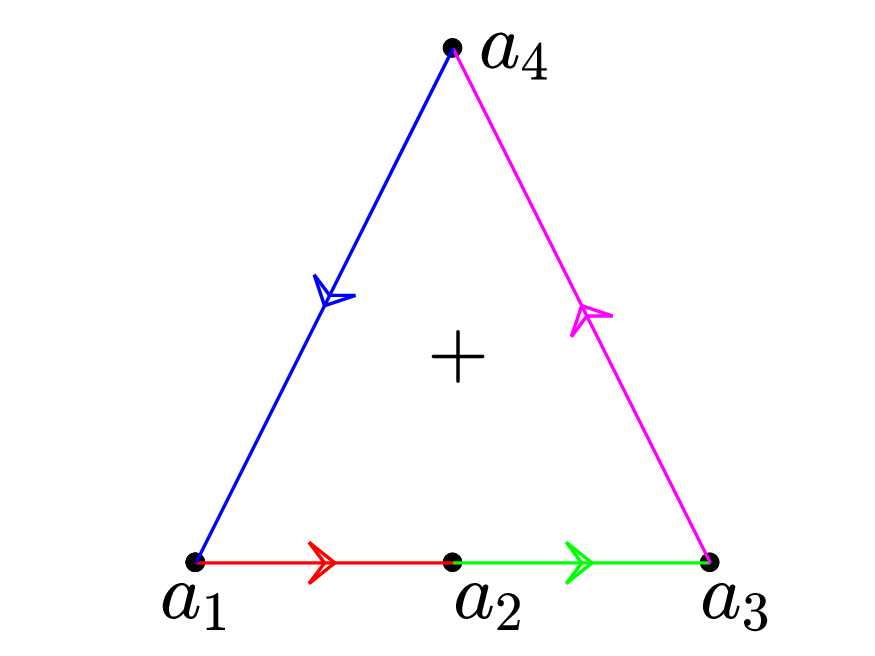}}
\subfigure[Initial trail: negative]{\includegraphics[width=.25 \textwidth]{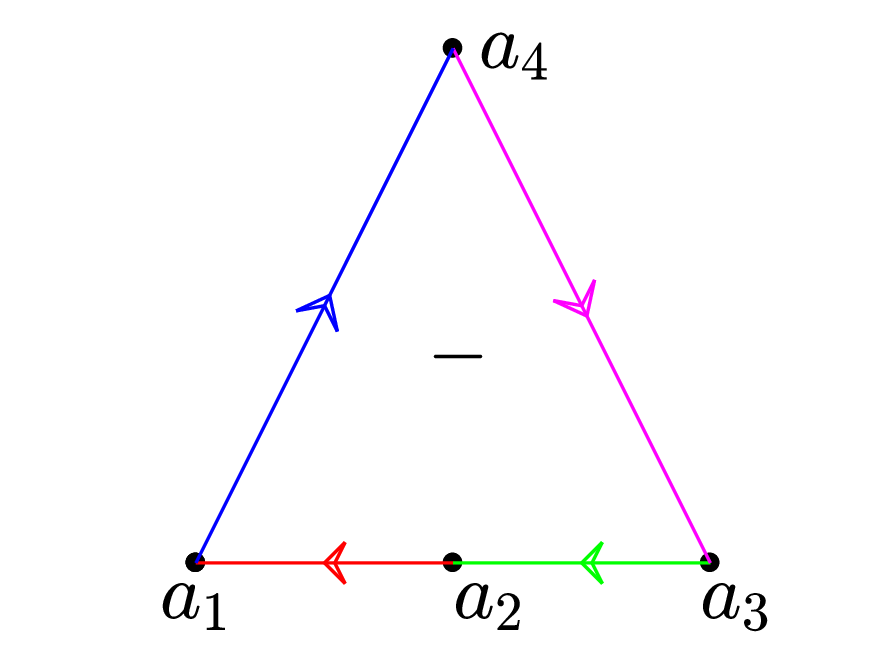}}
\subfigure[Euler tour]{\includegraphics[width=.25 \textwidth]{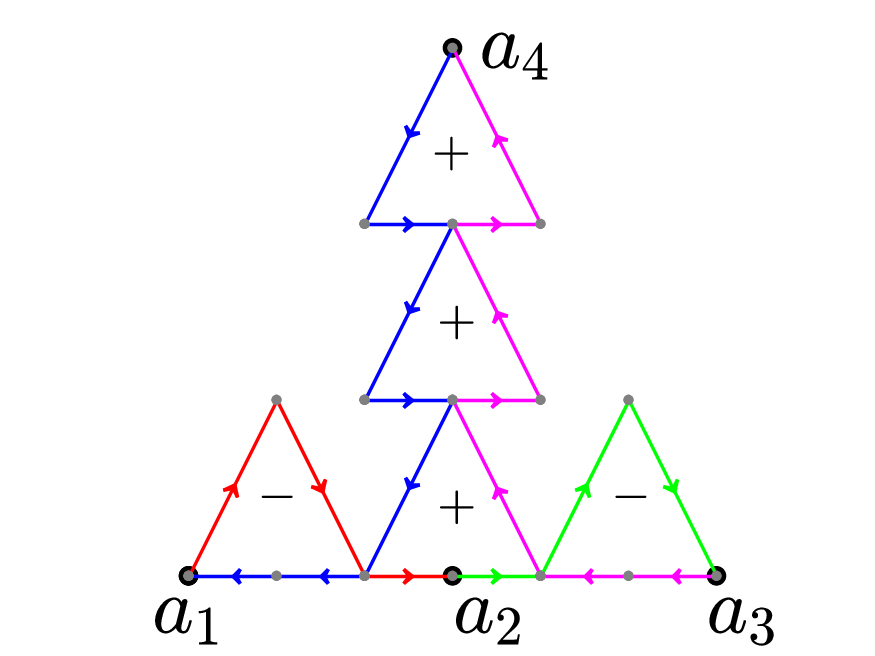}}\\
\subfigure[Initial pattern]{\includegraphics[width=.28 \textwidth]{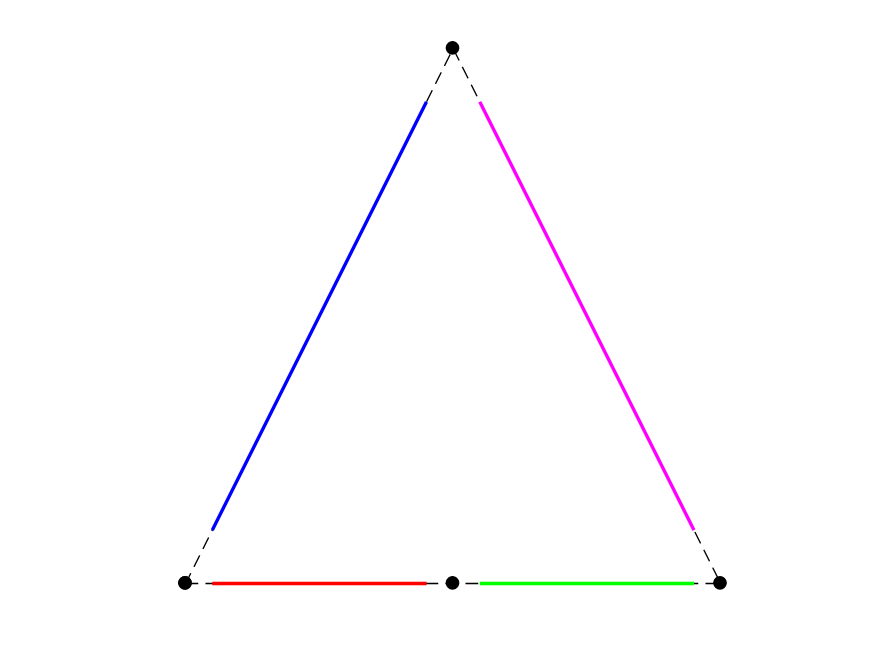}}\subfigure[First generation]{\includegraphics[width=.28 \textwidth]{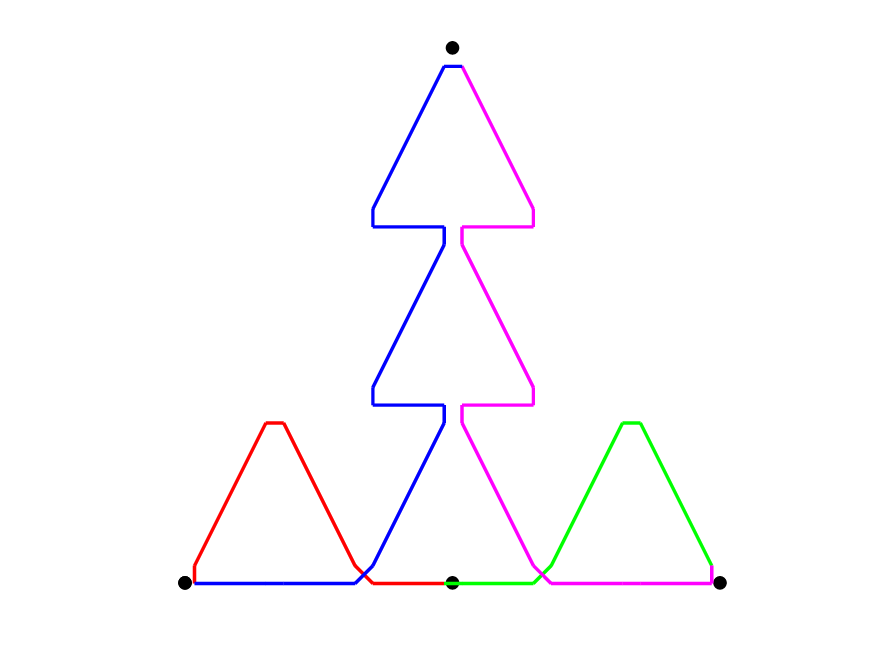}}
\subfigure[Second generation]{\includegraphics[width=.28 \textwidth]{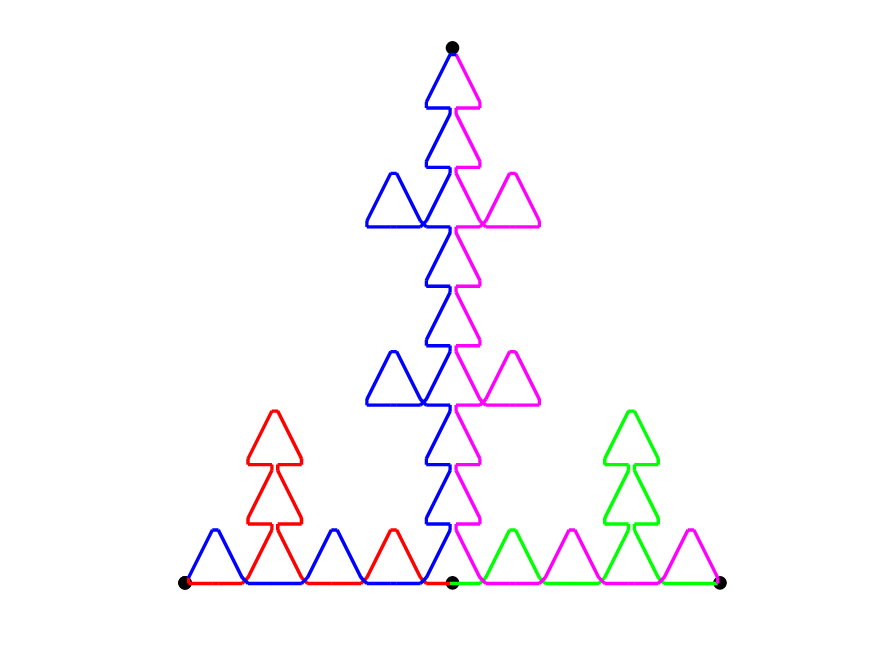}}\\
\subfigure[Third generation]{\includegraphics[width=.35 \textwidth]{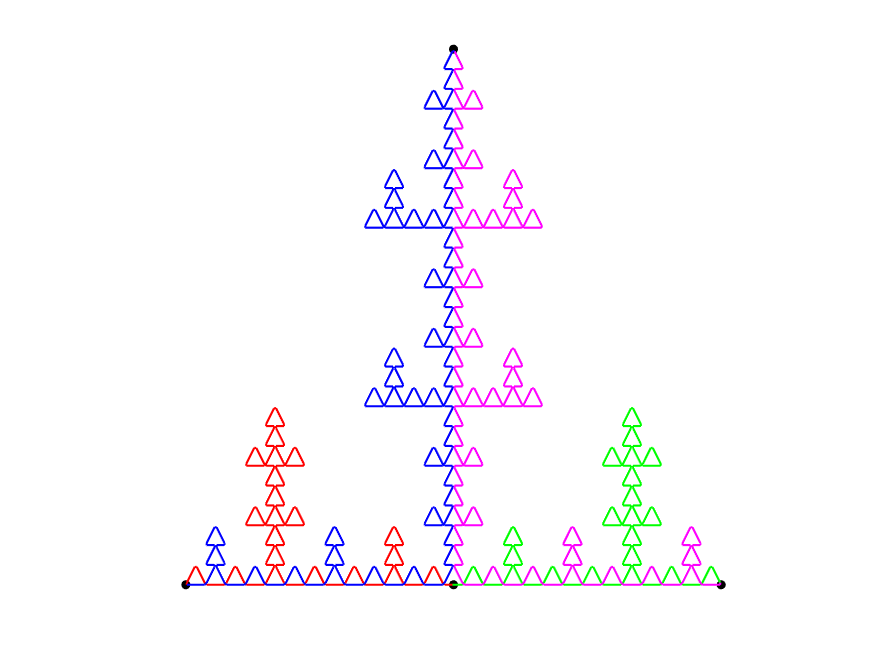}}
\subfigure[Invariant sets of the linear GIFS]{\includegraphics[width=.35 \textwidth]{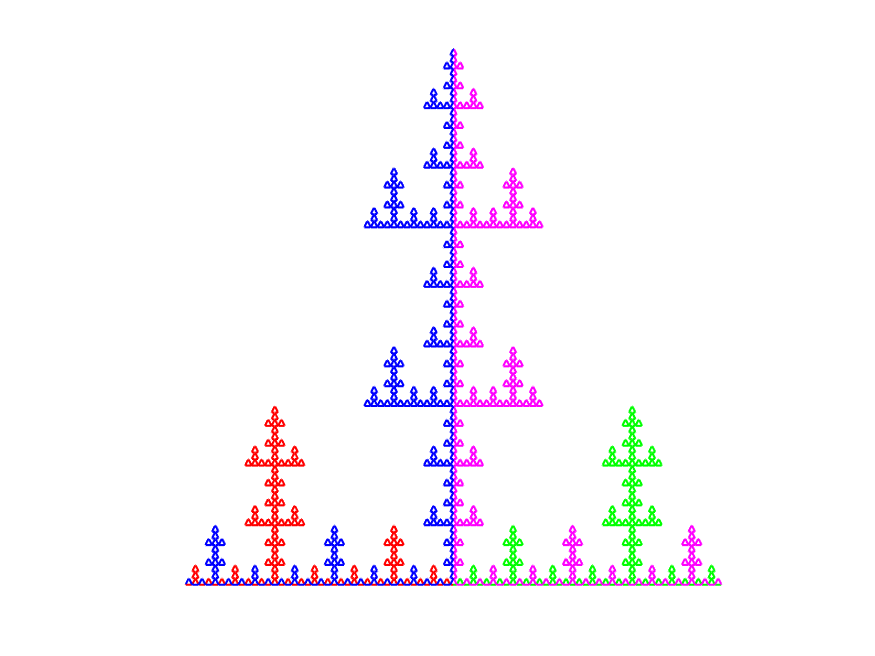}}
\caption{Christmas tree}\label{Fig_Christree}
\end{figure}

\subsection{The Christmas tree}\label{Ex-Christ}
%\textbf {The Christmas tree.}
 {\rm  The Christmas tree is the invariant set of the   IFS $\{S_j\}_{j=1}^5=\{(x+d)/3; d\in D \}$, where $D=\{0, 1, 2,1+\mi,1+2\mi\}$.
We choose $A=\{a_1, a_2, a_3, a_4\}=\{0, 1/2, 1, 1/2+i\}$, which are the fixed points of $S_1, S_2, S_3, S_5$.  Let $v_i=\overrightarrow{a_ia_{i+1}}$ (assume $a_5=a_1$) and let
$\Lambda=v_1+v_2+v_3+v_4.$

%be the initial closed trail, and $\Lambda_0^{-1}=v_4^{-1}+v_3^{-1}+v_2^{-1}+v_1^{-1}$ is the inverse trail of $V$.
%Let $\beta\in\{1, -1\}^5$, then the refined graph is
%$\displaystyle G=\cup_{i=1}^5S_i(\Lambda_0^{\beta_i}).$
To get a primitive substitution, we choose an orientation vector $\beta=(-1,1,-1,1,1)$.
Using the Euler-partition in Figure \ref{Fig_Christree} (e),  we get the following edge-to-trail
substitution
\begin{equation}
\left \{
\begin{array}{l}
v_1\mapsto S_1(v_4^{-1})+S_1(v_3^{-1})+S_2(v_1),\\
v_2\mapsto S_2(v_2)+S_3(v_4^{-1})+S_3(v_3^{-1}),\\
v_3 \mapsto S_3(v_2^{-1})+S_3(v_1^{-1})+S_2(v_3)+S_4(v_2)+S_4(v_3)+S_5(v_2)+S_5(v_3),\\
v_4 \mapsto S_5(v_4)+S_5(v_1)+S_4(v_4)+S_4(v_1)+S_2(v_4)+S_1(v_2^{-1})+S_1(v_1^{-1}).
\end{array}
\right .
\end{equation}
and the rules for $v_1^{-1}, v_2^{-1}, v_3^{-1}, v_4^{-1}$ can be obtained accordingly.

}
%\end{example}

 % sec:euler
     % \input{Example_Pinwheel}

%\input{Type_V1} % sec:type
%\input{Skeleton_V1} % sec:skeleton

%\input{Optimal_V1} % sec:optimal

%\input{Intro_V2}
%\input{Similar-Graph_V1}

\section{\textbf{Consistency}}\label{sec:consistency}

In this section and next section, we will confirm the conditions in Theorem \ref{thm:old}.
In this section, we study the existence of consistent Euler-partitions.

Let $\mathcal{S}=\{S_j\}_{j=1}^{N}$ be an IFS   possessing a  skeleton
$A=\{a_1,a_2,\dots,a_m\}$. Let $K$ be the invariant set of $\cS$.
Recall that
$
\Sigma=\{1,\dots,N\}.
$

 Let $\tau$ be a permutation of $\{1,\dots, m\}$. Let $\Lambda_\tau$ be the cycle passing
 the elements of $A_\tau=\{a_{\tau(1)}, \dots, a_{\tau(m)}\}$ one by one. As before, we define
\begin{equation}\label{eq:self2}
G(\cS, A_\tau, \beta)=\bigcup_{j\in \Sigma} S_j(\Lambda_\tau^{\beta_i}).
\end{equation}

 For simplicity,
we abbreviate $G(\mathcal{S}, A_{\tau},\beta)$ as $G(\cS, A_\tau)$  when  $\beta=(1,1,\dots,1)$ is  totally positive.
The main idea in this section
 is to  generate an Euler tour of $G(\cS, A_\tau)$ by a bubbling process.

\subsection{Bubbling property} First, we give some definitions.
 We say two subgraphs $H_1$ and $H_2$ are  disjoint, denoted by $H_1\cap H_2=\emptyset$, if they do not have a common edge.

Let $P$ be an Euler tour of a directed graph $G$, and let $C=e_1+\cdots+e_h\subset G$  be a cycle.
We call $C$ a \emph{remaining cycle} of $P$,  if  $e_j$ appears earlier than $e_{j+1}$ in $P$ for all $1\leq j<h$ when we regard $e_1$ as the starting edge of $P$.
 In this case,   $P$ can be uniquely written as
 \begin{equation}\label{outpath decompose}
P= e_1+L_1+\cdots+e_h+L_h,
 \end{equation}
 where  some $L_j$ may be  empty trails. We  call $L_j$ out-trails  w.r.t.  $C$, and
   call  \eqref{outpath decompose} the \emph{out-trail decomposition} of $P$ w.r.t. $C$.

\begin{figure}
  \centering
  \includegraphics[width=0.4\textwidth]{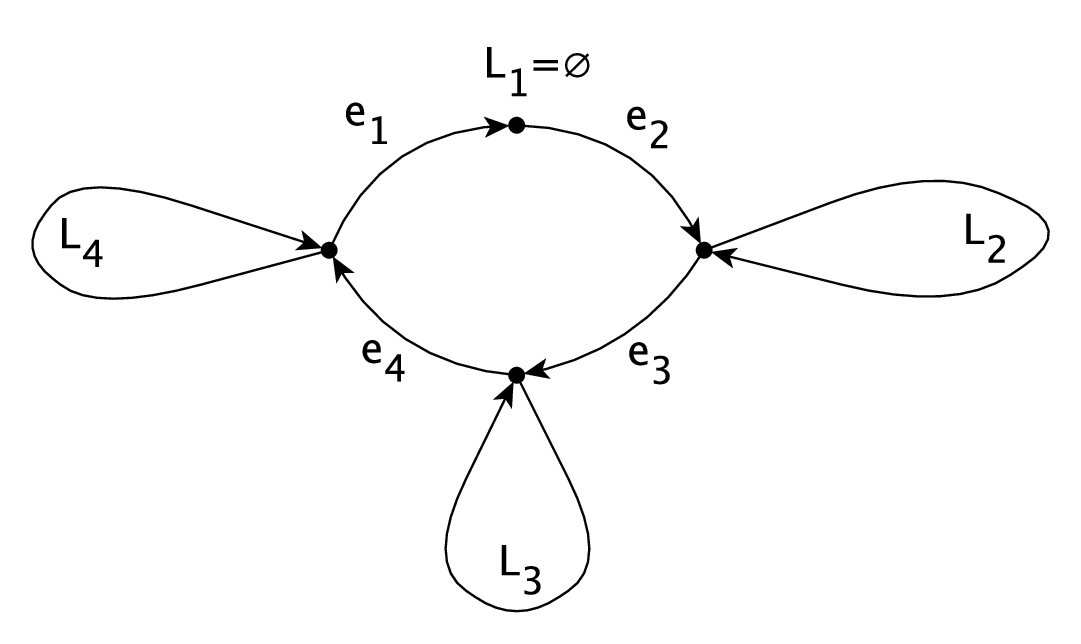}\\
 \caption{A remaining cycle $C=e_1+e_2+e_3+e_4$ and its out-trials.}
 %\label{}
\end{figure}

\begin{defi}{\rm
 \emph{(Cell-exact subgraph.)}  A subgraph $H$ of $G(\cS,A_\tau)$ is said to be \emph{cell-exact}, if
 for all cell $S_k(\Lambda_\tau)$, either $S_k(\Lambda_\tau)\subset H$ or $S_k(\Lambda_\tau)\cap H=\emptyset$.
}
\end{defi}

\begin{defi}\label{bu-pro}{\rm
\emph{(Bubbling property.)} Let $H\subset G(\mathcal{S}, A_\tau, \beta)$
be a cell-exact subgraph, and  let $P$ be an Euler tour of
$H$.
We say $P$    satisfies  the \emph{bubbling} property, if  for all $S_i(\Lambda)\subset H$,
    the cell $S_i(\Lambda)$ is a remaining cycle of $P$, and all the outpaths w.r.t. $S_i(\Lambda)$
     are cell-exact.
   }
\end{defi}

The goal of this section is to prove the following theorem.

\begin{thm}\label{thm:consistency}
There exist an integer  $n$ and a permutation $\tau$ such that $G(\mathcal{S}^n, A_\tau)$ has an Euler-partition
$P=P_1+\cdots +P_m$  which is consistent and satisfies the bubbling property.
\end{thm}

\begin{defi}{\rm Let $P$ be an Euler tour of $G(\cS, A_\tau,\beta)$. We say
a cell $S_i(\Lambda_\tau)$ is \emph{unbroken}, if $S_i(\Lambda_\tau)=e_1+\cdots +e_m$
is a subtrail of $P$. (Here $e_1$ can be any edge in $S_i(\Lambda_\tau)$.)
}
\end{defi}

\subsection{Bubbling process}

Recall that  $H(\mathcal{S},A)$ is the Hata graph induced by $A$, see Definition \ref{def-skeleton}. It  is a connected undirected graph.

 A subgraph $R$ of $H(\mathcal{S},A)$ is called a \emph{spanning  tree}, if $R$ is a tree containing A.
(Recall that an undirected connected graph is called a \emph{tree}, if it does not contain cycles.)
A vertex $v$ of a tree $R$ is called a \emph{top}, if $v$ is a vertex of degree one.
\begin{figure}
\includegraphics[width=0.3\textwidth]{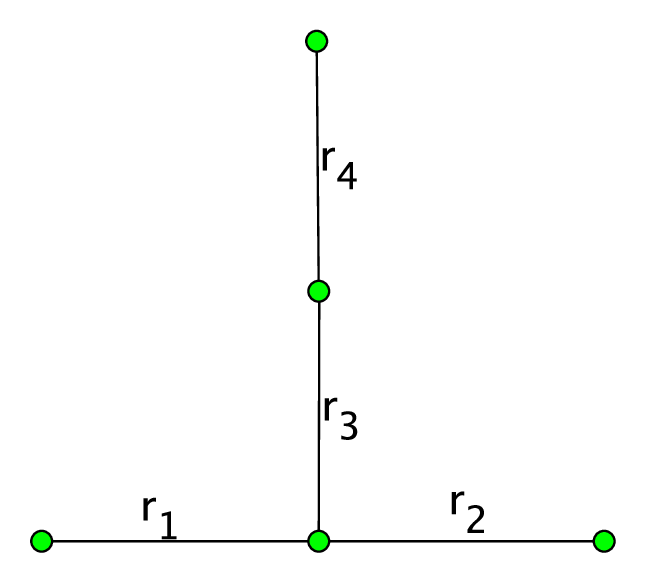}\quad \ \
\includegraphics[width=0.33\textwidth]{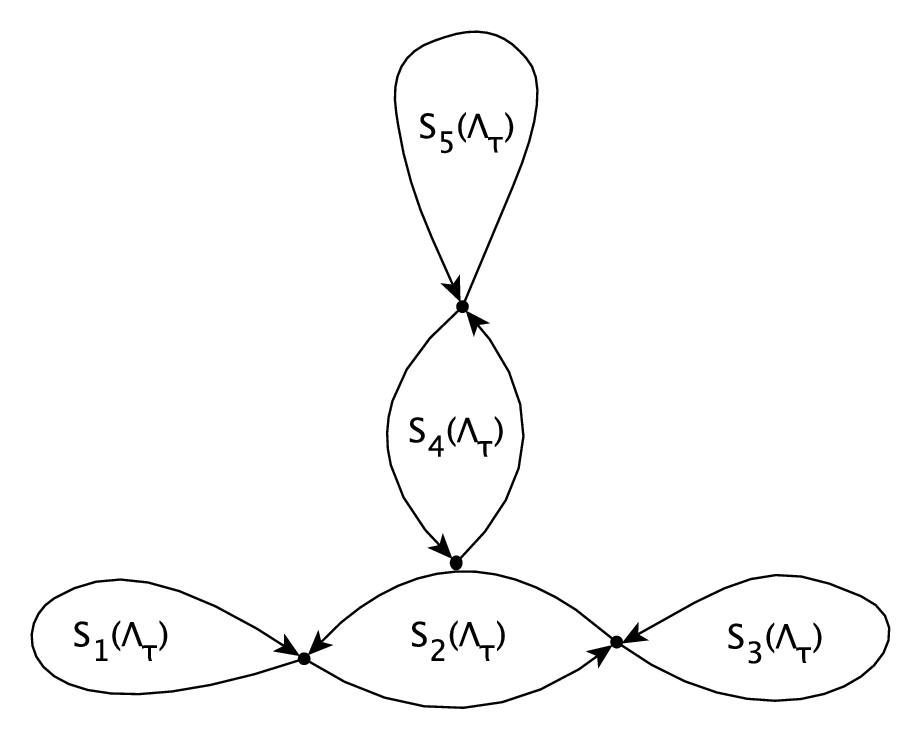}
\caption{A spanning tree and bubbling process.}
\label{fig:bub}
\end{figure}

From now on, we choose a spanning tree $R$ of $H(\cS, A)$ and fix it.
We choose an order of the edges of $R$, say,
\begin{equation}\label{bublling-order}
r_1, \dots, r_{N-1},
\end{equation}
 such that for any $1\leq j\leq N-1$,
the subgraph consisting of the edges  $r_1,\dots, r_j$ is connected.
This means that starting from a certain cell, we can add the cells one by one
following above order to obtain $G(S,A_\tau)$,
and all the subgraphs in the process are  connected.
 By changing the name of $S_j$, we may assume without loss of generality that the cells are added in the order $S_1(\Lambda_\tau),\dots, S_N(\Lambda_\tau)$.
  Then  one vertex of $r_j$ is $S_{j+1}$, and the other vertex of $r_j$ belongs  to $\{S_1,S_2,\dots, S_j\}$, which we denote  by $S_{\varphi(j)}$.
We note that if $S_k$ is a top of $R$, then $\varphi(j)\neq k$ for any $j$. See Figure \ref{fig:bub}.

Now, we  construct inductively a sequence   $(\Delta_k)_{k=1}^N$  such that
$\Delta_k$ is an Euler tour of
$\bigcup_{j=1}^k S_j(\Lambda_\tau)$
 satisfying the bubbling property in Definition \ref{bu-pro}.

Let $\Delta_1=S_{1}(\Lambda_\tau)$ be the first Euler tour, which clearly satisfies the bubbling property.

Suppose $\Delta_{k}$ has been constructed, $1\leq k\leq  N-1$. Now we add the cell $S_{k+1}(\Lambda_\tau)$ to $\Delta_k$.
In the spanning tree $R$, $S_{k+1}$ is connected to
 $S_{\varphi(k)}\in \{S_{1},\dots, S_{{k}}\}$, so
 $$
  S_{k+1}(A)\cap S_{\varphi(k)}(A)\neq \emptyset.
 $$
 Take any point $z^*$ from the above intersection.
 Write
  $$S_{\varphi(k)}(\Lambda_\tau)=e_1+\cdots+e_m \text{ and } S_{k+1}(\Lambda_\tau)=e_1'+\cdots+e_m', $$
  where both $e_1$ and $e_1'$ have origin $z^*$.
The outpath decomposition of $\Delta_k$ w.r.t. $S_{\varphi(k)}(\Lambda_\tau)$ can be written as
 $
 \Delta_k=e_1+L_1+\cdots+e_m+L_m.
 $
 Set
 $$
 \Delta_{k+1}=e_1+L_1+\cdots+e_m+L_m+(e_1'+\cdots +e_m').
 $$
 Clearly,  $\Delta_{k+1}$ is an Euler tour of the graph
$
 \bigcup_{j=1}^{k+1}S_{j}(\Lambda_\tau).
$

Now we show that $\Delta_{k+1}$ satisfies the bubbling property.
For a cell  $S_j(\Lambda_\tau)\subset \Delta_{k+1}$, no matter
it is  $S_{k+1}(\Lambda_\tau)$, or is
$S_{\varphi(k)}(\Lambda_\tau)$, or belongs to a certain $L_i (1\leq i \leq k)$, one can easily prove that
it is a remaining  cycle of $\Delta_{k+1}$ and all its out-trials are cell-exact. The bubbling property is proved.

Set $k=N$, we obtain  the following:

\begin{lemma} For any permutation $\tau$ of $A$, there
exists an  Euler tour $\Delta$ of
$G(\cS, A_\tau)$  satisfying the bubbling property.  Moreover, if  $S_i$ is a top of the
  spanning tree $R$, then $S_i(\Lambda_\tau)$ is an unbroken cell of $\Delta$.
\end{lemma}

\subsection{A natural Euler-parition of $\Delta$. }
Now we introduce a natural partition of $\Delta$.
We start from the vertex $a_1$, and walk along $\Delta$.
We record the elements of $A$ appearing on the Euler tour  but without repetition, then we get a rearrangement of $A$,
which we denote by
$$a_{\tau^*(1)}, \cdots, a_{\tau^*(m)}.$$
%Then   $\Delta$ visits $a_{\tau^*(j)}$ before $a_{\tau^*(j+1)}$.
Let $P_j$ be the subtrail of $\Delta$
starting from the first visiting of $a_{\tau^*(j)}$ and ending at the first visiting of $a_{\tau^*(j+1)}$, then
$$P_1+\cdots+P_m$$
is an Euler-partition of $G(\cS, A_\tau)$ with output permutation $\tau^*$,
which we denoted by ${\cal P}_\tau$.
 %We also call $\tau^*$ a follower of $\tau$ with respect to $P=P_1+\dots+p_m$.

\subsection{A cycle of permutations of $A$}
Starting from a permutation $\tau$ of $A$, we can construct an Euler-partition ${\cal P}_\tau$
of $G(S, A_\tau)$, and we denote its output permutation by $\tau^*$.
Since there are only finite many permutations of $A$,  there  exist $\tau_1,\dots, \tau_h$
 such that $\tau_{j+1}$ is the output permutation of ${\cal P}_{\tau_j}$, $1\leq j\leq h-1$,
  and  $\tau_{1}$ is the output permutation of ${\cal P}_{\tau_h}$.

 \subsection{Composition of Euler-partition}
Let $P=P_1+\cdots+P_m$ be an Euler-partition of $G(\cS, A_{\tau_1})$  with
 output permutation $\tau_2$,  and
 let $Q=Q_1+\cdots+Q_m$ be an  Euler-partition of $G(\cS^n,A_{\tau_2})$ with output permutation  $\tau_3$.
 Since $Q$ is an Euler tour of
 $$\bigcup_{I\in \Sigma^n} S_I(\Lambda_{\tau_2}),$$
 each edge in $Q$ can be written as
 \begin{equation}\label{eq:edge}
 S_I(\overrightarrow{a_{\tau_2(k)}a_{\tau_2(k+1)}}).
 \end{equation}
  Notice that
 $S_I(P_k)$ is a trail  having the same origin and terminus as the edge in \eqref{eq:edge}.
 So, we can define  an Euler tour of $G(S^{n+1},A_{\tau_1})$
by replacing  every edge $S_I(\overrightarrow{a_{\tau_2(k)}a_{\tau_2(k+1)}})$  in $Q$ by the trail $S_I(P_k)$.

 We denote this replacement rule by $\sigma$, that is,
 \begin{equation}\label{sigma}
\sigma (S_I(\overrightarrow{a_{\tau_2(k)}a_{\tau_2(k+1)}}))= S_I(P_k),\quad I\in \{1,\dots, N\}^n, \  k\in \{1,\dots, m\},
 \end{equation}
 and set $\sigma(\bomega
 )=\sigma(\omega_1)+\dots+\sigma(\omega_{\ell})$ for any trail $\bomega=\omega_1+\dots+\omega_{\ell}$.
We define the composition of $P$ and $Q$ to be
 $$Q\circ P=\sigma(Q_1)+\cdots \sigma(Q_m),$$
which is an Euler-partition   of $G(\mathcal{S}^{n+1}, A_{\tau_1})$.

\begin{lemma}\label{lem:follower}
Let $P=P_1+\cdots+P_m$ be an Euler-partition of $G(S, A_{\tau_1})$  with
 output permutation $\tau_2$,  and
 let $Q=Q_1+\cdots+Q_m$ be an Euler-partition of $G(S^n,A_{\tau_2})$ with output permutation  $\tau_3$. Then

 (i) If both $P$ and $Q$ have the bubbling property, then $Q\circ P$ also has the bubbling property.

(ii) The output permutation of $Q\circ P=\sigma(Q_1)+\cdots \sigma(Q_m)$ is $\tau_3$.

%(iii) If $Q$ contains pure cells, then $Q\circ P$ also contains pure cells.

(iii) If the numbers of unbroken cells in $P$ and $Q$ are $u$ and $v$, respectively, then the number of unbroken cells in $Q\circ P$
is no less than $(u-1)v$.
\end{lemma}

\begin{proof} (i) Let $Ij\in \Sigma^{n+1}$. Since $Q$ has the bubbling property, the outpath decomposition w.r.t. $S_I(\Lambda_{\tau_2})$
can be written as
$$
Q=e_1+L_1+\cdots+ e_m+L_m,
$$
where $e_j=S_I(\overrightarrow{a_{\tau_2(j)}a_{\tau_2(j+1)}})$, and the outpaths $L_j$  are cell-exact.
Then   $Q\circ P$ can be written as
 $$
 Q\circ P=\sigma(e_1)+\sigma(L_1)+\cdots+\sigma(e_m)+\sigma(L_m),
 $$
 where  $\sigma$ is the replacement rule defined in \eqref{sigma}.
  Clearly $\sigma(L_j)$ are cell-exact  in $G(\cS^{n+1}, A_{\tau_1})$ since $L_j$ are cell-exact in $G(\cS^n, A_{\tau_2})$.
 Since all $\sigma(L_j)$ does not contain edges in $S_{Ij}(\Lambda_{\tau_1})$, to prove $Q\circ P$
is bubbling  w.r.t. to  $S_{Ij}(\Lambda_{\tau_1})$,
 we need only show that
 $$
 \sigma(e_1)+\cdots +\sigma(e_m)=S_I(P)
 $$
 is bubbling w.r.t. $S_{Ij}(\Lambda_{\tau_1})$.
 This is clearly true since $P$ is bubbling w.r.t. $S_{j}(\Lambda_{\tau_1})$
 by our assumption, and the bubbling property does not change under a linear transformation.
 This finishes the proof of (i).

 (ii)  is obvious, since $\sigma(Q_k)$ has the same origin and terminus with $Q_k$.

 (iii) Let $S_I(\Lambda_{\tau_2})=e_1+\cdots+e_m$  be an unbroken cell of $Q_k$. Then
 $$\sigma(e_1+\cdots+e_m)=S_I(P_1)+\cdots+S_I(P_m)$$
  is the affine image of $P$ under $S_I$, and it is a subtrail of $Q\circ P$ by the unbroken property.
  This subtrail contains at least $u-1$ unbroken cells of $Q\circ P$,
  since under $S_I$, all unbroken cells in $P$ map to unbroken cells of $Q\circ P$, except at most one of them may be no longer unbroken after inserting the complement of $S_I(P)$. So $Q\circ P$ contains  at least $(u-1)v$ unbroken  cells.
 \end{proof}

 %\subsection{Consistency}

 %Now we prove the existence of Eulerian partition which is consistent.

%\begin{thm}\label{thm:consistency}
%There exist an integer  $n$ and a permutation $\tau$ such that $G(\mathcal{S}^n, A_\tau)$ has an Eulerian partition
%$P=P_1+\cdots +P_m$  which is consistent and satisfies the bubbling property.
%\end{thm}

\medskip
\noindent \textbf{Proof of Theorem \ref{thm:consistency}.} We have shown that there exists a sequence of permutations $\tau_1, \dots, \tau_h$ and a sequence of Euler-partitions $Q^1,\dots, Q^h$  such that for $k=1,\dots, h$,

$(i)$ $Q^k$
  is  an Euler-partition of $G(\cS, A_{\tau_k})$ with output permutation $\tau_{k+1}$,
  where we identify $\tau_1$ and $\tau_{h+1}$;

  $(ii)$ $Q^k$ has the bubbling property.

\noindent By  Lemma \ref{lem:follower},
\begin{equation}\label{eq:PQ}
P=Q^{h}\circ \cdots \circ Q^1
\end{equation}
is an Euler-partition of $G({\cal S}^h, A_{\tau_1})$ with output permutation $\tau_1$
and satisfying the bubbling property.
This $P$ is the desired Euler-partition.
$\Box$

\medskip

\begin{remark}\label{rem:pure}
{\rm
If $H(\cS, A)$
has a spanning tree with $3$ or more tops, then we can require  the Euler-partition $P$ in Theorem \ref{thm:consistency}  containing at least $3\cdot 2^{h-1}$ unbroken cells.
}
\end{remark}

\begin{remark}\label{rem:ite}
{\rm
Let $P$ be a consistent Euler-partition of $G(\cS, A_\tau, \beta)$,
then we can define the iteration of $P$ as $P^\ell=P\circ P\circ \cdots \circ P$ by induction, which is a consistent Euler-partition
of  $G(\cS^\ell, A_\tau, \beta')$ for some $\beta'\in \{-1,1\}^{N^\ell}$.
}
\end{remark}

%%%%%%%%%%%%%%%%%%%%%%%
\section{\textbf{Primitivity}}\label{sec:primitive}

  Let $Q=Q_1+\cdots+Q_m$
be an Euler-partition  of $G({\cal S}, A_\tau)$ which is consistent and satisfies the bubbling property.
In this section, we show that we can always  change the orientation of some cells in $G({\cal S}, A_\tau)$ and transfer $Q$ to a  primitive Euler-partition.

We shall use $A$ and $\Lambda$ instead of $A_\tau$ and $\Lambda_\tau$, since the permutation $\tau$ is fixed in this section.
Our goal  is to prove the following:

\begin{theorem}\label{thm:operation} Let $Q=Q_1+\cdots+Q_m$
be a consistent  Euler-partition of  $G(\cS, A)$ satisfying the bublling property. If
\begin{equation}\label{eq:length2}
\min_{1\leq j\leq m} |Q_j|\geq m^4+5m^2,
\end{equation}
then there exist an orientation vector  $\beta$ such that we can construct a consistent and primitive Euler-partition
$P=P_1+\cdots+P_m$ of $G({\cal S}, A,  \beta)$.  Moreover,  $P$   contains pure cells whenever $Q$ does.
\end{theorem}

\subsection{Classification of cells}
  We call the initial edges and terminate edges of trails $Q_1,\dots, Q_m$  \emph{special edges};  there are $2m$ special edges. We say a trail \emph{visits a cell}, if they have  common edges.

 Let $i\in \Sigma=\{1,\dots, N\}$. We call $S_i(\Lambda)$  a \emph{special cell}, if it contains special edges.
We call $S_i(\Lambda)$ a \emph{pure} cell,   a  \emph{bi-partition cell}, or a
\emph{poly-partition cell} if it is   visited by one, two or more than two members of
 $\{Q_k;~1\leq k\leq m\}$.

%\begin{figure}[h]
%  \centering
%  % Requires \usepackage{graphicx}
%  \includegraphics[width=6cm]{XmasTree3.eps}\quad\quad\includegraphics[width=6 cm]{XmasTree4.eps}
%  %\caption{Recall that we use $A=(a_1,a_2,a_3,a_4)$ as a skeleton of the Christmas tree and the orientation vector $\beta$ is positive. $T_1$ is the first iteration of the initial graph and $T_2$ is the second iteration. In $T_1$ and $T_2$, the red path describe $P_1$, the green path is $P_2$, the mauve path is $P_3$ and the blue line is $P_4$. It is easy to know that a cell passed by only one color is a pure cell, and except for special cells, a cell passed by two colors is a bi-partition cell and it is a poly-partition cell if it is passed by more than two colors. Hence, there is only one bi-partition cell and others are special cell in $T_1$, and in $T_2$ there are special cells, pure cells and bi-partition cells.} \label{cell}
%\end{figure}

 Take a cell $S_i(\Lambda)$,  let
  \begin{equation}\label{eq:out}
   Q=e_1+L_1+e_2+L_2++\cdots+ e_m+L_m
   \end{equation}
   be the outpath decomposition of $Q$ w.r.t.  $S_i(\Lambda)$.

%$$

\begin{lemma} \label{numb}
Let $S_i(\Lambda)$  be a non-special cell. Then
\begin{equation}\label{equ-3}
\#\{Q_j;\  Q_j \text { visits } S_i(\Lambda)\}=\#\{L_j; ~ L_j \text{ contains special edges}\}.
\end{equation}
\end{lemma}

\begin{proof} If $S_i(\Lambda)$ is pure, that is, it is a subset of one $Q_j$, then both sides of \eqref{equ-3} equal $1$.

Now suppose $S_i(\Lambda)$ is not pure.
If $Q_j$ visits $S_i(\Lambda)$, then $Q_j$ visits exactly two $L_j$ which contain special edges.
On the other hand, each outpath $L_j$ containing special edges intersects exactly two of
$Q_j$ visiting $S_i(\Lambda)$. The lemma is proved. (If one boy shakes hands exactly with two girls and one girl shakes hands exactly with two boys, then the number of boys and the number of girls are the same.)
\end{proof}

Now we estimate the number of poly-partition cells.
 \begin{lemma}\label{Estima:poly}
 The number of  poly-partition cells is less than $m(m-1)(m-2)/6$.
 \end{lemma}

 \begin{proof}
For a cell $S_i(\Lambda)$, we define $\kappa(i)=\{k; \ Q_k \text{ visits } S_i(\Lambda)\}.$
% to be the set of indices $k$ such that $Q_k$ visits the $S_i$-cell.

Let $S_i(\Lambda)$ and $S_j(\Lambda)$  be two poly-partition cells, that is, $\#\kappa(i), \#\kappa(j)\geq 3$.

Let $L$ be the outpath  of $S_i(\Lambda)$ containing the cell $S_j(\Lambda)$. Then there exist two edges $e_h$ and $e_{h+1}$ in the cell $S_i(\Lambda)$ such that
$$ e_h+L+e_{h+1}$$
is a subtrail of $Q$.
Hence, if $Q_k$ visits both the $S_i$-cell and the loop $L$, then
 it contains at least one of $e_h$ and $e_{h+1}$. This implies that at most two elements of $\{Q_k;\  k\in \kappa(i)\}$, visit $L$.
  Therefore $\kappa(i)$ and $\kappa(j)$ share  at most two elements.
   It follows that the smallest three elements of $\kappa(i)$
     and that of $\kappa(j)$ cannot be the same. The lemma is proved.
 \end{proof}

\subsection{Kingdom of a non-special bi-partition cell.}
\begin{figure}[htbp]
  \includegraphics[width=0.65\textwidth]{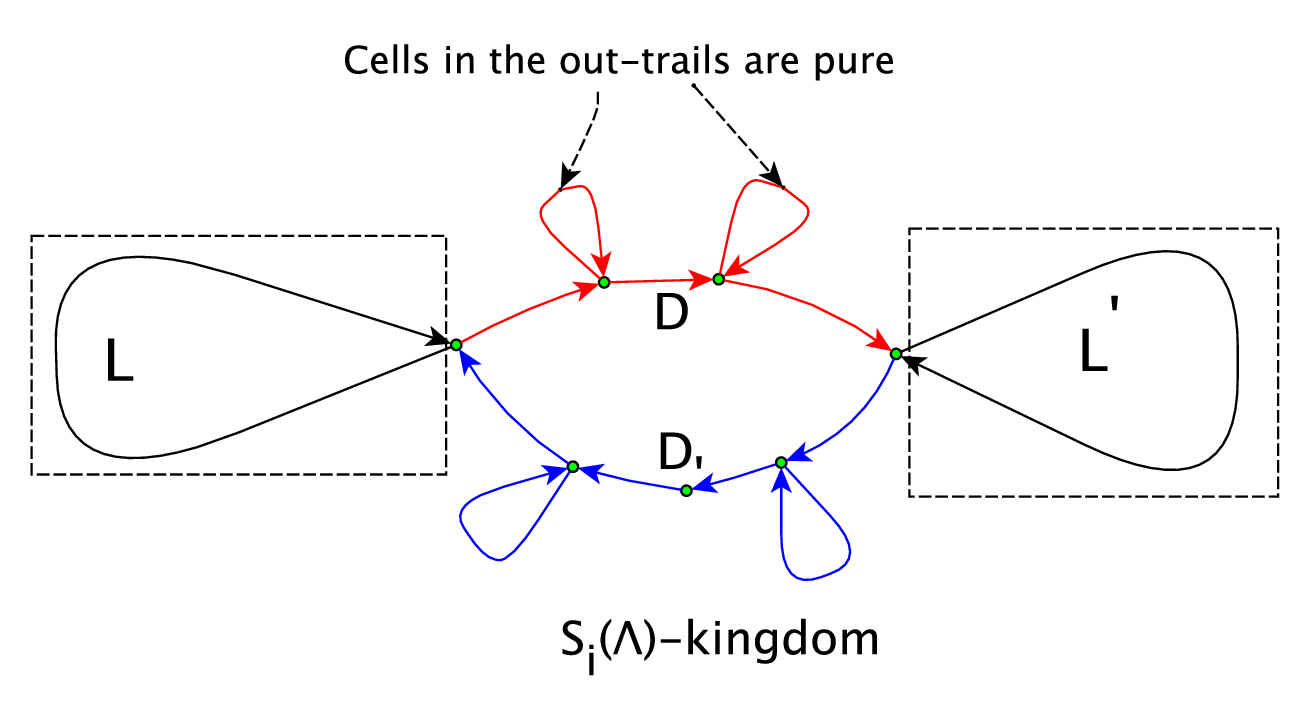}
 \caption{The re-decomposition of the Euler tour $Q$ }\label{S_i-kingdom}
 \end{figure}
Let $S_i(\Lambda)$ be a non-special  bi-partition cell,
then   $S_i(\Lambda)$   is visited by two elements in $\{Q_j; 1\leq j\leq m\}$, say, $Q_k$ and $Q_{k'}$.
By Lemma \ref{numb}, exactly two outpaths of $S_i(\Lambda)$, denoted by
 $L$ and $L'$,   contain  special edges.
Then the Euler tour $Q$ has the following decomposition:
 $$
 Q=L+D+L'+D'.
 $$
 We call $D\cup D'$ the
\emph{$S_i(\Lambda)$-kingdom}, and  call $D+D'$ the \emph{partition} of the $S_i(\Lambda)$-kingdom.
Then $Q_k$, also $Q_k'$,  are cut into three parts by $L$ and $L'$,  and we denote the three parts by
\begin{equation}\label{eq:Q_k}
Q_k=C+D+E, \quad  Q_{k'}=C'+D'+E',
\end{equation}
 where $C, E'\subset L$ and  $E, C'\subset L'$. (See Figure \ref{kingdom-decomp}.)  As a consequence of \eqref{eq:Q_k}, we have

 \begin{figure}[htbp]
 \includegraphics[width=0.65\textwidth]{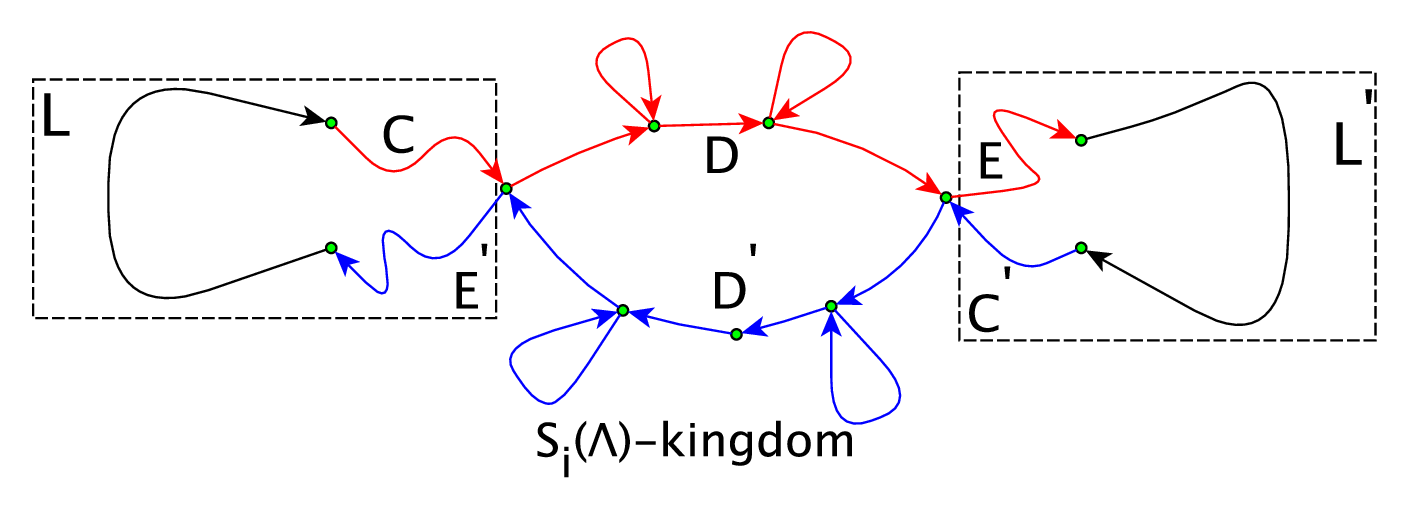}
 \caption{The decomposition of $Q_k$ and $Q_k'$ }\label{kingdom-decomp}
 \end{figure}

 \begin{lemma}\label{lem:pur}
 An $S_i(\Lambda)$-kingdom is a cell-exact subgraph; moreover, all cells in it are pure except $S_i(\Lambda)$.
 \end{lemma}

\begin{lemma} (Disjointness of Kingdoms.) Let $S_i(\Lambda)$ and $S_j(\Lambda)$ be two non-special bi-partition cells. Then the $S_i$-kingdom and the $S_j$-kingdom are disjoint.
\end{lemma}

\begin{proof} Since $S_i(\Lambda)$ is a bi-partition cell, $Q$ can be written as
$$Q=L+  D + L' + D'$$
where $D+ D'$ is the $S_i${-kingdom}.
The cell $S_j(\Lambda)$  does not belong to the $S_i$-kingdom since it is not pure (see Lemma \ref{lem:pur}),
 so $S_j(\Lambda)$ must belong to one of  $L$ and $L'$.
Let us assume that $S_j(\Lambda)\subset L'$ without loss of generality.
Then
$$Q\setminus L= L' + (S_i\text{-kingdom})$$
 is a closed subtrail of $Q$   not intersecting  $S_j(\Lambda)$.
So   $Q\setminus L$  must  belong to an outpath $L^*$ w.r.t. $S_j(\Lambda)$.
This outpath $L^*$ contains $L'$ and hence contains special edges, so it does not  belong to the $S_j$-kingdom.
It follows that the $S_i$-kingdom, as a subset of $L^*$, does not intersect the $S_j$-kingdom. The lemma is proved.
\end{proof}

 \subsection{Operation cells}
 First, let us reformulate the notion of primitivity. Let $Q$ be an Euler-partition.
 We construct a graph $T(Q)$ as following: the vertex set is $\{1,\dots, m\}$. If $Q_k$ contains a cell
 $S_j(\overrightarrow{a_{k'}a_{k'+1}})$ or $S_j(\overrightarrow{a_{k'}a_{k'+1}}^{-1})$ for some $j\in \Sigma$,
 then there is an edge from $k$ to $k'$. Then $Q$ is primitive if and only if there is an integer $n_0$, for any
 $i,j\in \{1,\dots, m\}$, there is a trail in $T(Q)$ of length $n_0$ from $i$ to $j$.

Set
$$V=\{k~; ~~Q_k \text{ contains pure cells}\}, \quad V'=\{1,\dots,m\}/V.$$

If $V'=\emptyset$, then $T(Q)$ is a complete graph and  clearly $Q$ is primitive.
 So we assume that $V'\neq \emptyset$. Let $k_0$ be the minimal element of $V'$.

For each $k\in V$, we select a pure cell contained in  $Q_k$, and call it  the \emph{protected cell}, or a \emph{black cell},  of $Q_k$.
It is possible that this black cell is contained in a $S_j$-kingdom.
If this happens, then we call $S_j(\Lambda)$  an \emph{indirectly protected cell}, or \emph{grey cell},
 associated with $Q_k$. (This means that
$S_j(\Lambda)$ is a non-special bi-partition cell).
By the disjointness of kingdoms,  there is at most one grey cell associated with  $Q_k$.

For any $k\in V'$, $Q_k$ only  visits  bi-partition cells and poly-partition cells. We  shall choose
operation cells among the bi-partition cells  which are visited by $Q_k$, non-special and are not grey cells;
the lemma below estimates the number of such cells.  Let $|Q_j|$ be the length of $Q_j$, i.e. the number of edges in $Q_j$.

\begin{lemma}\label{lem:length} Let $k\in V'$. If
\begin{equation}\label{eq:length}
\min_{1\leq j\leq m} |Q_j|\geq m^4+5m^2,
\end{equation}
then the number of bi-partition cells which are  visited by $Q_k$, non-special and not indirectly protected,  is at least
  $2m$.
\end{lemma}
\begin{proof}  The total number of  special cells, non-special  poly-partition cells, and the indirectly protected cells are no more than
$2m$, $ \frac{m(m-1)(m-2)}{6}$ and $m$, respectively (see Lemma \ref{Estima:poly}.). The number of cells visited by $Q_k$ is at least $|Q_k|/m$. So the number of
the desired cells is no less than
$|Q_k|/m-(2m+\frac{m(m-1)(m-2)}{6}+m)\geq 2m.$
\end{proof}

Now we assume that the condition of Lemma \ref{lem:length} is fulfilled.
Denote
$$r=m+(\# V')-1.$$
By the above lemma,
we can choose $r$ non-special, not indirectly protected, bi-partition cells
\begin{equation}\label{eq:op-cell}
T_1,\dots, T_m, T_{m+1}, \dots, T_r
\end{equation}
such that
$Q_{k_0}$ (where $k_0=\min V'$)  visits the first $m$ of them, and for the other $k\in V'$, each $Q_k$ visits
a different cell in $T_{m+1}, \dots, T_r$.
We call them \emph{operation cells}. In other words, we associate with $k_0$ the first $m$ operation cells, and for each $k\in V'\setminus\{k_0\}$, we associate with one operation cell.
%and denote  the set of operation cells by $T$.

 \subsection{$(v,k)$-operation}
Let $v$ be an edge of $\Lambda$ and let $k\in V'$.
Let $S_{\ell}(\Lambda)$ be an operation cell associated with  $k\in V'$.
 We do a $(v,k)$-operation to $S_{\ell}$ as follows.

Let  $Q_{k^*}$ be the other trail in $\{Q_j;~1\leq j\leq m\}$ visiting the cell $S_{\ell}(\Lambda)$.
Let $D+D^*$ be the partition of the $S_\ell$-kingdom.
By  \eqref{eq:Q_k}, $Q_k$ and $Q_{k^*}$ can be written as
$$Q_k=C+D+E, \quad Q_{k^*}=C^*+D^*+E^*$$

If $S_{\ell}(v)$  belongs to $Q_k$, the $v$-operation is a null operation,
that is, we do not change $Q$. Otherwise,
 the $v$-operation is to   construct a new Euler-partition $Q'=Q_1'+\cdots +Q_m'$  by exchanging $D$ and $D^*$, precisely,
we set
$$
Q_k'=C+(D^*)^{-1}+E,\quad
Q_{k*}'=C^*+D^{-1}+E^*,
$$
and for  $j\not\in \{ k, k^*\}$, set  $Q_j'=Q_j$.
 %Denote $Q'=Q_1'+\cdots+Q_m'$.
 Then the following statements (i)-(iv) hold:

 \medskip

$(i)$ \textit{$Q'$ is an Euler tour in $G(S, A, \beta')$ where $\beta'$ is obtained by changing
 the orientation of all the cells in the $S_{\ell}$-kingdom.}
 (This is clearly from the construction.)

 \medskip

 $(ii)$ \textit{$Q'$ contains pure cells if $Q$ does;}
  (This is true, since if $Q_j$ contains a pure cell,   the protected cell of $Q_j$ still belongs to $Q_j$ since the operation does not change
 any protected cell.)

 \medskip

 $(iii)$  \textit{$Q'$ is still consistent since $Q_k'$ and $Q_k$ have the same initial and terminate points, $1\leq k\leq m$.}

 \medskip

$(iv)$ \textit{Either $S_\ell(v)$ or  $S_\ell(v^{-1})$ belongs to  $Q_k'$.}
(If   $S_\ell(v)$ does not belong to  $Q_k$, then it must appear in $D^*$. After operation, $S_\ell(v^{-1})$ belongs to $Q_k$.)

Now we can proof the main theorem of this section.
\medskip

\noindent \textbf{Proof of Theorem \ref{thm:operation}.} Write $\Lambda=v_1+\cdots +v_m$.
For $j=1,\dots, m$,
 we do $(v_j,k_0)$-operation to the $j$-th cell in \eqref{eq:op-cell}, which are operation cells associated with the minimal element $k_0$ in $V'$.
For the operation cells associated to the  indices $k\in V'\setminus\{k_0\}$, we do the $(v_{k_0},k)$-operation.
We can do these $r$ operations consecutively, since the operation cells belonging to the kingdoms are disjoint.

Let $Q'=Q'_1+\cdots Q'_m$ be the resulted Euler-partition after doing all the operations.

First, $Q'$ is an Euler-partition of $G({\cal S}, A,  \beta')$,
where $\beta'$ is obtained by changing the orientations of the cells in the kingdoms which are involved in the operations.

Secondly, $Q'$ is consistent, since all the operations do not change the origins and terminates of the partition.

 Thirdly, $Q'$ is primitive, since  $v_{k_0}$ appears in every $Q_k$ for $k\in \{1,\dots, m\}$, and all $v_j$, $j=1,\dots,m$, appear in $Q_{k_0}$.

Finally, the pure cell property is clearly preserved. The theorem is proved.
$\Box$

\medskip

\section{\textbf{Proof of Theorem \ref{thm:main}}}\label{sec:pure}
In this section, we prove Theorem \ref{thm:main}. We consider two cases according to the behaviour of the Hata graphs $H(\cS^n, A)$, $n\geq 1$: if $H(\cS^n, A)$ is not a chain for some $n\geq 1$, then we show that $K$ admits SFCs constructed by the Euler tour method;
otherwise, $\cS$ is a self-similar zipper.

A graph with $n$ vertices is a chain if there exists an order $x_1, \dots,x_n$ of the vertices such that there is exactly one edge between $x_{i}$ and $x_{j}$ if $|i-j|=1$, otherwise there is no edge.

\begin{thm}\label{thm:case2} Let $\cS=\{S_i\}_{i\in \Sigma}$ be an IFS satisfying the OSC. Suppose that $\cS$ has skeleton $A$ and the Hata graph $H(\cS^k, A)~(k\geq 1)$ is not always a chain.
Then there exist an integer $n$, a permutation $\tau$ of $\{1,\dots,m\}$,
 a vector $\beta\in \{-1,1\}^{N^n}$ and an Euler-partition
$$P=P_1+\cdots +P_m$$ of
$G(\cS^{n}, A_\tau, \beta)$
such that $P$ is consistent, primitive and potentially contains pure cells.
\end{thm}

\begin{proof} By the assumption, there
exists an integer $n_1$ such that the Hata graph $H(\cS^{n_1}, A)$ is not a chain.
By Theorem \ref{thm:consistency}, there exist an integer $n_2$ and a permutation $\tau$ on $\{1,\dots, m\}$, and  an Euler-partition
$$Q=Q_1+\cdots+Q_m$$
of $G(\cS^{n_1n_2}, A_\tau)$  such that  $Q$ is consistent. Moreover, $Q$ 
   contains more than $3$ unbroken cells by Remark \ref{rem:pure}.

Notice that  $Q^\ell$, the $\ell$-th iteration of $Q$,
 is still consistent and contains more than $3$ unbroken cells.
If we choose $\ell$ large enough, then  the length of each $(Q^\ell)_j$ can be large enough to satisfy
inequality \eqref{eq:length2}, so $Q^\ell$ satisfies all  conditions of Theorem \ref{thm:operation}.
Therefore, there exist an orientation vector $\beta$, and an Euler-partition
$P=P_1+\cdots+P_m$
of
$G(\cS^{n_1n_2\ell}, A_\tau, \beta)$ such that $P$ is consistent and primitive.

Finally, we show that $P$ contains pure cells.
Since   $2^\ell>2m$ when $\ell$ is large enough, by Lemma \ref{lem:follower}(iii),
$Q^\ell$ contains more than $2m$ unbroken cells.
Since there are only $2m$ special edges in the Euler-partition $Q^\ell$,
there exists an  unbroken cell which  contains no special edge.
This unbroken cell, as a subtrail of $Q^\ell$,  can not cross the border of any elements in the partition $Q^\ell$, so it
  must  be a pure cell of $Q^\ell$.
  Therefore, by Theorem \ref{thm:operation}, $P$ also contains pure cells.
\end{proof}

\medskip
Recall that $\Sigma=\{1, 2, \dots, N\}$.

\begin{theorem}\label{thm:zipp}
Let $\cS=\{S_i\}_{i\in \Sigma}$ be an IFS with connected invariant set $K$ satisfying the OSC.  Suppose that $A$ is a skeleton of $K$ such that the Hata graph $H(\cS^n, A)$ is a chain for every $n\geq 1$, then an arrangement of $\{S_i\}_{i\in \Sigma}$ is  a self-similar zipper.
\end{theorem}

\begin{proof} We will abbreviate $H(\cS^n, A)$ as $H_n$.
Since $H_1$ is a chain, without lose of generality, we assume that $H_1$ passes the vertices in the order $S_1,S_2,\dots, S_N$.
To simplify the notation,  we denote $H_1$ by
$$H_1=(1, 2,\dots, N)$$
instead of  $(S_1, S_2, \dots, S_N)$. Moreover,
we set $1$ to be  the origin of the chain and $N$  the terminus.
For $1\leq j\leq N$, by self-similarity,
\begin{equation}\label{subchain}
({j1}, {j2}, \dots , {jN}) \text{ or } ({jN},\dots , {j2}, {j1})
\end{equation}
is a sub-chain of $H_2$,
the starting vertex of $H(\cS^2, A)$ is $S_{1k}$ for some $k\in \{1, N\}$,
and the ending vertex is $S_{Nk'}$ for some  $k'\in \{1, N\}$.
Accordingly,  we define a sequence $(\beta_j)_{j=1}^N$ as following:  $\beta_j=1$
if the former sequence in \eqref{subchain} is a sub-chain of $H(\cS^2, A)$, and $\beta_j=-1$ otherwise.

 We use $j(a_1,\dots, a_k)$ to denote the chain  $(ja_1,\dots, ja_k)$, then  we have
$$H_2=(1H_1)^{\beta_1}+(2H_1)^{\beta_2}+\dots+(NH_1)^{\beta_N},$$
where we use $`+'$ to denote the concatenation of two sequences, and denote $(a_1,\dots, a_k)^{-1}=(a_k,\dots, a_1)$
to be the reverse of a chain.
%Also, if we denote $H_j=(a_1,a_2,\dots,a_{N^j})$ which is a permutation of $\Sigma^j$, then we define
%$$iH_j=(ia_1,ia_2,\dots, ia_{N^j}).$$

%After we define $\{v_j\}_{j=1}^N$, we obtain the following two conclusions.
 First,
we prove by induction that
\begin{equation}\label{eq:Hn}
H_n=(1H_{n-1})^{\beta_1}+\cdots+(NH_{n-1})^{\beta_N}.
\end{equation}
By self-similarity, for any $j\in \{1,\dots, N\}$, either
$jH_{n-1}$ or $(jH_{n-1})^{-1}$ is a sub-chain of $H_n$, so there exist
$\epsilon_1,\dots, \epsilon_N\in \{-1,1\}$ such that
$$
H_n=(1H_{n-1})^{\epsilon_1}+\cdots+(NH_{n-1})^{\epsilon_N}.
$$
Since $(jH_{n-1})^{-1}=j(H_{n-1})^{-1}$, using induction hypothesis, we have
$$
\begin{array}{rl}
H_n=&1(H_{n-1})^{\epsilon_1}+\cdots+N(H_{n-1})^{\epsilon_N}\\
=&\cdots +j\left (1H_{n-2}^{\beta_1}+\cdots +NH_{n-2}^{\beta_N}\right )^{\epsilon_j}+(j+1)\left (1H_{n-2}^{\beta_1}+\cdots +NH_{n-2}^{\beta_N}\right )^{\epsilon_{j+1}}+\cdots
\end{array}
$$
Notice that if $i_1\cdots i_n$ and $j_1\cdots j_n$
are adjacent in $H_n$, then $i_1\cdots i_k$ and $j_1\cdots j_k$
are adjacent in $H_k$ for any $k<n$.  The above decomposition of $H_n$ forces that
$$
j(1,\dots, N)^{\epsilon_j}+(j+1)(1,\dots, N)^{\epsilon_{j+1}}
$$
is a sub-chain of $H_2$, so $\epsilon_j=\beta_j$ and $\epsilon_{j+1}=\beta_{j+1}$, which proves \eqref{eq:Hn}.

Using $(\beta_1,\dots, \beta_N)$, we define the following  ordered GIFS:
\begin{equation}\label{eq:two-state}
\begin{split}
E_1=&S_1(E_{\beta_1})+S_2(E_{\beta_2})+\dots+S_N(E_{\beta_N})\\
E_{-1}=&S_N(E_{-\beta_N})+\dots+S_2(E_{-\beta_2})+S_1(E_{-\beta_1}).
\end{split}
\end{equation}
Clearly,    $K=E_1=E_{-1}$ by the uniqueness of invariant set of $\cS$.

 Precisely, the vertex set of the GIFS
 is $\{E_1, E_{-1}\}$.  Let $\Gamma$ be the basic graph of the GIFS, then
 the $j$-th edge from $E_1$ is an edge from $E_1$ to $E_{\beta_j}$ with similitude $S_j$, which we denote by $(E_1, S_j, j, E_{\beta_j})$; the $j$-th edge starting from $E_{-1}$ goes to $E_{-\beta_{N-j+1}}$ and has similitude $S_{N-j+1}$, which we  denote by $(E_{-1}, S_{N-j+1}, j, E_{-\beta_{N-j+1}})$.

By Theorem \ref{lem:zipper}, to prove the theorem, we need only  prove  \eqref{eq:two-state} is a linear GIFS.

Let $G_n$ be the sequence of trails in the graph $\Gamma$ emanating from the state $E_1$, arranged in the increasing order.
Notice that for each trails in $G_n$, the associated contraction has the form $S_I$ with $|I|=n$, and for different trails, the associated contractions
 are distinct. So, we replace each trail in $G_n$ by the associated contraction and further more replace $S_I$ by the word $I$,
 then we obtain a sequence which is a permutation  of $\Sigma^n$. For simplicity, we still denote this sequence by $G_n$.
  Similarly, we denote by $G_n'$ the sequence of trails in the graph emanating from the state $E_{-1}$, arranged in the
  increasing order. We apply the same simplification to $G_n'$.

  We shall prove that $G_n=H_n$ and $G_n'=G_n^{-1}$.

 Clearly, $G_1=(1,2,\dots, N)$, $G_1'=(N,\dots, 1)$. Moreover, by \eqref{eq:two-state},
 $$
 G_2=(1G_1)^{\beta_1}+\cdots +(NG_1)^{\beta_N}.
 $$

We claim that, for every $n\geq 1$, it holds that
\begin{equation}\label{eq:Gn}
G_n=(1G_{n-1})^{\beta_1}+\cdots +(NG_{n-1})^{\beta_N}
\end{equation}
and $G_n'=G_n^{-1}$.

We prove by induction on $n$.
Let us consider the words in $G_n$ initialed by $j$.
If $\beta_j=1$, then after one step, we are still in the state $E_1$, hence
the arrangement of such words in increasing order is $jG_{n-1}$;
if $\beta_j=-1$, we switch to the state $E_{-1}$ after one step, and the corresponding arrangement
is $jG_{n-1}'=j(G_{n-1})^{-1}$.
Therefore, %By the construction of the ordered GIFS \eqref{eq:two-state}, we have
$$
G_n=(1G_{n-1})^{\beta_1}+(2G_{n-1})^{\beta_2}+\cdots+(NG_{n-1})^{\beta_N}.
$$
The same argument shows that
$$
\begin{array}{rl}
G_n'=&(NG_{n-1}')^{\beta_N}+\cdots+(2G'_{n-1})^{\beta_2}+(1G'_{n-1})^{\beta_1}\\
    =&(NG_{n-1})^{-\beta_N}+\cdots+(2G_{n-1})^{-\beta_2}+(1G_{n-1})^{-\beta_1}\\
    =&G_n^{-1}.
\end{array}
$$
The claim is proved.

Comparing \eqref{eq:Hn} and \eqref{eq:Gn}, we obtain that $G_n=H_n$ for all $n\geq 1$.
Hence, if $I$ and $I'$ are two adjacent words in $G_n$ (or in $G_n'$), then they are also adjacent in $H_n$, so
$$
S_I(K)\cap S_{I'}(K)\supset S_I(A)\cap S_{I'}(A)\neq \emptyset,
$$
which implies that \eqref{eq:two-state} is a linear GIFS.

\end{proof}

\noindent \textbf{Proof of Theorem \ref{thm:main}}.
Let $\{S_i\}_{i\in \Sigma}$ be an IFS satisfying the OSC. Let $A$ be a skeleton of $\cS$.

If the Hata graph $H(\cS^n, A)$ is  alway a chain, then $K$ can be generated by a self-similar zipper by Theorem \ref{thm:zipp},
and hence admits space-filling curves.

    If $H(\cS^n, A)$ is not alway a chain, by Theorem \ref{thm:case2}, there exists a integer $k$ such that an Euler-partition of $G(\cS^k, A, \beta)$ satisfying the conditions in Theorem \ref{thm:old}. Hence
    the invariant set of $\cS^k$, which is also $K$,     admits
    space-filling curves constructed by the  Euler-tour method.
 $\Box$
 % {sec:pure}
%\input{Hata_Chain_V1}
%\input{Reference}
%\input{BeiYong}

%
%\bibitem{CT} Cannon, James W.; Thurston, William P. (2007) [1982], "Group invariant Peano curves", Geometry & Topology 11: 1315¨C1355, doi:10.2140/gt.2007.11.1315, ISSN 1465-3060, MR 2326947
%
%\bibitem{Jones90} Rectifiable sets and the traveling salesman problem
%PW Jones - Inventiones Mathematicae, 1990 - Springer
%
%\bibitem{ds} ¡°Spacefilling curves and the planar travelling salesman problem¡± with L. K. Platzman, Journal of the Association for Computing Machinery 36(4):719-737 (1989).
%
%"Robust Multidimensional Searching with Spacefilling Curves" by J. Bartholdi and W. Nulty, Proceedings of the Sixth International Symposium on Spatial Data Handling, Edinburgh, Scotland, September 1994.
%
%Continuous indexing of hierarchical subdivisions of the globe by J. Bartholdi and P. Goldsman (2000). This appeared in slightly revised form in Int. J. Geographical Information Science 15(6):489-522 (2001).
%
%Vertex-labelling algorithms for the Hilbert spacefilling curve by J. Bartholdi and P. Goldsman (2000).
% This appeared in slightly revised form in Software -- Practice and Experience 31:395-408 (2000).
%
%The vertex-adjacency dual of a triangulated irregular network has a Hamiltonian cycle by J. Bartholdi and P. Goldsman, Operations Research Letters 32 (2004). Perouz Taslakian has produced a very nice illustration of the ideas and algorithm.
%
%http://www2.isye.gatech.edu/~jjb/mow/mow.html
%
%

%\input{appendix_A}

\end{document}